\documentclass[12pt]{article}
\usepackage[english]{babel}
\usepackage[square,numbers]{natbib}
\usepackage{amscd,float}
\usepackage{latexsym,amsmath,amssymb,amsbsy, amsthm}

\usepackage[colorlinks = TRUE, linkcolor = blue, citecolor = blue,
urlcolor = blue]{hyperref}
\usepackage{setspace}
\usepackage{lscape,fancyhdr,fancybox}
\usepackage{graphicx,epsfig, xy}
\usepackage{color}
\usepackage[hmarginratio=1:1, vmarginratio =5:5,
textheight=22cm,bindingoffset=1.5cm, textwidth=14.6cm]{geometry}
\usepackage{float}
\usepackage{graphicx}
\usepackage{caption}
\usepackage{subcaption}

\newcommand{\comment}[1]{}

\numberwithin{equation}{section}

\newtheorem{theorem}{Theorem}[section]
\newtheorem{lemma}{Lemma}[section]
\newtheorem{proposition}{Proposition}[section]

\newtheorem{algorithm}{Algorithm}[section]

\numberwithin{equation}{section}

\theoremstyle{definition}
\newtheorem{definition}{Definition}[section]

\usepackage{amsfonts}
\newtheorem{assumption}{Assumption}[section]

\usepackage{txfonts,enumerate}

\DeclareMathOperator{\E}{E} \DeclareMathOperator{\var}{Var}

\DeclareMathOperator{\Tr}{Tr}

\newcommand{\beq}{\begin{eqnarray}}
\newcommand{\eeq}{\end{eqnarray}}
\newcommand{\ben}{\begin{eqnarray*}}
\newcommand{\een}{\end{eqnarray*}}

\title{A new combinatorial approach for edge universality of Wigner matrices}
 \author{
\sc Debapratim Banerjee
 \\ \small Indian Statistical Institute, Kolkata\\ debopratimbanerjee@gmail.com\\}
 
\begin{document}
 \maketitle
 \begin{abstract}
 In this paper we introduce a new combinatorial approach to analyze the trace of large powers of Wigner matrices. Our approach is motivated from the  paper by \citet{sosh}. However the counting approach is different. We start with classical word sentence approach similar to \citet{AZ05} and take the motivation from \citet{sinaisosh}, \citet{sosh} and \citet{peche2009universality} to encode the words to objects similar to Dyck paths. To be precise the map takes a word to a Dyck path with some edges removed from it. Using this new counting we prove edge universality for large Wigner matrices with sub-Gaussian entries. One novelty of this approach is unlike \citet{sinaisosh}, \citet{sosh} and \citet{peche2009universality} we do not need to assume the entries of the matrices are symmetrically distributed around $0$. The main technical contribution of this paper is two folded. Firstly we produce an encoding of the ``contributing words" (for definition one might look at Section \ref{sec:word}) of the Wigner matrix which retrieves the edge universality. Hence this is the best one can do. We hope this method will be applicable to many other scenarios in random matrices. Secondly in course of the paper we give a combinatorial description of the GOE Tracy Widom law. The explanation for GUE is very similar. This explanation might be important for the models where exact calculations are not available but some combinatorial structures are present.
 \end{abstract}
 \section{Introduction}
 Since the groundbreaking discovery of \citet{wig1}, Wigner matrices have been a topic of key interest in the mathematics and physics communities. Later on these matrices proved to be important for many models in engineering, high dimensional statistics and many other branches. Since the introduction of these matrices many problems regarding the eigenvalue distribution of these matrices have been solved. The results are so vast and diverse that we shall not be able to discuss all of them in this introduction. Here we mention some of them which we find relevant in the context of current paper. 
 
 In particular, a Wigner matrix is a $n \times n$ symmetric (hermitian) matrix with real (complex) entries where the entries of the upper diagonal part are i.i.d. with mean $0$ and variance $\frac{1}{n}$. One is interested in the eigenvalue distribution of the matrix when the dimension grows to infinity. 
 The study of eigenvalues of Wigner matrices started with characterizing the limiting distribution of the histogram of the eigenvalues. This is done in the seminal papers of \citet{wig1} and \citet{wig2}. It is known that this limiting distribution exists and coined as the famous semicircular distribution. In particular, it is given by the following density.
\begin{equation}\label{eq:semicircular}
 f(x)=\left\{
 \begin{array}{ll}
 \frac{1}{2\pi} \sqrt{4-x^2} & \text{when~ $|x|\le 2$}\\
 0 & \text{otherwise} 
 \end{array}
 \right.
\end{equation}  
However after specifying the spectral distribution, there has been a remarkable advancement in this topic. One of the most important direction is the universality properties of the Wigner matrices. 

In general there are two types of universality properties of Wigner matrices. The first one is the bulk universality and the second one is the edge universality. In this paper we shall mostly consider the edge universality. However for the shake of completeness we briefly describe the important results on the bulk universality of the Wigner matrices. In general bulk universality concerns about the point process at any fixed point in the interior of the semicircle distribution. When the entries of the matrices are real (complex) Gaussians, they are called GOE (GUE). The explicit eigenvalue distributions in these cases are well known. For the explicit formula in the GOE case one might look at \eqref{eq:GOEeigen}. In general, doing calculation with the exact eigenvalue formulas are considerably difficult in the GOE case than the GUE case. Using the exact formulas in the GUE case in the early days in a series of papers, Wigner, Dyson, Gaudin, and Mehta proved that under proper rescaling, the joint eigenvalue distribution or the point process at an interior point of the spectrum is described by the sine kernel. One might look at \citet{mehta1960statistical}, \citet{mehta1960density}, \citet{dyson1962statistical} and \citet{dyson1970correlations} for some references. In these papers they also conjectured that this result holds for general Wigner matrices. These results were proved for a very general class of invariant 
ensembles by \citet{deift1999strong}, \citet{deift1999uniform} , \citet{bleher1999semiclassical} and \citet{pastur2008bulk}.   
 
Later on \citet{johansson} proved bulk universality for Wigner divisible ensembles. For general Wigner matrices a new approach was introduced by Erdos, Schlein, Yau, and others. One might look at \cite{erdHos2009semicircle}, \cite{erdHos2009local}, \cite{erdHos2012spectral}, \cite{erdHos2013spectral}, \cite{erdHos2010wegner}, \cite{erdHos2011universality}, \cite{erdHos2012local}, \cite{erdHos2012bulk}, \cite{erdos2010universality}, \cite{erdHos2012rigidity} for some references.  In another different approach bulk universality was also proved by \citet{tao2009random}.  One might also look at \citet{erdHos2012universality}, \citet{erdHos2012comment} for a review on the literature.

Apart from the bulk universality a different type of universality is observed at the edge of the spectrum. Here we look at the point processes of the eigenvalues near $\pm 2$ which is the support of the semicircular distribution. Using the explicit distribution of the eigenvalues the fluctuation of the largest eigenvalue of the Wigner matrix was first proved in the seminal works of Tracy and Widom \cite{tracy1994level}, \cite{tracy1996orthogonal}. In particular it is proved that 
\begin{equation}
\mathbb{P}\left[n^{\frac{2}{3}} \left( \lambda_{(1),n}- 2 \right)\le s\right] \to F_{\beta}(s)
\end{equation}
where the Tracy-Widom distribution functions $F_{\beta}$ can be described by Painleve equations, and $\beta= 1,2,4$ corresponds to Orthogonal/Unitary/Symplectic ensemble, respectively. Here $\lambda_{(1),n}\ge \ldots \ge \lambda_{(n),n}$ are the ordered eigenvalues of the Wigner matrix $W$. The joint distribution
of $k$ largest eigenvalues can be expressed in terms of the Airy kernel, which was shown by \citet{forrester1993spectrum}. In general the joint distribution of $\left( \lambda_{(1),n},\ldots , \lambda_{(k),n} \right)$ will also converge after proper rescaling and centering. In particular 
\begin{equation}
\begin{split}
\mathbb{P}\left[ n^{\frac{2}{3}}\left( \lambda_{(1),n}-2 \right)\le s_{1},\ldots, n^{\frac{2}{
3}}\left( \lambda_{(k),n} -2\right)\le s_{k} \right] \to F_{\beta,k}(s_{1},\ldots , s_{k}). 
\end{split}
\end{equation}
The $k$ dimensional distribution $F_{\beta,k}$ will also be coined as Tracy Widom distribution.
Now coming to the universality at the edge, the first result of this kind was obtained by \citet{sosh}. He assumed that the distributions of the entries of the matrix are sub-Gaussian and symmetric. The method of this paper is combinatorial in nature and is the main inspiration of our paper. Essentially the proof analyzes the trace of a high power of the Wigner matrix. Based on a similar technique and truncation \citet{ruzmaikina2006universality} proved the universality under the assumption that entries are symmetric and the tail of the entries of the matrix decay at the rate $x^{-18}$. The universality for the non-symmetric entries was first proved by \citet{tao2010random}. Here one assumes that entries have vanishing third moment and the tail decays exponentially. Finally through a different approach initiated by Erd\H{o}s, Yau and others the vanishing third moment condition was removed. One might look at \cite{erdHos2012rigidity} and \cite{erdHos2012spectral}. In these papers the results are obtained through a detailed analysis of the Green's function of the matrix. Finally a necessary and sufficient condition for the edge universality was obtained in \citet{lee2014necessary}. This paper proves that the edge universality holds if and only if $\lim_{s\to \infty} s^{4}\mathbb{P}\left[ \left|x_{1,2}\right| \ge s \right] \to 0.$ Here $x_{1,2}= \sqrt{n}W(1,2)$.

Before moving to the next section, we spend a few moments about the novelty of the current paper. As mentioned earlier we use the combinatorial approach initiated by \citet{sosh}. However unlike \citet{sosh} we do not need to assume that entries are symmetrically distributed. This is done by a very refined encoding of the contributing words defined in section \ref{sec:word}. In particular the encoding in this paper out performs the encoding by  \citet{furedi1981eigenvalues}, \citet{vu2005spectral} and \citet{peche2007wigner}. To the best of our limited knowledge this is the first paper to establish the edge universality for general non-symmetric entry distribution through combinatroial methods. We also hope the counting strategy introduced in this paper will be useful in many other different scenarios. On the other hand, by the methods in this paper we have been able to prove a combinatorial description of the Tracy Widom law. This method might be useful to characterize the Tracy Widom law when the exact calculation is not available.
 \section{The model}
 In this section we introduce the matrix ensembles. Firstly we start with the definition of Wigner matrices. 
 \begin{definition}\label{def:wig}
 We call a matrix $W= \left( x_{i,j}/ \sqrt{n} \right)_{1\le i,j \le n}$ to be a Wigner matrix if $x_{i,j}=\bar{x}_{j,i}$, $\left(x_{i,j}\right)_{1\le i<j\le n}$ are i.i.d., $\E[x_{i,j}]=0$ and $\E[|x_{i,j}|^2]=1$. 
 \end{definition}
 In this paper we deal with the real symmetric matrices and for the ease of calculation we scale the whole matrix by a factor $2$. With slight abuse of notation we shall also call this matrix a Wigner matrix and denote it by $W$. Following are the assumptions of the matrices we consider.
 \begin{assumption}\label{ass:wig}
 We consider a matrix $W$ given by $W= (x_{i,j}/\sqrt{n})_{1\le i,j \le n}$ such that the following conditions are satisfied:
 \begin{enumerate}[(i)]
 \item $x_{i,j}=x_{j,i}$ for $i\le j$.
 \item $\var(x_{i,j})=\frac{1}{4}$
 \item $\left(x_{i,j}\right)_{1\le i<j \le n}$ are i.i.d.
 \item $\E[x_{i,j}^{2k}]\le \left( \mathrm{const} .k \right)^{k}$ $\forall ~ k \in \mathbb{N}$ 
 %\item $k$ is fixed, $u_{i}$'s are mutually orthonormal vectors and $\frac{1}{2}> h_{1}>h_{2}>\ldots > h_{k}$.
 \end{enumerate}
 \end{assumption}
 Given a Wigner matrix $W$ of dimension $n \times n$ we denote its eigenvalues by $\lambda_{1,n}\ge \ldots \ge \lambda_{n,n}$. It is well known that for a Wigner matrix in Definition \ref{def:wig}, the measure $\frac{1}{n}\sum_{i=1}^{n} \delta_{\lambda_{i,n}}$ converge weakly to the semicircular law in the almost sure sense. The law is given by density in \eqref{eq:semicircular}. When we scale the entries by a factor $2$, the distribution is supported in $[-1,1]$ and its density is given by 
 \begin{equation}\label{eq:semicircularII}
 f(x)=\left\{
 \begin{array}{ll}
 \frac{2}{\pi} \sqrt{1-x^2} & \text{when~ $|x|\le 1$}\\
 0 & \text{otherwise} 
 \end{array}
 \right.
\end{equation}  
 \section{Powers of generating function of Catalan numbers}
 Since Dyck paths play a crucial role in this paper, we discuss the some properties of the Dyck paths and the generating function of the Dyck paths.
 \begin{definition}(Dyck paths)
 A Dyck path of length $2k$ is a path of a simple symmetric random walk which starts from $y=0$ , returns to $y=0$ after $2k$ step and it always stays on or above the $X$ axis.
 \end{definition}
 \begin{definition}(Catalan numbers)
 The count of all Dyck paths of length $2k$ is well known, coined as the $k$ th Catalan number and is given by 
 \begin{equation}\label{def:catalan}
 C_{k}=\left\{
 \begin{array}{ll}
 1 & \text{if $k=0$}\\
 \frac{1}{k+1}\binom{2k}{k} & \text{otherwise.}
 \end{array}
 \right.
 \end{equation}
 \end{definition}
 It is well known that whenever a random variable follows the semicircular distribution given in \eqref{eq:semicircular}, then its $2k$ th moment is $C_{k}$. 
 \begin{definition}
 The generating function of the Catalan numbers and its powers will be  quantities of interest. The generating function of the Catalan numbers is denoted by $C(x)$ and is defined as follows: 
 \begin{equation}\label{def:catgen}
 C(x)= \sum_{k=0}^{\infty} C_{k} x^{k}.
 \end{equation}
 The $m$ th power of $C(x)$ will be of fundamental interest. Fortunately the an explicit formula for the $m$ th power of $C(x)$ is known (see \citet{lang} for a reference). This is given as 
 \begin{equation}\label{def:mthcatalan}
 C^{m}(x)= \sum_{k=0}^{\infty} \frac{m}{m+2k}\binom{2k+m}{k} x^{k}.
 \end{equation}
 \end{definition}
 
\section{A brief overview of Tracy-Widom law and related stuffs}\label{sec:tw}
In this section we give a very brief overview of the point process corresponding to the eigenvalues of a Wigner matrix and the Tracy-Widom law. This part is mostly taken from \citet{sosh}. 

\noindent 
We start with the definition of GOE (Gaussian Orthogonal Ensemble)
\begin{definition}
Suppose we have a Wigner matrix $W$ as defined in Assumption \ref{ass:wig}. Then $W$ is called a GOE (Gaussian Orthogonal Ensemble) if $x_{i,j} \sim_{indep} N(0,\frac{1}{4})$ and $x_{i,i} \sim_{indep} N(0,\frac{1}{2})$.
\end{definition}

\noindent
\begin{definition}(Eigenvalue distribution of Wigner matrices:)
For GOE the eigenvalue distribution is well known. Suppose $\lambda_{1,n},\ldots , \lambda_{n,n}$ be the eigenvalues of a GOE matrix and we assume $P_{n,1}(\lambda_{1},\ldots, \lambda_{n})$ is the eigenvalue distribution. Then 
\begin{equation}\label{eq:GOEeigen}
\begin{split}
dP_{n,1}\left( \lambda_{1,n},\ldots, \lambda_{n,n} \right)= C \prod_{1\le i < j \le n} \left|\lambda_{i,n} - \lambda_{j,n} \right|\exp\left\{ - n \sum_{i=1}^{n} \lambda_{i,n}^2 \right\}\prod_{i=1}^{n}d\lambda_{i,n}
\end{split}
\end{equation}
Although not explicitly known, we shall denote the eigenvalue distribution of a general Wigner matrix by $P_{n,\mathrm{Gen}}(\lambda_{1,n},\ldots, \lambda_{n,n})$
\end{definition}

\noindent 
Studying the eigenvalues near the support is done by studying the $k$ point correlation function of the eigenvalue distribution. This is what we define next.
\begin{definition}($k$ point correlation function)
The $k$ point correlation function of the eigenvalue distribution of the GOE is defined as 
\begin{equation}
\rho_{n,1,k}(\lambda_{1,n},\ldots, \lambda_{k,n})= \frac{n!}{(n-k)!}\int_{\mathbb{R}^{n-k}} \frac{d^{n}P_{n,1}(\lambda_{1,n},\ldots, \lambda_{n,n})}{d\lambda_{1,n}d\lambda_{2,n}\ldots d\lambda_{n,n}} d \lambda_{k+1,n} \ldots d\lambda_{n,n}.
\end{equation} 
Similarly for general Wigner matrices the $k$ point correlation function is defined as 
\begin{equation}
\rho_{n,\mathrm{Gen},k}(\lambda_{1,n},\ldots, \lambda_{k,n})= \frac{n!}{(n-k)!}\int_{\mathbb{R}^{n-k}} \frac{d^{n}P_{n,\mathrm{Gen}}(\lambda_{1,n},\ldots, \lambda_{n,n})}{d\lambda_{1,n}d\lambda_{2,n}\ldots d\lambda_{n,n}} d \lambda_{k+1,n} \ldots d\lambda_{n,n}.
\end{equation}
\end{definition} 
For the precise formulas of $\rho_{n,1,k}(\lambda_{1,n},\ldots, \lambda_{k,n})$ one might look at \citet{mehta2004random} Chapter 5 and 6. $k$ point correlation functions are
particularly useful in calculating the moments of the number of eigenvalues
in an interval $I \subset \mathbb{R}$ . Let $\nu_{n,I,1}(\text{resp.} ~ \nu_{n,I,\mathrm{Gen}})$ be the number of eigenvalues in $I$. Then the mathematical expectation of $\nu_{n,I,1}(\text{resp. } \nu_{n,I,\mathrm{Gen}})$
is given by the formula
\begin{equation}
\E\left[ \nu_{n,I,1}(\text{resp. } \nu_{n,I,\mathrm{Gen}}) \right] = \int_{I} \rho_{n,1,1}(x)(\text{resp. } \rho_{n,1,\mathrm{Gen}}(x)) dx,
\end{equation}
and in general 
\begin{equation}
\E\left[ \nu_{n,I,1}(\nu_{n,I,1}-1)\ldots (\nu_{n,I,1}-k+1)\right]= \int_{I^{k}} \rho_{n,1,k}(x_{1},\ldots, x_{k}) dx_{1} \ldots dx_{k}.
\end{equation}
In the subsequent part of this section we shall only consider the GOE case since the explicit form of the eigenvalue distribution is known.

\noindent 
In order to find out the scaling limit of the eigenvalues at the edge, we shrink the length of the interval $I$ such that in expectation there are only finitely many eigenvalues in the interval. In order to achieve this for a given point $x$ we take the interval $I_{n,x}= \left[x- \frac{c_{1}}{\rho_{n,1,1}(x)}, x+ \frac{c_{2}}{\rho_{n,1,1}(x)}  \right]$ so that the integral 
\[
\int_{I_{n,x}} \rho_{n,1,1}(y)dy=O(1).
\]
This allows us to consider the following rescaling of the eigenvalues for any given $x$
\[
\lambda_{i,n}:= x+ y_{i,n}\frac{1}{\rho_{n,1,1}(x)}. 
\]
we also consider the following rescaled correlation function:  
\[
R_{n,1,k,x}(y_{1},\ldots, y_{k}) = \frac{1}{\rho^{k}_{n,1,1}(x)}\rho_{n,1,k}\left( \lambda_{1,n},\ldots , \lambda_{k,n} \right)
\]
It is easy to see that if $I_{n,x}= \left[x- \frac{c_{1}}{\rho_{n,1,1}(x)}, x+ \frac{c_{2}}{\rho_{n,1,1}(x)}  \right]$ is an interval containing $x$ and is of length $O\left(\frac{1}{\rho_{n,1,1}(x)}\right)$ then
\[
\E\left[ \nu_{n,I_{n,x},1}\left( \nu_{n,I_{n,x},1} -1\right)\ldots \left( \nu_{n,I_{n,x},1}-k +1 \right) \right]= \int_{[c_{1},c_{2}]^{k}} R_{n,1,k,x}(y_{1},\ldots, y_{k}) dy_{1}\ldots dy_{k}.
\] 
To show that $\nu_{n,I_{n,x},1}$ converges to a limit in distribution as $n \to \infty$, one needs
to show that the rescaled $k$-point correlation functions have a limit too. 
This is well known in the Gaussian case (i.e. $R_{n,1,k,x}(y_{1,n},\ldots, y_{k,n}) \to R_{1,k,x}(y_{1},\ldots, y_{k})  $). One might look at \citet{mehta2004random} for a reference. Here we shall be interested in $R_{1,k,1}(y_{1},\ldots, y_{k})$. For the GOE case the function $R_{1,k,1}(y_{1},\ldots, y_{k})$ is defined in the following way: Let 
\begin{equation}\label{eq:kxy}
\begin{split}
K(y,z)= \frac{\mathcal{A}i(y)\mathcal{A}i'(z)- \mathcal{A}i(z) \mathcal{A}i'(y)}{y-z}.
\end{split}
\end{equation}
Here $\mathcal{A}i(x)$ is the Airy function and is given by the solution of the differential equation $f''(y)=yf(y)$ with the asymptotics $f(y)\sim \frac{1}{2\sqrt{\pi} y^{\frac{1}{4}}}\exp\left\{ -\frac{2}{3} y^{\frac{3}{2}} \right\}$ as $y \to \infty$. 
Now $R_{1,k,1}$ can be represented as the square root of determinant of a $2k \times 2k$ matrix consisting of $2\times 2$ blocks $\left(\xi_{1}(y_{i},y_{j})\right)_{1\le i ,j \le k}$. To define $\xi_{1}$ we introduce a few more notations. Let 
\begin{equation}\label{eq:dk}
\begin{split}
DK(y,z) &= - \frac{d}{dz} K(y,z)\\
JK(y,z) &= - \int_{y}^{\infty} K(t,z)dt - \frac{1}{2} \mathrm{sgn}(y-z).
\end{split}
\end{equation}
Then 
\begin{equation}\label{eq:xi}
\xi_{1}(y,z)= \left[
\begin{array}{ll}
K(y,z) + \frac{1}{2}\mathcal{A}i(y)\int_{-\infty}^{t} \mathcal{A}i(t) dt & -\frac{1}{2} \mathcal{A}i(y)\mathcal{A}i(z) + DK(y,z)\\
JK(y,z) + \frac{1}{2} \int_{z}^{y}\mathcal{A}i(u) du + \frac{1}{2} \int_{z}^{\infty} \mathcal{A}i(u) du \int_{-\infty}^{z}\mathcal{A}i(v)dv & K(z,y) + \frac{1}{2}\mathcal{A}i(z)\int_{-\infty}^{y}\mathcal{A}i(t)dt
\end{array}
\right]
\end{equation} 
and 
\begin{equation}\label{eq:tracywidomcorrelation}
R_{1,k,1}(y_{1},\ldots, y_{k})= \sqrt{ \left(\mathrm{det}\left( \xi_{1}(y_{i},y_{j}) \right)_{i,j=1}^{k}  \right)}.
\end{equation}
Finally the distribution of the largest entry of this point process can be determined in the following way: Let $\lambda_{\max}$ be the largest entry of the limiting point process. Then 
\begin{equation}
\begin{split}
&\mathbb{P}\left[ \lambda_{\max} \le t \right]= \mathbb{P}\left[ \# \text{of points of the point process } \in [t,\infty) = 0 \right]\\
& = \sum_{k=0}^{\infty} \frac{(-1)^{k}}{k!} \int_{[t,\infty)^{k}} R_{1,k,1}(y_{1},\ldots, y_{k}) dy_{1}\ldots dy_{k}.
\end{split}
\end{equation}
This concludes our discussion about the GOE Tracy-Widom distribution.
\section{Main results}
In this section we state our main result. These results can also be proved under weaker assumptions by the approach of Erd\H{o}s, Yau and others (see \cite{erdHos2012rigidity} and \cite{erdHos2012spectral}). However as mentioned at the beginning of the paper, the fundamental objective of this paper is to provide a combinatorial method to prove these results and extend the approach in \citet{sosh} for matrices where entries are sub-Gaussian but possibly have non-symmetric distributions.

\noindent 
We now state the main result of this paper.
\begin{theorem}\label{thm:universal}
Suppose $W$ is a Wigner matrix satisfying Assumption \ref{ass:wig} with eigenvalues $\lambda_{(1),n}\ge \ldots \ge\lambda_{(n),n}$. Then following are true:
\begin{enumerate}[(i)]
\item The point process at the edge of the spectrum (i.e. at $x=\pm 1$ according to notations in Section \ref{sec:tw}) converges weakly to the limiting point process at the edge of the spectrum of GOE ensemble.  
\item As a consequence, for any fixed $k$ the joint distribution of $\left(n^{\frac{2}{3}}2\left(\lambda_{(1),n}-1\right), \ldots, n^{\frac{2}{3}}2\left(\lambda_{(k),n} -1 \right) \right)$ converge weakly to the limiting joint distribution of the top $k$ eigenvalues of the GOE. In particular $n^{\frac{2}{3}}2\left(\lambda_{(1),n}-1\right)$ weakly converges to the GOE Tracy Widom distribution.
\end{enumerate} 
\end{theorem}
\section{Strategies to prove the edge universality}
As discussed earlier, our main technique is combinatorial in nature. In this section we briefly give an overview of the strategies to prove the edge universality. 

\noindent
At the beginning, we start with the approach by \citet{sosh}. Here given the matrix $W$ with eigenvalues $\lambda_{1,n}, \ldots , \lambda_{n,n}$, we write 
\[
\lambda_{i,n}= 1+ \frac{\theta_{i,n}}{2n^{\frac{2}{3}}}.
\]
if $\lambda_{i,n}>0$ and 
\[
\lambda_{i,n} = -1 - \frac{\tau_{i,n}}{2n^{\frac{2}{3}}}
\] if  $\lambda_{i,n}\le 0$.
Now we shall consider a very high power of the matrix and compute it's trace. In particular, for any $t \in (0,\infty)$ we consider $[tn^{\frac{2}{3}}]$ and consider $\Tr\left[ W^{2[tn^{\frac{2}{3}}] } \right]$. Since 
\[
\Tr[W^{k}] = \sum_{i=1}^{n} \lambda_{i,n}^{k},
\] 
we have 
\begin{equation}\label{eq:traceexplaplace}
\begin{split}
&\Tr[W^{2[tn^{\frac{2}{3}}]}]= \sum_{\lambda_{i,n}>0} \left(1+ \frac{\theta_{i,n}}{2n^{\frac{2}{3}}}\right)^{2[tn^{\frac{2}{3}}]}+ \sum_{\lambda_{i,n}\le 0} \left( 1+ \frac{\tau_{i,n}}{2n^{\frac{2}{3}}} \right)^{2[tn^{\frac{2}{3}}]}.
\end{split}
\end{equation}
Now among the terms in \eqref{eq:traceexplaplace}, we can ignore the terms when $0\le \lambda_{i,n} \le \left( 1- \frac{1}{2n^{\frac{1}{2}}} \right)$  and $0 \ge \lambda_{i,n} \ge (-1 + \frac{1}{2n^{\frac{1}{2}}}) $. This is due to the fact that both $\left(1- \frac{1}{2n^{\frac{1}{2}}}\right)^{[tn^{\frac{2}{3}}]}= O\left(e^{-cn^{\frac{1}{6}}}\right)$ for some fixed constant $c$. As a consequence, the sum corresponding to these terms goes to $0$.  

\noindent 
Before proceeding further, we introduce two results. These results are the main technical contribution of the paper and their proofs are given in Section \ref{sec:proof}. Similar results can be found in \citet{sosh} but with additional assumption that the distributions of the entries of the matrix are symmetric around $0$. 
\begin{theorem}\label{thm:traceconvergence}
Consider the Wigner matrix $W$ satisfying Assumption \ref{ass:wig}. Then for any fixed $t \in (0,\infty)$ taking $k= \left[tn^{\frac{2}{3}}\right]$, we have the following results 
\begin{enumerate}
\item $\E\Tr\left[ W^{2k} \right]=O(1)$ and $\E\Tr\left[W^{2k+1}\right]=o(1)$.
\item If the limit of $\lim_{n \to \infty} \E\Tr\left[ W^{2k} \right] $ for some $t\in (0,\infty)$ exists, then the limit only depends on the first and second moment of entries. 
\item As the limit exists for Gaussian entries, the limit exists and is universal for any Wigner matrix satisfying Assumption \ref{ass:wig}. 
\end{enumerate}
\end{theorem}
\begin{theorem}\label{thm:jointmom}
Consider the Wigner matrix $W$ satisfying Assumption \ref{ass:wig}. Then for any fixed $t_{1}, \ldots, t_{l} \in (0,\infty)^{l} $ taking $k_{i}= \left[ t_{i} n^{\frac{2}{3}} \right]$, we have the following results
\begin{enumerate}
\item 
\begin{equation}\label{eq:jointlim}
\E\left[ \prod_{i=1}^{l} \left[\Tr\left[W^{k_{i}}\right] - \E\left[  \Tr\left[W^{k_{i}}\right]\right] \right] \right]=O(1).
\end{equation}
\item If the limit in \eqref{eq:jointlim} exists for some $t_{1},\ldots , t_{l}$, then the limit only depends on the first and second moment of entries.
\item As the limit exists for Gaussian entries, the limit exists and is universal for any Wigner matrix satisfying Assumption \ref{ass:wig}.
\end{enumerate}
\end{theorem}
Now assuming part 1 and 2 of both Theorems \ref{thm:traceconvergence} and \ref{thm:jointmom} we describe rest of the proof techniques. 

\noindent 
We now divide the sum in \eqref{eq:traceexplaplace} into a few further cases. In particular we consider 
\begin{equation}\label{eq:truncttail}
\begin{split}
s_{1}&= \left[\sum_{\lambda_{i,n} \ge 1+ \frac{2}{n^{\frac{1}{2}}}}  \lambda_{i,n}^{2[tn^{\frac{2}{3}}]} \right]\\
s_{2}&=  \left[\sum_{\lambda_{i,n} \le -1- \frac{2}{n^{\frac{1}{2}}}}  \lambda_{i,n}^{2[tn^{\frac{2}{3}}]} \right].
\end{split}
\end{equation}
We shall at first show that the terms described in \eqref{eq:truncttail} go to $0$ almost surely. We shall show this only for $s_{1}$ and the argument for $s_{2}$ is exactly same. 
\begin{equation}
\begin{split}
\mathbb{P}\left[ s_{1} >0 \right]&\le \mathbb{P}\left[ \lambda_{(1),n} > \left(1+\frac{2}{n^{\frac{1}{2}}}\right) \right]\\
& \le \mathbb{P}\left[ \Tr[W^{2[tn^{\frac{2}{3}}]}] \ge \left( 1+ \frac{2}{n^{\frac{1}{2}}} \right)^{2[tn^{\frac{2}{3}}]}\right]\\
& \le \frac{\E\left[\Tr[W]^{2[tn^{\frac{2}{3}}]}\right] }{\left( 1+ \frac{2}{n^{\frac{1}{2}}} \right)^{2[tn^{\frac{2}{3}}]}}=O\left( e^{-cn^{\frac{1}{6}}} \right).
\end{split}
\end{equation}  
So by Borel-Cantelli theorem $s_{1}= 0$ almost surely. As a consequence, 
\begin{equation}\label{eq:trnc1}
\Tr[W^{2[tn^{\frac{2}{3}}]}]- \sum_{\left(1- \frac{2}{\sqrt{n}}\right)\le \lambda_{i,n}\le \left( 1+ \frac{2}{\sqrt{n}} \right)} \lambda_{i,n}^{2[tn^{\frac{2}{3}}]} - \sum_{\left( -1 - \frac{2}{\sqrt{n}} \right)\le \lambda_{i,n}\le \left( -1 + \frac{2}{\sqrt{n}} \right)}\lambda_{i,n}^{2[tn^{\frac{2}{3}}]} \stackrel{a.e.}{\to} 0.
\end{equation}
Since all the terms of the l.h.s. of \eqref{eq:trnc1} are always greater than $0$ and by Theorem \ref{thm:jointmom} we have $$\limsup_{n}\E\left[ \Tr\left[W^{2[tn^{\frac{2}{3}}]}\right]^2 \right]=O(1)$$, we have the expectation of the l.h.s. of \eqref{eq:trnc1} also goes to $0$ by uniform integrability. Now
\begin{equation}\label{eq:limgau}
\begin{split}
&\sum_{\left(1- \frac{2}{\sqrt{n}}\right)\le \lambda_{i,n}\le \left( 1+ \frac{2}{\sqrt{n}} \right)} \lambda_{i,n}^{2[tn^{\frac{2}{3}}]} + \sum_{\left( -1 - \frac{2}{\sqrt{n}} \right)\le \lambda_{i,n}\le \left( -1 + \frac{2}{\sqrt{n}} \right)}\lambda_{i,n}^{2[tn^{\frac{2}{3}}]}\\
& = \left(\sum_{\theta_{j} \le n^{\frac{1}{6}}} e^{t\theta_{j}} + \sum_{\tau_{j}\le n^{\frac{1}{6}}} e^{t \tau_{j}}\right)\left(1+ O \left(  n^{-\frac{1}{3}}\right)\right)\\
\Rightarrow & \E\left[ \sum_{\left(1- \frac{2}{\sqrt{n}}\right)\le \lambda_{i,n}\le \left( 1+ \frac{2}{\sqrt{n}} \right)} \lambda_{i,n}^{2[tn^{\frac{2}{3}}]} + \sum_{\left( -1 - \frac{2}{\sqrt{n}} \right)\le \lambda_{i,n}\le \left( -1 + \frac{2}{\sqrt{n}} \right)}\lambda_{i,n}^{2[tn^{\frac{2}{3}}]}  \right]\\
& = \left(1+ O\left( n^{-\frac{1}{3}} \right)\right)\left(\int_{-\infty}^{\infty} e^{ty} R_{n,1,1,1}(y)dy + \int_{-\infty}^{\infty} e^{ty} R_{n,1,1,-1}(y)dy \right)
\end{split}
\end{equation} 
Similarly one can prove that 
\begin{equation}
\begin{split}
&\E\left[ \Tr\left[W^{2[tn^{\frac{2}{3}}]+1}\right]  \right]\\
&= \E\left[ \sum_{\left(1- \frac{2}{\sqrt{n}}\right)\le \lambda_{i,n}\le \left( 1+ \frac{2}{\sqrt{n}} \right)} \lambda_{i,n}^{2[tn^{\frac{2}{3}}]+1} - \sum_{\left( -1 - \frac{2}{\sqrt{n}} \right)\le \lambda_{i,n}\le \left( -1 + \frac{2}{\sqrt{n}} \right)}\lambda_{i,n}^{2[tn^{\frac{2}{3}}]+1} \right]\\
&= \left(1+ O\left( n^{-\frac{1}{3}} \right)\right)\left(\int_{-\infty}^{\infty} e^{ty} R_{n,1,1,1}(y)dy - \int_{-\infty}^{\infty} e^{ty}R_{n,1,1,-1}(y)dy  \right)
\end{split}
\end{equation}
Hence 
\begin{equation}
\begin{split}
\frac{1}{2}\E\left[\Tr\left[ W^{2[tn^{\frac{2}{3}}]} \right]+ \Tr \left[ W^{2[tn^{\frac{2}{3}}]+1} \right]\right]= \left( 1+ O\left(n^{-\frac{1}{3}}\right) \right)\left(\int_{-\infty}^{\infty} e^{ty} R_{n,1,1,1}(y)dy\right).
\end{split}
\end{equation}
This implies that for the Gaussian case the l.h.s. of the last expression of \eqref{eq:limgau} exists and equals to 
\begin{equation}
\int_{-\infty}^{\infty} e^{ty} R_{1,1,1}(y)dy.
\end{equation}
This proves the part 3 of Theorem \ref{thm:traceconvergence}.

\noindent
Same can be said for part 3 of Theorem \ref{thm:jointmom}. However here we shall get a polynomial of multi-dimensional Laplace transform. 

\noindent 
Now coming back to general case, by these two results we get the Laplace transforms of the general correlation functions also converge to the same limit as the Gaussian. As convergence of the Laplace transform for all $t>0$ implies the weak convergence of a measure, we have the correlation functions for the general case converge weakly to the limit of the correlation function of the Gaussian case. This completes the proof of Theorem \ref{thm:universal}.  
\section{Combinatorial preliminaries}\label{sec:word}
\subsection{Introductory definitions}
In this subsection we develop some preliminaries about the method of moments and the word sentence approach for random matrices.

 %In this section we develop some introductory notations and definitions essential for the word sentence analysis of method moments of random matrix theory.
To begin with we start with a matrix $W$ of dimension $n \times n$. Its $k$ th moment is given by 
\begin{equation}\label{eq:traceexpression}
\Tr[W^{k}]= \sum_{i_{0},i_{1},\ldots ,i_{k-1}, i_{0}} W_{i_{0},i_{1}}\ldots W_{i_{k-1},i_{0}}.
\end{equation} 
The word sentence method systematically analyzes the tuples $(i_{0}, \ldots, i_{k-1},i_{0})$ for some suitable $k$. To do this we need some notations and definitions.
\label{subsec:word}
 
In this part we give a very brief introduction to words, sentences and their equivalence classes essential for the combinatorial analysis of random matrices. 
The definitions are taken from \citet{AGZ} and \citet{AZ05}.
For more general information, see \cite[Chapter 1]{AGZ} and \cite{AZ05}. 
% The definitions in this section have been taken from  Anderson et al. \cite{AGZ} and Anderson and Zeitouni \cite{AZ05}.
\begin{definition}[$\mathcal{S}$ words]
Given a set $\mathcal{S}$, an $\mathcal{S}$ letter $s$ is simply an element of
$\mathcal{S}$. An $\mathcal{S}$ word $w$ is a finite sequence of letters $s_1
\ldots s_k$, at least one letter long.
An $\mathcal{S}$ word $w$ is \emph{closed} if its first and last letters are the same. In this paper $\mathcal{S}=\{1,\ldots,n\}$ where $n$ is the number of nodes in the graph.
\end{definition}
% \begin{definition}
Two $\mathcal{S}$ words
$w_1,w_2$ are called \emph{equivalent}, denoted $w_1\sim w_2$, if there is a bijection on $\mathcal{S}$ that
maps one into the other.
For any word $w = s_1 \ldots s_k$, we use $l(w) = k$ to denote the \emph{length} of $w$, define
the \emph{weight} $wt(w)$ as the number of distinct elements of the set ${s_1,\ldots , s_k
}$ and the
\emph{support} of $w$, denoted by $\mathrm{supp}(w)$, as the set of letters appearing in $w$. 
With any word $w$ we may associate an undirected graph, with $wt(w)$ vertices and at most $l(w)-1$ edges,
as follows.
\begin{definition}[Graph associated with a word] 
	\label{def:graphword}
Given a word $w = s_1 \ldots s_k$,
we let $G_w = (V_w,E_w)$ be the graph with set of vertices $V_w = \mathrm{supp}(w)$ and (undirected)
edges $E_w = \{\{s_i, s_{i+1}
\}, i = 1,\ldots ,k - 1
\}.$
\end{definition}
The graph $G_w$ is connected since the word $w$ defines a path connecting all the
vertices of $G_w$, which further starts and terminates at the same vertex if the word
is closed. 
% For $e \in E_w$, we use $N^w_
% e$ to denote the number of times this path traverses
% the edge $e$ (in any direction).
We note that equivalent words generate the same
graphs $G_w$ (up to graph isomorphism) and the same passage-counts of the edges. 
% $N^w_
% e$.
Given an equivalence class $\mathbf{w}$, we shall sometimes denote $\#E_{\mathbf{w}}$ and $\#V_{\mathbf{w}}$ to be the common number of edges and vertices for graphs associated with all the words in this equivalence class $\mathbf{w}$. 

\begin{definition}[Weak Wigner words]
	\label{def:weakwigner}
Any word $w$ will be called a \emph{weak Wigner word} if the following conditions are satisfied:
\begin{enumerate}
\item $w$ is closed.
\item $w$ visits every edge in $G_{w}$ at least twice. 
\end{enumerate}
\end{definition}
Suppose now that $w$ is a weak Wigner word. If $wt(w) = (
l(w) + 1)/2$, then we
drop the modifier ``weak" and call $w$ a \emph{Wigner word}. (Every single letter word is
automatically a Wigner word.) 
Except for single letter words, each edge in a Wigner word is traversed exactly twice.
If $wt(w) = (l(w)-1)/2$, then we call $w$ a \emph{critical
weak Wigner word}.

It is a very well known result in random matrix theory that there is a bijection from the set of the Wigner words of length $2k+1$ to the set of Dyck paths of length $2k$. We briefly discuss this map when we construct the  well behaved words.

 \subsection{Mapping of words to Dyck paths}
 The fundamental idea of \citet{sosh} is to map the closed words such that all edges are traversed even number of times to Dyck paths. It is worth noting that given a Dyck path there will be multiple equivalence classes of words. In particular the map is not one to one. The main goal of this section is to understand this map explicitly and extend the ideas to possible cases when the closed words does not have all edges traversed even number of times. To understand the ideas clearly we need the following terminologies.

 \begin{definition}(Well behaved words)
These are the words that can be naturally mapped to a Dyck path in the following way. We start with a Wigner word and from \citet{AGZ} we know that the corresponding graph is a tree. We merge some vertices in the tree to incorporate cycles. As an example one might consider the following word: $w=(1,2,3,5,3,2,4,2,1)$ and we merge the vertices $5$ and $1$ and we call the common letter $1$. So the transformed word is $w'=(1,2,3,1,3,2,4,2,1)$. These words can be mapped to a Dyck path as follows: one start a random walk from $0$ and whenever one traverse an edge odd number of time one goes one step up in the random walk and whenever one traverse an edge even number of times one goes one step down in the random walk. For example the random walk values corresponding to the word $(1,2,3,1,3,2,4,2,1)$ look like $(0,1,2,3,2,1,2,1,0)$.with the additional constraint that the vertex corresponding to value $3$ in the random walk is labeled as $1$ which is the same as the vertex corresponding to value $0$.  
\begin{figure}[H]
        \begin{center}
                \includegraphics[width=0.8\textwidth]{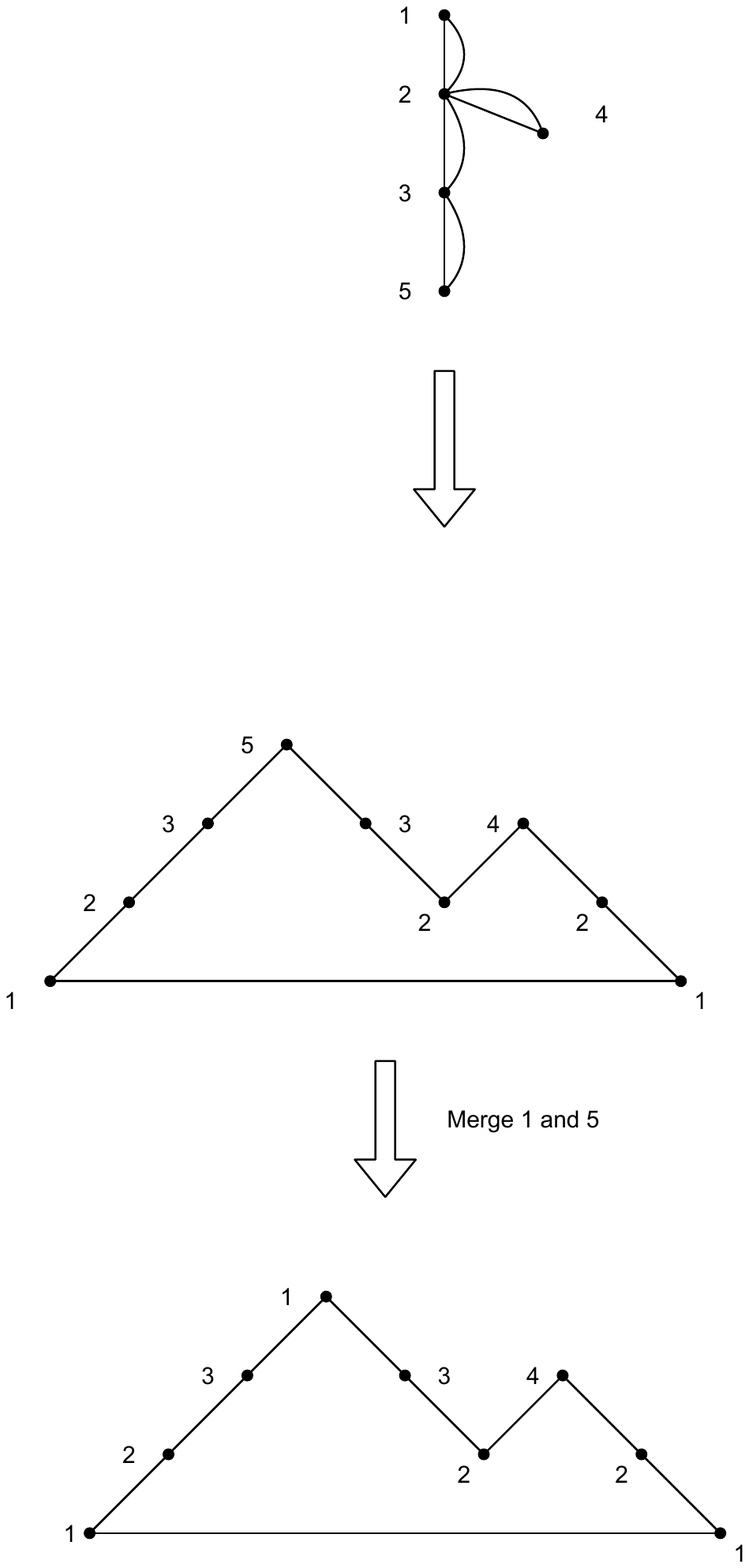}
        ~ %add desired spacing between images, e. g. ~, \quad, \qquad etc.
          %(or a blank line to force the subfigure onto a new line)
      \end{center}   
 \end{figure} 
\end{definition}
Before moving forward, we state an important convention. In several places we need to count the number of times a Dyck path returns to certain level. However from the construction of the Dyck paths it is clear when the Dyck path falls down from a certain level and comes back to the level from below, these points have possibly different labels. So whenever we talk about the Dyck path coming to a specific level we shall always mean that the Dyck path returns to the level before falling down. 

Unfortunately not all words are well behaved. Perhaps the simplest example of such word is $(1,2,3,1,2,3,1)$. Observe that if we want to construct a Dyck path just like the well behaved words we shall encounter problems. To make the idea understandable, we start constructing the Dyck path as follows we start from $0$(vertex $1$) and strictly increase to $3$ (which again comes to vertex $1$) but once we reach that point there is no obvious way to continue the Dyck path. This happens because the next edge is $\{ 1,2 \}$ which although appeared in the Dyck path previously but not the edge appeared immediately before which is $\{ 1,3\}$. 
\begin{figure}[H]
        \begin{center}
                \includegraphics[width=0.5\textwidth]{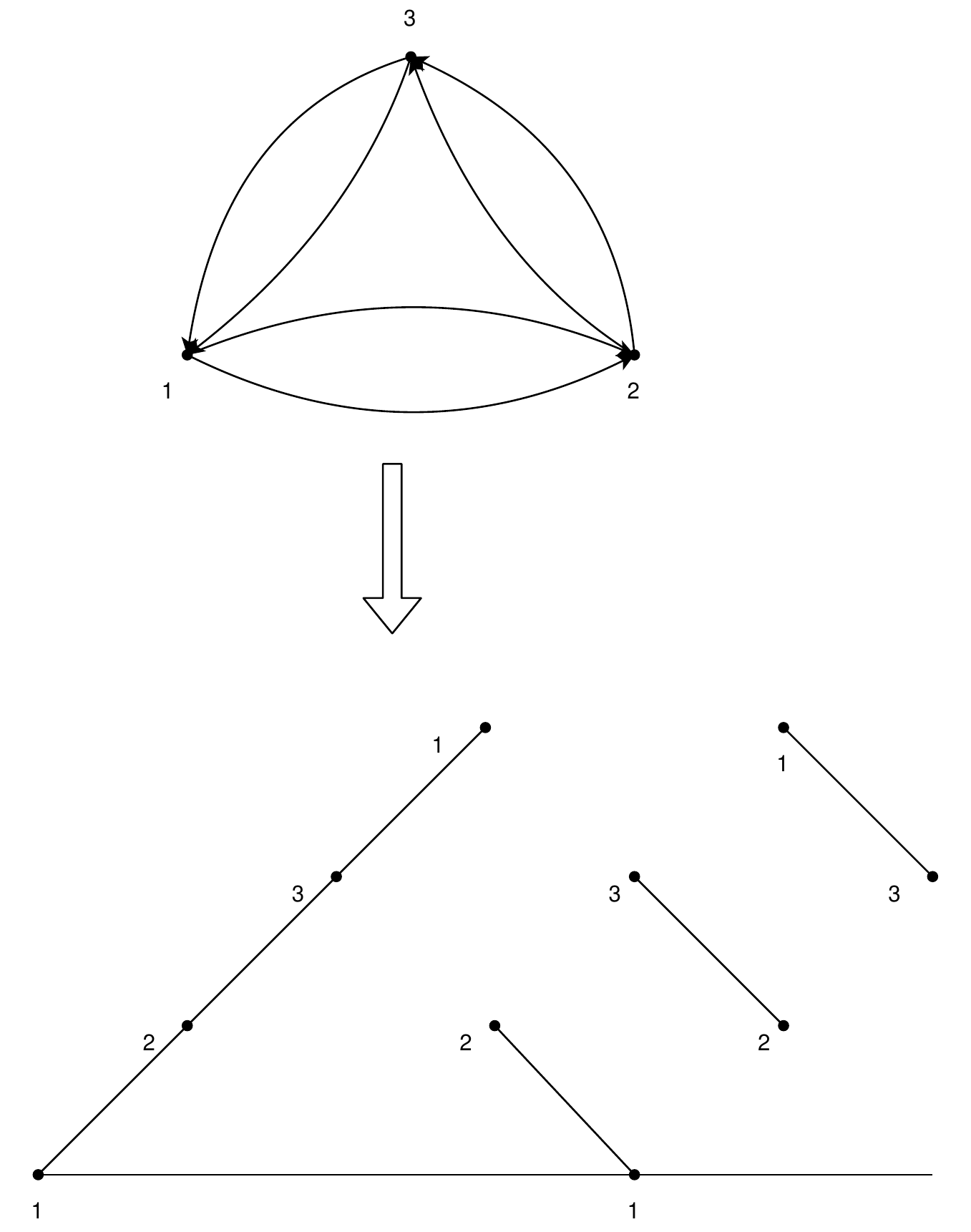}
        ~ %add desired spacing between images, e. g. ~, \quad, \qquad etc.
          %(or a blank line to force the subfigure onto a new line)
      \end{center}   
 \end{figure} 
 Now we peek at the stack interpretation of the Dyck path. From this example we see that we encounter problems continuing the Dyck path when an edge is closed in the word which is not the top most edge in the stack.

 This leads to the definition of non well behaved words. 
\begin{definition}(Non well behaved words)
The non well behaved words are defined as follows. We start a Dyck path following the vertex exploration of the word. However there is at least one instant such the Dyck path can not be continued further if we follow the vertex exploration of the graph. In other words the word exploration is such that there is an instant when the word closes an edge which is not at the top of the stack at that time instant. 
\end{definition} 
To elaborate our idea clearly we go back to the example of the word $(1,2,3,1,2,3,1)$. So we start forming a Dyck path in standard way. In particular the random walk goes strictly in the up ward direction until it hits $3$(with corresponding vertex $1$). Now the Dyck path can't be continued. From this point we start creating the segments. The rest of the word looks like $1,2,3,1$. So at the first step it closes the edge $\{ 1,2 \}$ which corresponds to the upward step from $0$ to $1$ in the Dyck path formed till now. To match this edge in our first segment we draw a downward segment from $1$ to $0$. (Observe that this creates discontinuity in the path. So in the last step of assembling we assemble the segments in such a way that it corresponds to a Dyck path.) In the next two steps we create two more downward segments from $2$ to $1$ (corresponding to edge $\{ 3,2 \}$) and from $3$ to $2$ (corresponding to edge $\{ 3,1 \}$). Now in the final step we 
assemble the segments to get a Dyck path. In particular for the example we are concerned, there is only one Dyck path which can be assembled from the segments. That is a random walk strictly increase from $0$ to $3$ and then strictly decreased from $3$ to $0$. %The following figure gives a pictorial description of what we discussed. 
%\begin{figure}[H]
%        \begin{center}
%                \includegraphics[width=0.5\textwidth]{2.pdf}
%        ~ %add desired spacing between images, e. g. ~, \quad, \qquad etc.
%          %(or a blank line to force the subfigure onto a new line)
%      \end{center}   
% \end{figure} 

\noindent  
 In the next part we elaborate this idea for general non well behaved words. We in particular, give an algorithm (Algorithm \ref{alg:encode}) to encode a general non well behaved word. 
 \begin{algorithm}\label{alg:encode}
 This map is recursive in nature. It is done in the following steps:
\begin{enumerate}[(i)]
\item We start constructing the Dyck path according to the exploration of the word. We continue until we arrive at a situation such that we have to close an edge which is not the immediate one.
\item In this scenario we jump to that level and close the edge. Hence we arrive at a different level. From that level we go back to the previous step. To make the output of the map look like a Dyck path, whenever we close an edge from right to left and create a new edge in the very next step, we create the upward segment in the left of the edge just closed. 
\item The process ends when all the edges are covered. 
\end{enumerate}
 \end{algorithm}

\noindent 
One important comment about Algorithm \ref{alg:encode} is, given any word the output path might not be a Dyck path. This is simply because of the fact that some edges in the Dyck path might not be present in the word. For example one might consider the word $(1,2,3,1)$. Following is the path corresponding to the given word:
\begin{figure}[H]
        \begin{center}
                \includegraphics[width=0.5\textwidth]{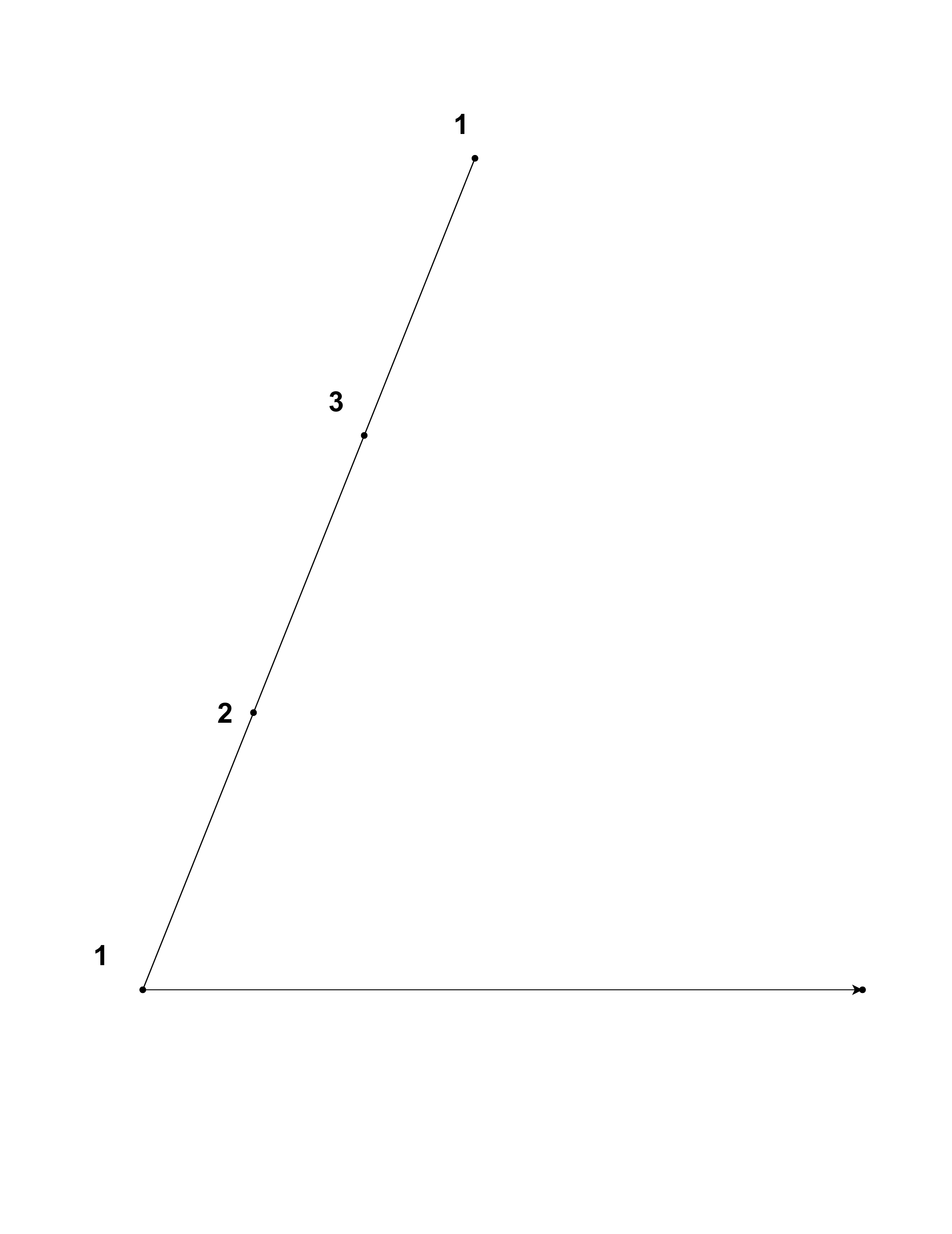}
        ~ %add desired spacing between images, e. g. ~, \quad, \qquad etc.
          %(or a blank line to force the subfigure onto a new line)
      \end{center}   
 \end{figure} 
By Algorithm \ref{alg:encode} we have created a map from the class of words to the paths which are obtained by removing some edges from a Dyck path. However this map is no way one to one. To see this one might consider the following two words $w_{1}:=(1,2,3,1,2,3,1)$ and $w_{2}:=(1,2,3,1,3,2,1)$. It is easy to see that the paths corresponding to the words $w_{1}$ and $w_{2}$ are the same. So given a path the main task is to bound the cardinality of the inverse image of the path in such a manner that the expectation of the quantity in \eqref{eq:traceexpression} remains bounded.
  
% we need to reinterpret the definition of type of an instant and open and closed instants in \citet{sosh}. 
 
\begin{definition}(Type of an instant)\label{def:type}
We start with a word (possibly non well behaved) and apply Algorithm \ref{alg:encode} to the word. We call an instant to be of type $k\ge 1$ if there is an upward step at that instant and it is the $k$ th appearance of the vertex corresponding to that instant as the right endpoint of an upward step. Further if a vertex appears $k$ times as a right endpoint of an upward edge, we call that vertex to be of type $k$.
%We start with a word (possibly non well behaved) and we start forming a Dyck path corresponding to the exploration of the word. We continue Dyck path until we encounter a previously encountered vertex. We call this vertex of type $2$. Now construct a possibly new segment of Dyck path from this vertex of type $2$ continue until another previously encountered vertex we call this vertex another vertex of type $2$ if it's value is not same as the first vertex of type $2$. Otherwise call this vertex of type $3$. We again continue with a possible new segment of Dyck path and define the types of the vertices recursively. Note that the definition of type of a vertex coincides with that of \citet{sosh} for well behaved words.
\end{definition}  
\begin{definition}(Open and closed instants)\label{def:openclosed}
An instant of type $>1$ is called open if the exploration of the word is such that when the instant is encountered, there is at least one unmatched edge which is not the immediate edge, on the vertex corresponding to that instant. Otherwise an instant is called closed.
\end{definition}
To clarify the idea consider the word $(1,2,3,1,3,2,4,3,4,2,1)$. Firstly observe that this word is well behaved. However the idea can be interpreted analogously for non well behaved words. The random walk corresponding to this word looks like this: it starts from $0$ and move in strictly upward direction until it reaches $3$ (corresponding vertex is $1$). Then it goes down $2$ steps to reach $1$ (corresponding vertex is $2$). Now it further goes upward $2$ steps to reach $3$ once again (corresponding vertices are $4$ and $3$). Then it strictly decrease to come to $0$. Observe that there are two instants of type $2$ in this example. The first one is the second appearance of $1$ and the second one is the third appearance of $3$. Among them the second appearance of $1$ is open as there is an unmatched edge (i.e. $\{ 1,2\}$)  apart from the immediate edge (i.e. $\{1,3 \}$). On the other hand the third appearance of $3$ is closed since all the  edges incident to $3$ apart from the immediate edge (i.e. $\{ 4,3 \}$) (i.e. edges $\{ 2,3\}$ and $\{ 3,1 \}$) are matched.
\begin{figure}[H]
        \begin{center}
                \includegraphics[width=0.5\textwidth]{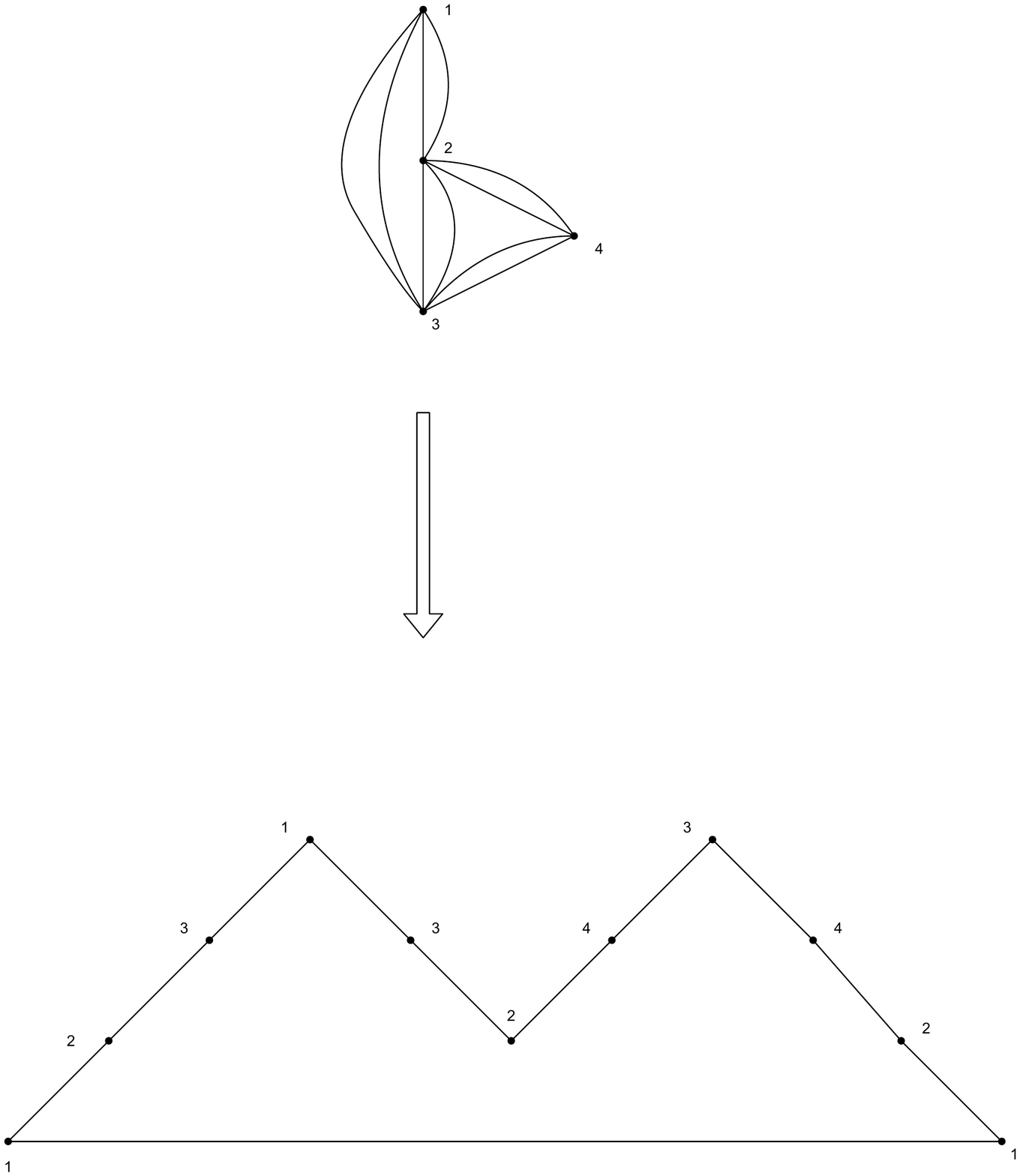}
        ~ %add desired spacing between images, e. g. ~, \quad, \qquad etc.
          %(or a blank line to force the subfigure onto a new line)
      \end{center}   
 \end{figure} 
\noindent
One might note that the definition of open and closed instants are properties of word rather than the path itself. For example consider the word $(1,2,3,1,2,4,3,4,2,3,1)$. It is another word having the same path as $(1,2,3,1,3,2,4,3,4,2,1)$. However here the vertex $1$ and $3$ both appear as open instants.

\subsection{Skeleton words}
One of the main ingredient of the paper is to keep track of the jumps of  non-well behaved words in a systematic way. In order to do that we introduce the concept of the skeleton word.

\noindent
Although the concept of skeleton word came naturally to the author, formally constructing it is some what tricky. We construct this in two steps. In the first step we consider the special case when the word is such that every edge is repeated at most twice. Once this is done we generalize the construction for general words. 

\noindent
\textbf{Construction of skeleton words when every edge is repeated at most twice:}
In this case the skeleton word is formed in the following way: We start constructing the Dyck path according to the word exploration until we encounter first type $j \ge 2$ open instant. We consider all the unmatched edges until this instant. Let there be $r$ many open edges until this instant. We call these edges $$(i_{0}=\alpha_{1},\beta_{1}),(\alpha_{2},\beta_{2}),\ldots, (\alpha_{r},\beta_{r}).$$ We at first prove that for any $i$ $\alpha_{i}=\beta_{i-1}$. Firstly observe that the level of $\alpha_{i}$ can not be below the level of $\beta_{i-1}$, this will contradict the fact that $(\alpha_{i-1},\beta_{i-1})$ is open. Now suppose $\alpha_{i}$ is at a level strictly higher that $\beta_{i-1}$. This also can't happen as the path until the first type $2$ open vertex is continuous hence in order to arrive at the level of $\alpha_{i}$ the path has to cross the intermediate levels and there is at least one open edge incident to each of these levels before reaching the level of $\alpha_{i}$. This contradicts to the fact that $i=(i-1)+1$. Now observe that the path restricted in between any $\beta_{i-1}$ and $\alpha_{i}$ is a Dyck path. Here is a hypothetical illustration of the path until this instant. 
\begin{figure}[H]
        \begin{center}
                \includegraphics[width=0.5\textwidth]{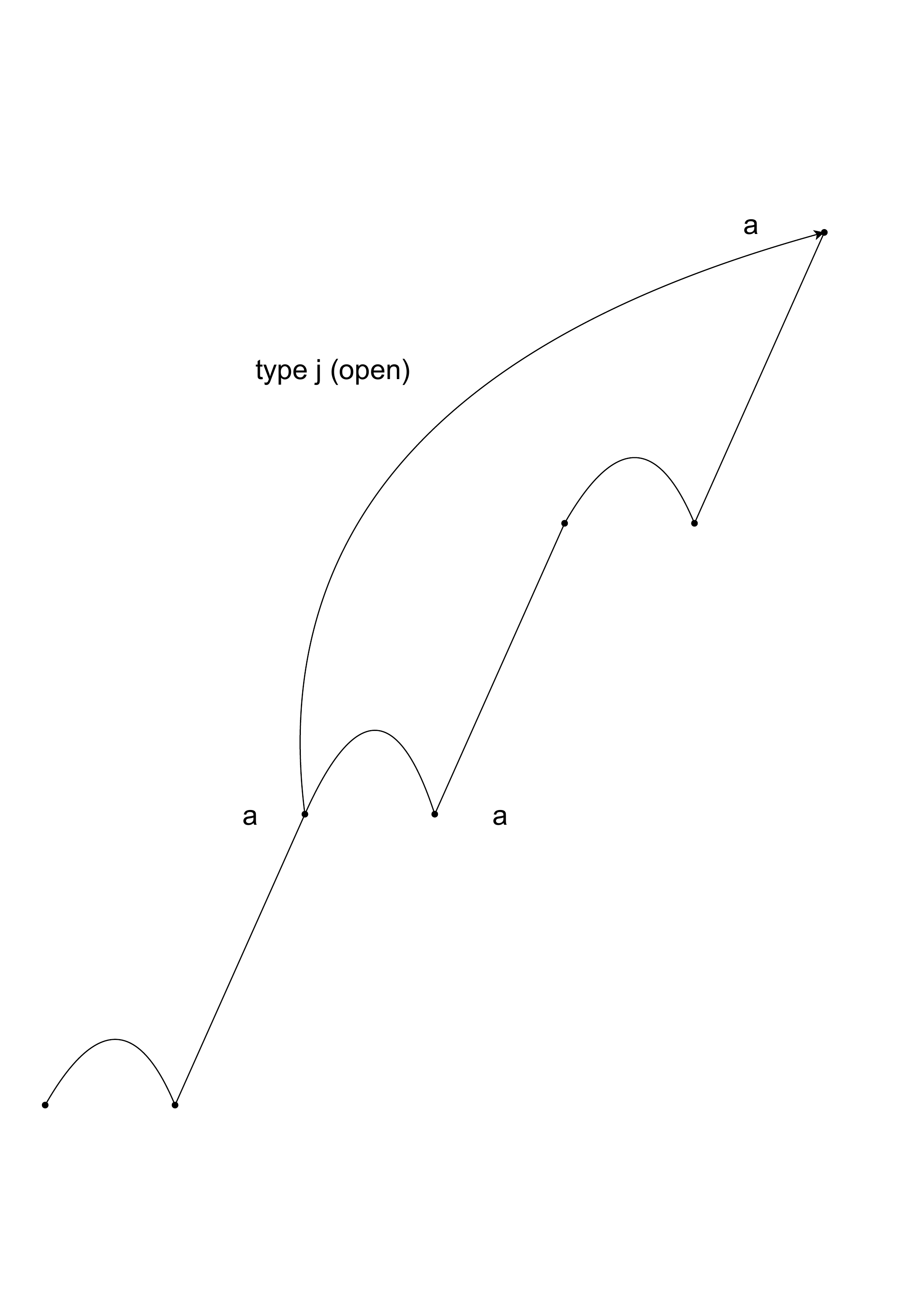}
        ~ %add desired spacing between images, e. g. ~, \quad, \qquad etc.
          %(or a blank line to force the subfigure onto a new line)
      \end{center}   
 \end{figure} 
Here the semicircles represent generic Dyck paths. We include the instants corresponding to the open edges in the skeleton word. Now the word can cover another Dyck path before possibly jumping. After this step the word closes a number of edges in the existing skeleton word before creating a new edge which remains open until the second open instant. We include all the instants corresponding to these edges in the skeleton word and continue the process until the full word is explored. Here is a generic figure to explain the skeleton word of a well behaved word.
\begin{figure}[H]
        \begin{center}
                \includegraphics[width=0.5\textwidth]{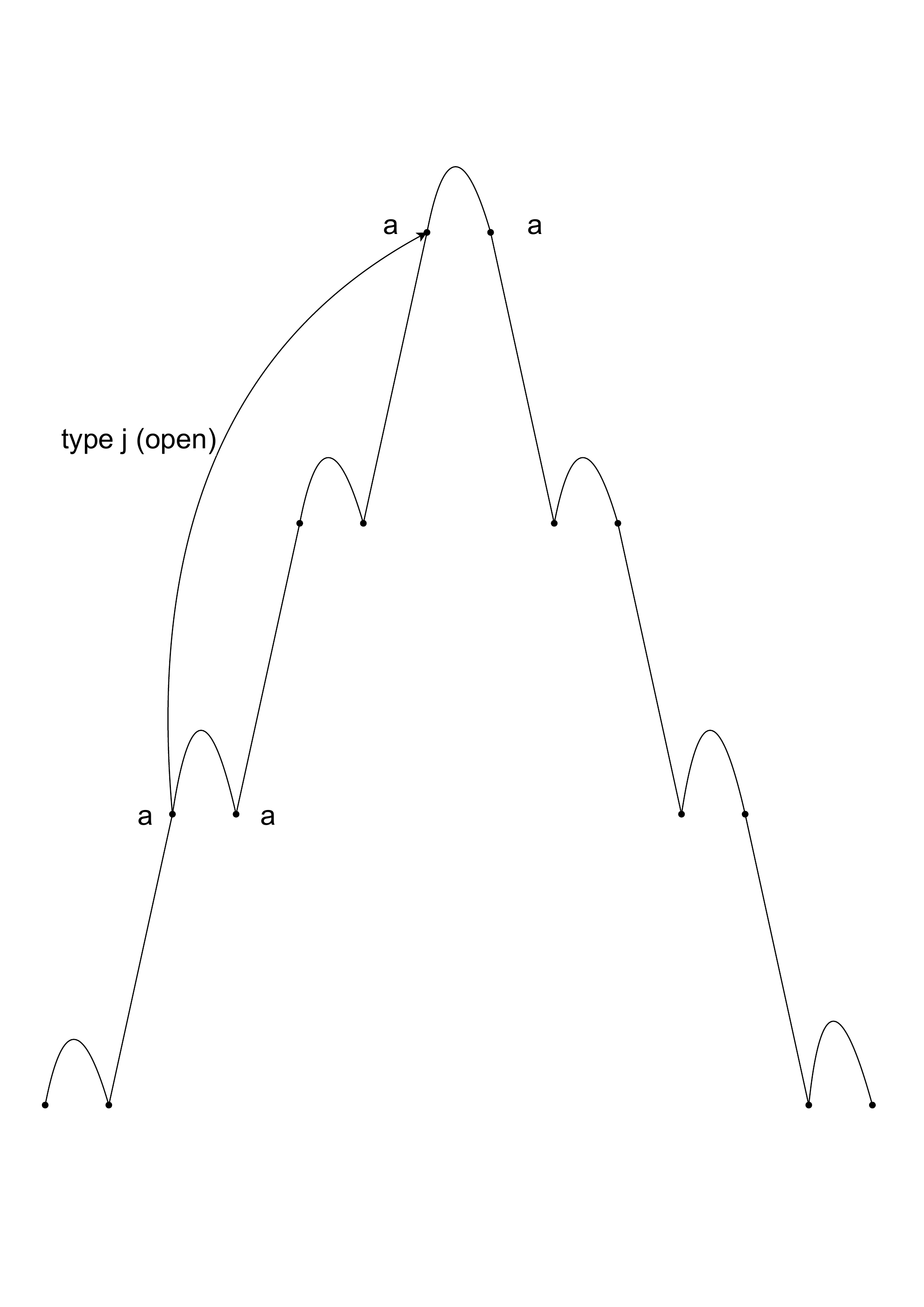}
        ~ %add desired spacing between images, e. g. ~, \quad, \qquad etc.
          %(or a blank line to force the subfigure onto a new line)
      \end{center}   
 \end{figure} 
\textbf{Construction of skeleton words for general words:}
We now give the construction of skeleton words for general words. The crucial fact here we need is, the edge set of the skeleton word and the edge set of the remaining word must be disjoint. First of all, given the word $w$ we follow the procedure just described. We call this pre-skeleton word at step $1$. However in this case there might be edges in the pre-skeleton word at step $1$ which appear in the remaining word. Observe that all the appearances of such edges apart from the skeleton word happens in some Dyck path where the instant corresponding to the starting point of the Dyck path is in the skeleton word. One might look at the following figure for clarification.
\begin{figure}[H]
        \begin{center}
                \includegraphics[width=0.5\textwidth]{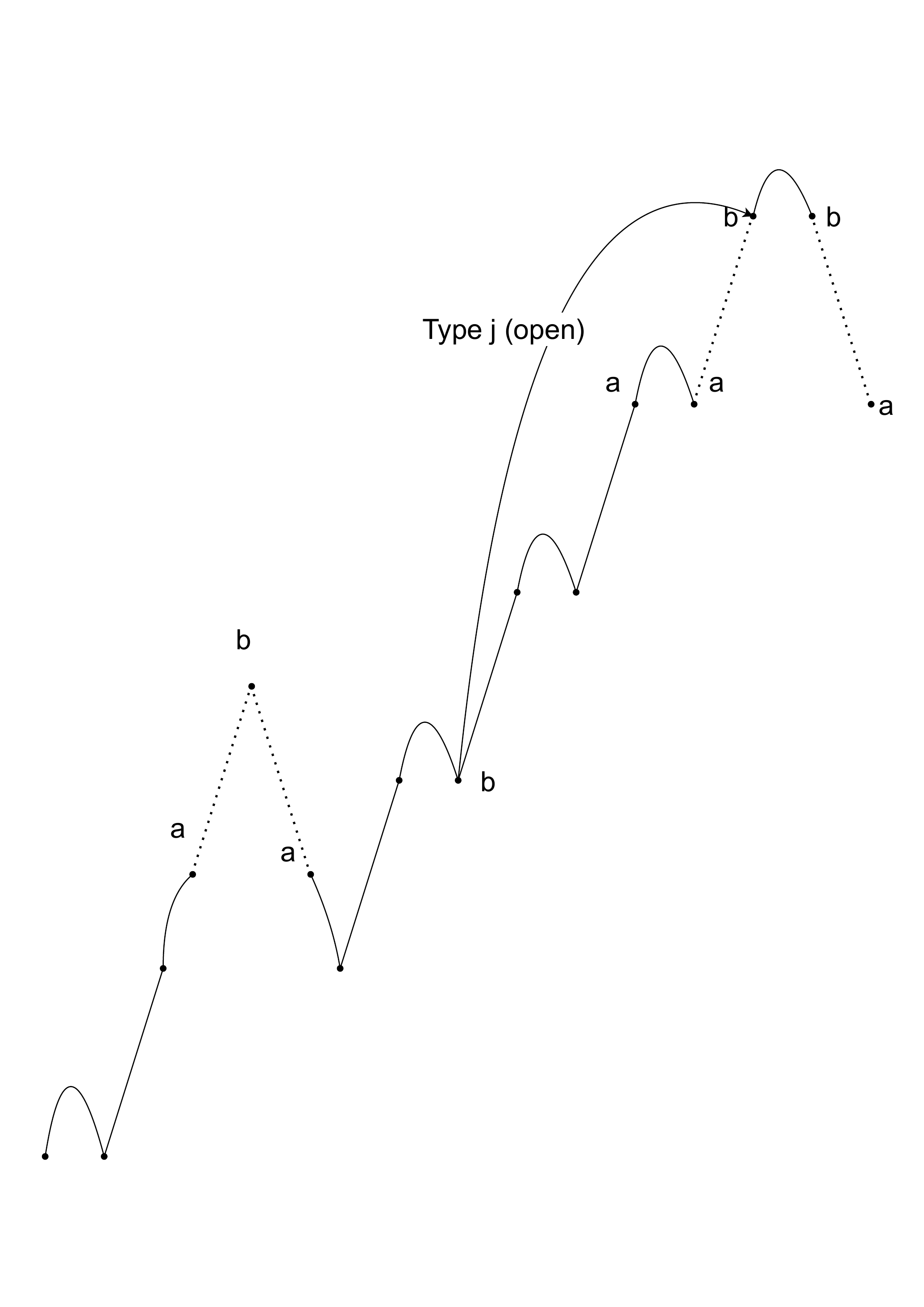}
        ~ %add desired spacing between images, e. g. ~, \quad, \qquad etc.
          %(or a blank line to force the subfigure onto a new line)
      \end{center}   
 \end{figure} 
 Here the dotted lines denote same edge.
 In step $2$ we consider the mentioned edges in the Dyck paths and the instants corresponding to the starting points of the Dyck paths. We look at the exploration restricted to the Dyck path. We include all the instants corresponding to edges open until we reach the edge of consideration, in the skeleton word and the instants which close the edges. As a matter of fact we increase the skeleton word. However we again might include edges  which appear in the remaining word. We now go to step $3$ with the same strategy. We continue until we have the edges in the skeleton word and remaining word disjoint or we exhaust all the edges in the word.
%\begin{definition}
%We start with a word $w$ and use Algorithm \ref{alg:encode} to get a walk $\mathcal{P}(w)$ which is possibly disconnected. Now we collect all the edges such that there exists at least one type $k>1$ open instant at which the edge is open.  We call this collection pre-skeleton edges. The skeleton word $S(w)$ is the minimal word which contains all the traversals of the edges just mentioned.
%\end{definition}
%As an example we again consider the word $(1,2,3,1,3,2,4,3,4,2,1)$. As mentioned earlier, the type $2$ instant for vertex $1$ is open. However the type $2$ instant for vertex $3$ is closed. Hence the skeleton word is $(1,2,3,1,3,2,1)$. The following figure gives an exposition of $S(w)$.
%\begin{figure}[H]
%        \begin{center}
%                \includegraphics[width=1.5\textwidth]{4prnew.pdf}
%        ~ %add desired spacing between images, e. g. ~, \quad, \qquad etc.
%          %(or a blank line to force the subfigure onto a new line)
%      \end{center}   
% \end{figure}  
%  
%One crucial observation is that given a word $w$, $S(w)$ is also a closed word. Although the path corresponding to $S(w)$ in $\mathcal{P}(w)$ might not be continuous. 
%The main scope of the rest of this section is to understand the structure of $S(w)$ for a word $w$. How $w$ is related to $S(w)$.
In order to count all the words we at first fix the skeleton word and enumerate the number of words with this skeleton and finally sum over all the skeleton words.\\

\noindent 
One fact about this representation is, the words corresponding to the intermediate Dyck paths will have empty skeleton words. This introduces an additional constraint. However for the purpose of calculation we shall ignore this constraint as we are only concerned about the upper bound.  
\subsection{Calculation of number of Skeleton words}\label{subsec:skeletonwordcount}
Now we give an upper bound to the number of skeleton words. Observe that firstly there is only one choice until the word hits the first open type $2$ instant. Up to this instant the walk increases monotonically. Now once the word hits the first type $2$ open instant, we have to fix the location of the point where this type $2$ instant appear for the first time. Once that is also fixed, in the next step to continue the word there is at most $3$ choices. It can close the immediately traversed edge or it can close one of the remaining maximum two open edges incident to the instant of the first appearance of the type $2$ instant. Now the word closes some edges for a few steps. Then the walk goes up for some steps until we hit the second type $k>1$ vertex in the skeleton word. We fix the number of upward steps, this fix the location of the second type $k>1$ instant in the skeleton word. Now we have to fix the instant where this vertex appeared as type $k-1\ge 1$ instant for the first time. Now the walk again decreases for some steps. Then it goes up to hit the next type $k>1$  instant and so on. Let $p_{i}$ denote the length of the $i$ th upward chunk before appearance of $i$ th type $k>1$ vertex after the $(i-1)$ th downward chunk, $q_{i}$ denotes the length of the $i$ downward chunk. Let $r_{i}$ denote the positions of the first appearances of the type $k\ge 2$ vertices in the pre-skeleton edges. 
Firstly observe that if a vertex does not appear as a type $k>1$ instant, then there is only one way to continue the skeleton word. It goes one step up or one step down depending upon the position of that instant as an upward or downward chunk respectively. We at first fix the values of $\{ p_{i} \}_{i=1}^{N}, \{ q_{i} \}_{i=1}^{N-1}$ and $r_{i}$'s and find the number of skeleton words with these parameters. Note that instead of $N$ the index of $q_{i}$ runs up to $N-1$. This due to the fact that once all $\{p_{i}\}_{i=1}^{N}, \{ q \}_{i=1}^{N-1}$ and $\{ r_{i} \}$ are fixed then $q_{N}$ has to be fixed. This is due to the fact that the skeleton word is closed. So in the last downward chunk one goes on closing the edges until it hits the starting vertex.

We  at first come to the power of $n$ for this skeleton word. Observe that in the $i$ upward chunk, there are $p_{i+1}$ many vertices. Now every upward chunk ends with a type $k>1$ vertex and also the starting vertex of each upward chunk apart from the first one is also fixed. So this gives us at most $n^{\sum_{i=1}^{N}(p_{i}-1)+1}$ choices.

 Now we look at a vertex (say $i$) which appear as a type $j\ge 2$ instant. In this case there will be $j$ different upward instants where this vertex appears. Let among them there are $\alpha_{i,j}$ instants with no downward chunk. Then at these instants there is only one choice to continue the skeleton word i.e. to go upward. However these instants increase the number of open directions incident to the vertex. %The following figure gives an illustration of this.%(****** insert figure******). 
 Now each one of such instants will increase the number of open directions by $2$. As a consequence when we arrive at a type $l\le j$ instant corresponding to the vertex $i$, the maximum number of choices to continue the skeleton word is bounded by $2\alpha_{i,j}+1$. Now notice that the instants with downward chunks do not increase the number of open direction incident to a vertex. Since they introduce and close an open direction at that instant. Now along with the upward instants there might be downward instants where the walk arrives at the vertex $i$. At these instants there might be multiple choice for the next step. However observe that whenever the walk arrive at the vertex $i$ going downward we either do not have any choice or the choices might be from the other type $k\ge 2$ vertices. Now while leaving vertex $i$ there is multiple number of choices. Whenever the walk arrive and leave at vertex $i$ while going downward, it closes two open directions. As a consequence, there will be at most $\alpha_{i,j}$ of these leavings will each have at most $2\alpha_{i,j}$ choices. 

Let there be $N$ type $k>1$ instants. Among them there are $N_{j}$ instants each of type $j$. Giving $N =\sum_{j\ge 2}(j-1)N_{j}$. If we fix the values of $\{ p_{i} \}_{i=1}^{N}, \{ q_{i} \}_{i=1}^{N-1}$ and $\{r_{i}\}_{i=1}^{\sum_{j}N_{j}}$, the number of skeleton words are bounded by  

\begin{equation}\label{countskeleton}
\begin{split}
 &\prod_{j} \prod_{x_{i,j}=1}^{N_{j}}(2\alpha_{i,j}+1)^{j-1-\alpha_{i,j}}(2\alpha_{i,j})^{\alpha_{i,j}} n^{\sum_{i=1}^{N}p_{i}-N +1}\\
 & \le \prod_{j} (3j)^{(j-1)N_{j}} n^{\sum_{i=1}^{N}p_{i}-N +1}.
 \end{split}
\end{equation}  
choices. 

%Observe that in the bound \eqref{countskeleton} we have possibly counted many cases which do not correspond to a valid word. We call such cases infeasible words. Not all infeasible word can be described. However we describe a special class of infeasible words which will be useful at a later scenario. We here start with a feasible word and look at the skeleton walk. This skeleton walk corresponds to a Dyck path with some edges removed from it. Now we construct an infeasible in the following way. We start creating the word from the starting point and we follow the word until we arrive at the first instant where we have multiple choice while closing an edge. Now we take any of the choice and continue closing edges until we reach the next instant where face multiple choices while closing an edge and continue. Here we might close some edge which did not arrive yet. This will create the infeasiblities.     
%Observe that traversals in these fashion may not give the complete word since we can run out of choices to continue before all edges are traversed. In this case we start with the left most edge which is not traversed and there is no other edge which comes immediately before that edge if we include that edge in any word and continue. %Although these words are not feasible but we have counted them. Now $r_{e,e_{1}}! r_{e,e_{2}}! \ge \frac{1}{2}^{\left[\frac{r_{e}}{2}\right]} \left[\frac{r_{e}}{2}\right]!$. We get the actual feasible word count is less than $2^{N} \prod_{e} \frac{1}{\left[\frac{r_{e}}{2}\right]!}$ times the count we introduced in \eqref{countskeleton}.  

 We shall see later that these $p_{i}$, $q_{i}$ and $r_{i}$'s are typically of order $n^{\frac{1}{3}}$ when the length of the word is of order $n^{\frac{2}{3}}$, all the type $k>1$ open instants are actually type $2$ and $N$ is finite with high probability. In this case \eqref{countskeleton} reduces to 
\begin{equation}\label{countskeletontype2}
 3^{N} n^{\sum_{i=1}^{N}p_{i}- N+1}
\end{equation}
\section{Proof of Theorem \ref{thm:traceconvergence}}\label{sec:proof}
%We at first prove a lemma which is needed for proving the universality. 

We now fix a number of Dyck paths along with their starting points and calculate the number of words corresponding to the Dyck paths. This calculation is quite similar to \citet{peche2009universality}.
\begin{proposition}\label{lem:countingstrategy}
Given any number $m$, we fix $m$ many Dyck paths $\mathcal{P}_{1},\ldots, \mathcal{P}_{m}$ of length $2k_{1},\ldots, 2k_{m}$ respectively. We also fix the initial points of the Dyck paths. Then the number of words obeying the Dyck path exploration is asymptotically of the order $n^{\sum_{i=1}^{m}}k_{i}$.
\end{proposition}
\begin{proof}
We at first start with the path $\mathcal{P}_{1}$ and fix its vertices then move to $\mathcal{P}_{2}$ and so on.

\noindent
Firstly observe that there can be common vertices among the Dyck paths $\mathcal{P}_{1},\ldots, \mathcal{P}_{m}$. We also have to keep track of the instants of type $k>1$. In our calculation we shall keep track of the vertices appeared in all the previous Dyck paths. In particular, while calculating the words corresponding to $\mathcal{P}_{l}$ we shall keep track of the vertices appeared in $\mathcal{P}_{1},\ldots, \mathcal{P}_{l-1}$ and define the type $k>1$ vertices whenever a vertex appeared as the right end point of an upward edge for the $k$ th time in the joint path tuple $\mathcal{P}_{1},\ldots, \mathcal{P}_{m}$. 

\noindent 
Let $\{\Gamma_{j}\}_{j=1}^{\infty}$ be the number of vertices of type $j$ in the joint path tuple $\mathcal{P}_{1},\ldots, \mathcal{P}_{m}$. We at first start with the instants of type $2$. There are a total of $\Gamma_{2}$ of them. Firstly we fix their locations. As there are a total of $\sum_{i=1}^{m} k_{i}$ many up ward instants, the type $2$ instants will be among them. Hence the locations are given by $\Gamma_{2}$ positions $j_{1}<j_{2}<\ldots< j_{\Gamma_{2}}$ where $j_{i}$'s are within $\{1,2,\ldots, \sum_{i=1}^{m} k_{i}\}$. Given the locations $j_{1}<\ldots<j_{\Gamma_{2}}$ we are to choose their values. For the location $j_{1}$ there are at most $(j_{1}-1)$ choices for the value of the first type $2$ instant. Similarly for the locations $j_{i}$ there are at most $(j_{i}-i)$ choices for the value of the $i$ th type $2$ instant. Hence the total number of choices for the type $2$ instants are given by 
\begin{equation}\label{eq:type2cont}
\begin{split}
&\sum_{1\le j_{1}< \ldots < j_{\Gamma_{2}}\le \sum_{i=1}^{m} k_{i}} \prod_{i=1}^{\Gamma_{2}} (j_{i}-i) \\
& \le \sum_{1\le j_{1}< \ldots < j_{\Gamma_{2}}\le \sum_{i=1}^{m} k_{i}} \prod_{i=1}^{\Gamma_{2}} j_{i}
 = \frac{1}{\Gamma_{2}!}\sum_{1\le j_{1} \neq j_{2}\neq \ldots \neq j_{\Gamma_{2}}\le \sum_{i=1}^{m} k_{i}} \prod_{i=1}^{\Gamma_{2}} j_{i}\\
 &\le \frac{1}{\Gamma_{2}!} \sum_{i \le j_{1},\ldots j_{\Gamma_{2}}\le \sum_{i=1}^{m}k_{i}} \prod_{i=1}^{\Gamma_{2}} j_{i}= \frac{1}{\Gamma_{2}!} \left(\frac{\left(\sum_{i=1}^{m}k_{i}\right)\left( \sum_{i=1}^{m} k_{i} -1\right)}{2}\right)^{\Gamma_{2}}\le \frac{1}{\Gamma_{2}!}\left( \frac{\left(\sum_{i=1}^{m}k_{i}\right)^{2}}{2} \right)^{\Gamma_{2}}.
\end{split}
\end{equation} 
A similar calculation proves that the number of choices for the type $j$ instant is bounded by $\frac{1}{\Gamma_{j}!(j-1)!^{\Gamma_{j}}}\left( \sum_{i=1}^{m} k_{i} \right)^{j\Gamma_{j}}$. As a consequence, we get the total number of words corresponding to paths $\mathcal{P}_{1},\ldots, \mathcal{P}_{m}$ are bounded by 
\begin{equation}\label{eq:boundwordcountdyck}
\begin{split}
\prod_{l=1}^{\sum_{i}k_{i}- \sum_{j\ge 2}(j-1)\Gamma_{j}} (n-l+1)\sum_{\Gamma_{2}}\frac{1}{\Gamma_{2}!}\left( \frac{\left(\sum_{i=1}^{m}k_{i}\right)^{2}}{2} \right)^{\Gamma_{2}}\sum_{\Gamma_{3},\Gamma_{4},\ldots} \prod_{j\ge 3}\frac{1}{(j-1)!^{\Gamma_{j}}\Gamma_{j}!}\left( \sum_{i=1}^{m} k_{i} \right)^{j\Gamma_{j}}
\end{split}
\end{equation} 
Here we have used the fact that the initial choices of the paths are fixed. This corresponds to the power $\sum_{i=1}^{m}k_{i}$ instead of $\sum_{i=1}^{m} k_{i}+m$. Now we analyse the term $\prod_{l=1}^{\sum_{i}k_{i}- \sum_{j\ge 2}(j-1)\Gamma_{j}} (n-l+1)$ in some details. Observe that 
\begin{equation}
\begin{split}
&\prod_{l=1}^{\sum_{i}k_{i}- \sum_{j\ge 2}(j-1)\Gamma_{j}} (n-l+1)\\
&= n^{\sum_{i}k_{i}- \sum_{j\ge 2}(j-1)\Gamma_{j}}\prod_{l=1}^{\sum_{i}k_{i}- \sum_{j\ge 2}(j-1)\Gamma_{j}} \left(\frac{n-l+1}{n}\right)\\
&= n^{\sum_{i}k_{i}- \sum_{j\ge 2}(j-1)\Gamma_{j}}\prod_{l=1}^{\sum_{i}k_{i}-\sum_{j\ge 2}(j-1)\Gamma_{j}}\left( 1- \frac{l-1}{n} \right).
\end{split}
\end{equation}  
So the whole thing boils down to analyse $\sum_{l=1}^{\sum_{i}k_{i}-\sum_{j\ge 2}(j-1)\Gamma_{j}}\log\left(1- \frac{l-1}{n}\right)$. Now using the fact that $\frac{x}{(1-x)}\le \log(1+x) \le x$
\begin{equation}\label{eq:expansionlog}
\begin{split}
&\sum_{l=1}^{\sum_{i}k_{i}-\sum_{j \ge 2}(j-1)\Gamma_{j}}-\frac{\frac{l-1}{n}}{1+ \frac{l-1}{n}} \le \sum_{l=1}^{\sum_{i}k_{i}-\sum_{j \ge 2}(j-1)\Gamma_{j}} \log\left( 1- \frac{l-1}{n} \right)\le \sum_{l=1}^{\sum_{i}k_{i}-\sum_{j \ge 2}(j-1)\Gamma_{j}} -\frac{l-1}{n}\\
\Rightarrow & \sum_{l=1}^{\sum_{i}k_{i}-\sum_{j \ge 2}(j-1)\Gamma_{j}} \log\left( 1- \frac{l-1}{n} \right) = - \left( 1+O\left(\frac{\sum_{i}k_{i}}{n}\right) \right)\left( \sum_{l=1}^{\sum_{i}k_{i}-\sum_{j \ge 2}(j-1)\Gamma_{j}}\frac{l-1}{n} \right)\\
& \le -\frac{1}{2n}\left(1+O\left( \frac{\sum_{i}k_{i}}{n} \right)  \right)\left( \sum_{i}k_{i}-\sum_{j\ge 2} (j-1)\Gamma_{j} \right)^2\\
&\le -\frac{1}{2n}\left(1+O\left( \frac{\sum_{i}k_{i}}{n} \right)  \right) \left( \left(\sum_{i}k_{i}\right)^2 -2\left( \sum_{i}k_{i} \right)\left(\sum_{j\ge 2}(j-1)\Gamma_{j}\right) \right)\\
&= - \frac{\left(\sum_{i}k_{i}\right)^2}{2n}+O(1) + \left( 1+O\left( \frac{\sum_{i}k_{i}}{n} \right)\right)\frac{\left( \sum_{i}k_{i} \right)\left(\sum_{j\ge 2}(j-1)\Gamma_{j}\right) }{n}\\
&\le - \frac{\left(\sum_{i}k_{i}\right)^2}{2n}+O(1) + 2\frac{\left( \sum_{i}k_{i} \right)\left(\sum_{j\ge 2}(j-1)\Gamma_{j}\right) }{n}
\end{split}
\end{equation}
Here in the last line of \eqref{eq:expansionlog} we have used the fact that $\sum_{i=1}^{m}k_{i}=O(n^{\frac{2}{3}})$. Now exponentiating \eqref{eq:expansionlog} and putting it in \eqref{eq:boundwordcountdyck} we have the count of the words are bounded by :
\begin{equation}\label{eq:boundwordcountdyckII}
\begin{split}
&O(1)n^{\sum_{i}k_{i}- \sum_{j\ge 2}(j-1)\Gamma_{j}}\exp\left(  - \frac{\left(\sum_{i}k_{i}\right)^2}{2n} \right) \exp\left( 2\frac{\left( \sum_{i}k_{i} \right)\left(\sum_{j\ge 2}(j-1)\Gamma_{j}\right) }{n} \right)\\
&~~~~~~~~~~~~~~~~~\sum_{\Gamma_{2}}\frac{1}{\Gamma_{2}!}\left( \frac{\left(\sum_{i=1}^{m}k_{i}\right)^{2}}{2} \right)^{\Gamma_{2}}\sum_{\Gamma_{3},\Gamma_{4},\ldots} \prod_{j\ge 3}\frac{1}{(j-1)!^{\Gamma_{j}}\Gamma_{j}!}\left( \sum_{i=1}^{m} k_{i} \right)^{j\Gamma_{j}}\\
& \asymp n^{\sum_{i}k_{i}} \exp\left(  - \frac{\left(\sum_{i}k_{i}\right)^2}{2n} \right) \sum_{\Gamma_{3},\Gamma_{4},\ldots} \prod_{j\ge 3} \frac{1}{(j-1)!^{\Gamma_{j}}\Gamma_{j}!}\left(\frac{\left(\sum_{i}k_{i}\right)^{j}\exp\left((j-1) \frac{2\sum_{i}k_{i}}{n} \right)}{n^{j-1}}\right)^{\Gamma_{j}}\\
&~~~~~~~~~~~~~~~~~~~~ \sum_{\Gamma_{2}} \frac{1}{\Gamma_{2}!}\left( \frac{\left(\sum_{i}k_{i}\right)^2\exp\left( \frac{2\sum_i{k_{i}}}{n} \right)}{2n} \right)^{\Gamma_{2}}\\
&\asymp  n^{\sum_{i}k_{i}} \exp\left(  - \frac{\left(\sum_{i}k_{i}\right)^2}{2n} \right)\exp \left(  \frac{\left(\sum_{i}k_{i}\right)^{2}\exp\left( \frac{2\sum_i{k_{i}}}{n} \right)}{2n}\right)\exp\left( \sum_{j\ge 3} \frac{\left(\sum_{i}k_{i}\right)^{j}\exp\left((j-1) \frac{2\sum_{i}k_{i}}{n} \right)}{(j-1)!n^{j-1}} \right)
\end{split}
\end{equation}
It can be showed with some elementary calculation that the term 
\begin{equation}
\exp\left(  - \frac{\left(\sum_{i}k_{i}\right)^2}{2n} \right)\exp \left(  \frac{\left(\sum_{i}k_{i}\right)^{2}\exp\left( \frac{2\sum_i{k_{i}}}{n} \right)}{2n}\right)\exp\left( \sum_{j\ge 3} \frac{\left(\sum_{i}k_{i}\right)^{j}\exp\left((j-1) \frac{2\sum_{i}k_{i}}{n} \right)}{(j-1)!n^{j-1}} \right) \asymp 1
\end{equation}
whenever $\sum_{i}k_{i}\asymp n^{\frac{2}{3}}$. As a consequence, we get the required bound on the word count. This completes the proof.
\end{proof}
We now state a proposition which tells that among all the words, the words with every edge traversed exactly twice give a bounded contribution to the trace.
\begin{proposition}\label{prop:2times}
Suppose we consider all the words of length $2k+1$ where every edge traversed exactly twice. We call this class of words $\mathcal{W}_{2k}$. Then the following is true whenever $k=[tn^{\frac{2}{3}}]$ for some $t\in (0,\infty)$:
\begin{equation}\label{eq:2timescont}
\frac{1}{n^{k}} \sum_{w \in \mathcal{W}_{2k}} \E[X_{w}]= O(1).
\end{equation} 
\end{proposition}
\begin{proof}
At first observe that as every edge traversed exactly twice $\E[X_{w}]=\left(\frac{1}{4}\right)^{k}=\frac{1}{2^{2k}}$. Hence \eqref{eq:2timescont} is equivalent to proving $\#\mathcal{W}_{2k}=O\left( 2^{2k} \right).$ This is what we prove here. 

As we mentioned earlier, the key approach of this paper is to fix a skeleton word and then do the calculation of the number of words having that specific skeleton word and finally take the sum over all the skeleton words.

So at the beginning we fix a skeleton word. Let there be $N$ type $k\ge 2$ vertices. We at first fix the values of $\{p_{i}\}_{i=1}^{N}, \{ q_{i} \}_{i=1}^{N-1}, \{ r_{i} \}_{i=1}^{\sum_{j}N_{j}}$ where these parameters were defined in subsection \ref{subsec:skeletonwordcount}. %First observe that if there is a type $k>2$ vertex in the skeleton word, then the number of choices of $r_{i}$'s will be reduced. This is due to the fact that same vertex will be used to map different type $k>2$ instants of the skeleton word. We shall see later this will lead to a reduced count.
As we have all the edges traversed exactly twice, we have $\sum_{i=1}^{N}p_{i}= \sum_{i=1}^{N} q_{i}$. 
%However this does not fix any additional constraint. Simply because we have argued that we only need to fix $\{ q_{i} \}_{i=1}^{N-1}$ and fixing these choices $q_{N}$ is automatically fixed upto finitely many choices. It might be possible that we might get $\sum_{i=1}^{N}p_{i}\neq \sum_{i=1}^{N} q_{i}$ depending on the structure of the skeleton word but the following gives an example when we can take any $q_{i=1}^{N-1}$ such that $\sum_{i=1}^{j}p_{i}\ge \sum_{i=1}^{j}q_{i}$ for all $j$ and $\sum_{i=1}^{N} p_{i}= \sum_{i=1}^{N} q_{i}$. This is simply done by constructing the word without discontinuity. 

Now once the skeleton word is fixed, we spend $\sum_{i=1}^{N}\left(p_{i}+q_{i}\right)= 2\sum_{i=1}^{N}p_{i}=2m(say)$ many edges. In particular there will be $2m+1$ many Dyck paths adjacent to each of these edges of the skeleton word. We call them $\mathcal{P}_{1},\ldots, \mathcal{P}_{2m+1}$. Let their lengths be $2k_{1},\ldots, 2k_{2m+1}$ respectively. Then we have that 
\begin{equation}\label{eq:basicequality}
\sum_{i=1}^{2m+1} 2k_{i}+2m =2k.
\end{equation}
Since the vertices of the skeleton path are fixed, the end points of these Dyck paths are fixed.
By Proposition \ref{lem:countingstrategy}, we have the number of words corresponding to these Dyck paths are of the order $n^{\sum_{i=1}^{2m+1}k_{i}}$. On the other hand there are at most $n^{\sum_{i=1}^{N}p_{i}- N +1}= n^{m-N+1}$ many choices for the vertices in the skeleton word. So in particular, fixing the skeleton word and the Dyck paths adjacent to each edge in the skeleton word we get the number of words are bounded by $n^{m-N+1+ \sum_{i=1}^{2m+1}k_{i}}= n^{k-N+1}$. 

We now calculate the total number of choices of $\left( \mathcal{P}_{1},\ldots, \mathcal{P}_{2m+1}\right)$. By \eqref{def:mthcatalan} we have this is exactly equal to $\frac{2m+1}{2k+1}\binom{2k+1}{k+m+1}$. 

Now in the final step we take the sum over the choices of the skeleton words. At first we assume that all $\{p_{i}\}_{i=1}^{N}, \{ q_{i} \}_{i=1}^{N-1}$ and $\{ r_{i} \}_{i=1}^{\sum_{j}N_{j}}$ are positive. We also assume all type $k$ instants to be actually of type $2$. Hence $N= N_{2}= \sum_{j}N_{j}$. By subsection \ref{subsec:skeletonwordcount} this is bounded by 
\begin{equation}
3^{N} n^{\sum_{i}p_{i}-N+1}.
\end{equation}
Hence in this case we get the all possible words are bounded by 
 \begin{equation}\label{eq:wordcounttype2}
 \begin{split}
 \sum_{N}\sum_{p_{1},\ldots,p_{N}}\sum_{q_{i},\ldots,q_{N-1}}\sum_{r_{1},\ldots,r_{N}} 3^{N}  \frac{2m+1}{2k+1}\binom{2k+1}{k+m+1} n^{k-N+1}
 \end{split}
 \end{equation}
 Firstly observe that 
 \begin{equation}
 \begin{split}
 \binom{2k+1}{k+m+1} & = \frac{(2k)!}{(k-m)! (k+m)!} \frac{2k+1}{k+m+1} \le 2 \binom{2k}{k+m}
 \end{split}
 \end{equation}
 Hence we replace the $\binom{2k+1}{k+m+1}$ in \eqref{eq:wordcounttype2} by $2 \binom{2k}{k+m}.$ We now apply Stirling approximation to get that whenever $m\le k-1$
 \begin{equation}\label{eq:binomstirling}
 \begin{split}
 \binom{2k}{k+m}&\asymp (1+o(1))2^{2k+\frac{1}{2}}\frac{\sqrt{k}}{\sqrt{\pi(k^2-m^2)}}\exp\left(- \sum_{l\ge 2~|~ \text{even}}\frac{2m^{l}}{k^{l-1}l(l-1)}\right). 
 \end{split}
\end{equation}  
We shall apply \eqref{eq:binomstirling} in \eqref{eq:wordcounttype2}. Further we at first fix $N$ and sum over all the other indexes. Now our sum in \eqref{eq:wordcounttype2} is of the order of  
\begin{equation}\label{eq:wordcountmod}
\begin{split}
2^{2k+\frac{1}{2}}n^{k}\sum_{N}3^{N}\sum_{p_{1},\ldots,p_{N}}\sum_{q_{i},\ldots,q_{N-1}}\sum_{r_{1},\ldots,r_{N}} \frac{2m+1}{2k+1}\frac{\sqrt{k}}{\sqrt{\pi(k^2-m^2)}}\exp\left(- \sum_{l\ge 2~|~ \text{even}}\frac{2m^{l}}{k^{l-1}l(l-1)}\right)\left(\frac{1}{\sqrt{k}}\right)^{3N-3}
\end{split}
\end{equation}
The key idea is to represent the sum inside the summand $\sum_{p_{1},\ldots,p_{N}}\sum_{q_{i},\ldots,q_{N-1}}\sum_{r_{1},\ldots,r_{N}}$ as a Riemann sum of an integral. To do this we take the mesh size of $\frac{1}{\sqrt{k}}$ and write $x_{i}:=\frac{p_{i}}{\sqrt{k}}, X_{j}:=\sum_{i=1}^{j}x_{i}, y_{i}:= \frac{q_{i}}{\sqrt{k}}, Y_{j}:= \sum_{i}^{j}y_{i}$ and $z_{i}= \frac{r_{i}}{\sqrt{k}}$. Since the final function only depends on $m= \sum_{i=1}^{N}p_{i}$, we do a trick to change the co-ordinates from $x_{1},\ldots, x_{N}$ to $X_{1},\ldots, X_{N}$ and $y_{1},\ldots, y_{N-1}$ to $Y_{1},\ldots, Y_{N-1}$. This transformation is one to one and we have the additional constraint $X_{1}\le X_{2}\le \ldots \le X_{N} $ and similarly $Y_{1}\le Y_{2}\le \ldots \le Y_{N-1}.$ We also define $P_{1},\ldots,P_{N}$ and $Q_{1},\ldots,Q_{N}$ in the analogous way. Now we fix $m= \sum_{i=1}^{N}p_{i}$ and take the sum 
\begin{equation}
\begin{split}
&\sum_{p_{1}\ldots p_{N}} \sum_{q_{1}\ldots q_{N}} f(m)\\
&\le \sum_{P_{1}\le P_{2}\le \ldots P_{N}=m} \sum_{Q_{1}\le Q_{2}\le \ldots \le Q_{N-1}\le m} f(m)\\
&= \frac{\sqrt{k}^{2N-2}}{\sqrt{k}^{2N-2}}\sum_{P_{1}\le P_{2}\le \ldots P_{N-1}\le m} \sum_{Q_{1}\le Q_{2}\le \ldots \le Q_{N-1}\le m} f(m)\\
& \le f(m)\left(\sqrt{k}\right)^{2N-2} \int_{X_{1}\le X_{2}\le \ldots \le X_{N-1}\le \frac{m}{\sqrt{k}}} \int_{Y_{1}\le Y_{2}\le \ldots \le Y_{N-1}\le \frac{m}{\sqrt{k}}} 1\prod_{i=1}^{N-1}dX_{i} \prod_{i=1}^{N-1}dY_{i}\\
&= f(m) \left(\sqrt{k}\right)^{2N-2} \frac{1}{(N-1)!^2}\left(\frac{m}{\sqrt{k}}\right)^{2N-2} 
\end{split}
\end{equation}
Now we consider the sum over the indexes $r_{i}$'s. We know that for each $i$, $1 \le r_{i} \le m$. So 
\begin{equation}
\begin{split}
\sum_{r_{1}, \ldots r_{N}} 1 \le  m^{N}. 
\end{split}
\end{equation} 
Putting these in \eqref{eq:wordcountmod} we reduce our job to bound
\begin{equation}\label{eq:firstintegralrep}
\begin{split}
&2^{2k+\frac{1}{2}}n^{k}\sum_{N}3^{N} \sum_{N\le m \le k-1}\left(\sqrt{k}\right)^{3N-2} \frac{1}{(N-1)!^2}\left(\frac{m}{\sqrt{k}}\right)^{3N-2} \\
&~~~~~~~~\frac{2m+1}{2k+1}\frac{\sqrt{k}}{\sqrt{\pi(k^2-m^2)}}\exp\left(- \sum_{l\ge 2~|~ \text{even}}\frac{2m^{l}}{k^{l-1}l(l-1)}\right)\left(\frac{1}{\sqrt{k}}\right)^{3N-3}\\
& \le 2^{2k+\frac{1}{2}}n^{k}\sum_{N}3^{N}\frac{1}{(N-1)!^2}\sum_{N\le m \le k-1} {\sqrt{k}} \left(\frac{m}{\sqrt{k}}\right)^{3N-2}\frac{2m+1}{2k+1}\frac{\sqrt{k}}{\sqrt{\pi(k^2-m^2)}}\exp\left(-\frac{m^{2}}{k}\right)
\end{split}
\end{equation}
Now observe that the sum under the summand $\sum_{N\le m \le k-1}$ is a Riemann sum of the function
\begin{equation}\label{eq:secondintegralrep}
\int_{\frac{N}{\sqrt{k}}}^{\frac{k-1}{\sqrt{k}}} X_{N}^{3N-2}\left( 2X_{N} + \frac{1}{\sqrt{k}} \right) \frac{\sqrt{k}}{2\sqrt{k}+ \frac{1}{\sqrt{k}}}\frac{\sqrt{k}}{\sqrt{k-X_{N}^2}}\exp\left( -  X_{N}^{2} \right)dX_{N}
\end{equation} 
Note that the function in \eqref{eq:secondintegralrep} is not monotonically decreasing. However the function inside the integrand of \eqref{eq:secondintegralrep} can be uniformly dominated by a function of the form: 
\begin{equation}\label{eq:integraldom}
\zeta(X_{N}):=c_{1}\Gamma\left( \frac{3N+ \hat{\xi}}{2} \right)\exp\left( - X_{N}^2 \right)\mathbb{I}_{0\le X_{N}\le c}+ c_{2}X_{N}^{\xi+3N-1}\exp\left( - X_{N}^{2} \right)\mathbb{I}_{c\le X_{N}}
\end{equation} 
where $\xi$ and $\hat{\xi}$ are deterministic constants independent of $N$ and $k$ and the function $\zeta(X_{N})$ is monotonically decreasing.
As a consequence, 
\begin{equation}
\begin{split}
&\sum_{N\le m \le k-1} {\sqrt{k}} \left(\frac{m}{\sqrt{k}}\right)^{3N-2}\frac{2m+1}{2k+1}\frac{\sqrt{k}}{\sqrt{\pi(k^2-m^2)}}\exp\left(-\frac{m^{2}}{k}\right)\\
&\le \int_{\frac{N-1}{\sqrt{k}}}^{\frac{k-1}{\sqrt{k}}} \zeta(X_{N})dX_{N}\\
&\le c_{1}c_{1}'\Gamma\left( \frac{\hat{\xi}+3N}{2}\right)+  c_{2}\int_{0}^{\infty} X_{N}^{\xi+3N-1}\exp\left( - X_{N}^{2} \right)dX_{N}\\
&=c_{1}c_{1}'\Gamma\left( \frac{\hat{\xi}+3N}{2}\right)+ \frac{c_{2}}{2} \int_{0}^{\infty} z^{\frac{\xi + 3N}{2}-1}\exp(-z)dz\\
&= c_{1}c_{1}'\Gamma\left( \frac{\hat{\xi}+3N}{2}\right) + \frac{c_{2}}{2} \Gamma\left(\frac{\xi+3N}{2}\right) 
\end{split}
\end{equation} 
It is easy to see that 
\begin{equation}
\sum_{N =1}^{\infty} 3^{N} \left( c_{1}c_{1}'\Gamma\left( \frac{\hat{\xi}+3N}{2}\right) + \frac{c_{2}}{2} \Gamma\left( \frac{\xi+3N}{2} \right) \right)\frac{1}{(N-1)!^2}\asymp 1.
\end{equation}
Now we consider the boundary case $m=k$. When $m=k$ we have 
\begin{equation}\label{eq:boundary1}
\begin{split}
&\lim_{k \to \infty}\sum_{1\le N \le k}3^{N}\frac{1}{(N-1)!^2} \frac{2k+1}{2k+1}\binom{2k}{2k} \left( \sqrt{k} \right)^{3N-1}\frac{1}{2^{2k}}\\
&=\lim_{k \to \infty} \sum_{1 \le N \le k} 3^{N}\frac{1}{(N-1)!^2}\left( \sqrt{k} \right)^{3N-1}\frac{1}{2^{2k}}\\
&\asymp \lim_{k \to \infty}\sum_{1\le N \le k} N\exp\left(N\log 3- k\log 4 + \frac{3N-1}{2}\log k - 2\log(N)N- N  \right)\\
& \le \lim_{k \to \infty}\sum_{1\le N \le k} k \exp\left( N\log 3- k\log 4 + \frac{3N-1}{2}\log k - 2\log(N)N- N \right)
\end{split}
\end{equation} 
Maximizing the term in the exponential we get that $\log N = \frac{3}{4}\log k +c$. Hence among the terms inside the exponential $\exp(-k\log 4)$ dominates. Now if we take the sum it is $\text{Poly}(k)\frac{1}{2^{2k}} \to 0.$

\noindent 
Hence \eqref{eq:wordcounttype2} is asymptotically of the same order as $n^{k}2^{2k}$. 

Now we consider the other case when there is atleast one instant of type strictly greater than $2$ or there is atleast one instant where $q_{i}=0$. We prove that in these cases we get a negligible contribution. 

To start with we need a few notations. Let $N_{j}$ be the total number of vertices of type $j$ and for the $i$ th of type $j$ vertex let $\alpha_{i,j}$ be the number of instants with no downward chunk. It is easy to observe that $N= \sum_{j} (j-1) N_{j}$. As before we fix the values of $\{p_{i}\}_{i=1}^{N}, \{ q_{i} \}_{i=1}^{N-1}$ and $\{ r_{i} \}_{i=1}^{\sum_{j}N_{j}}$. Here we observe that some $q_{i}$'s will be $0$. From subsection \ref{subsec:skeletonwordcount} we have given $\{p_{i}\}_{i=1}^{N},\{ q_{i} \}_{i=1}^{N-1}$ and $\{ r_{i} \}_{i=1}^{\sum_{j}N_{j}}$ and $\{ \alpha_{i,j} \}$'s, the number of skeleton words of this kind is bounded by:
\begin{equation}\label{eq:genskeletonwordcount}
\begin{split}
& n^{\sum_{i=1}^{N-1}p_{i} -\sum_{j}(j-1)N_{j} +1}\prod_{j} (3j)^{(j-1)N_{j}}.
\end{split}
\end{equation}  
Now arguing as before we get the total number of all possible words are bounded by 
\begin{equation}\label{eq:totwordcountgen}
\begin{split}
&\sum_{N_{j}}\sum_{\text{positions of the } \alpha_{i,j}'s}\sum_{\alpha_{i,j}}\sum_{p_{1},\ldots,p_{N}}\sum_{q_{i},\ldots,q_{N-1}}\sum_{r_{i}'s}  \frac{2m+1}{2k+1}\binom{2k+1}{k+m+1} n^{k-N+1}\prod_{j}\left(3j\right)^{(j-1)N_{j}}\\
&\le \sum_{N_{j}} \sum_{\alpha_{i,j}}\sum_{p_{i},\ldots,p_{N}}\sum_{q_{i} \neq 0}\sum_{r_{i}'s} \binom{N}{\sum_{j}\sum_{x_{i,j}}^{N_{j}}\alpha_{i,j}}\frac{2m+1}{2k+1}\binom{2k+1}{k+m+1} n^{k-N+1}\prod_{j}(3j)^{(j-1)N_{j}}\\
& \le \sum_{N_{j}}\sum_{\alpha_{i,j}}\sum_{p_{i},\ldots,p_{N}}\sum_{q_{i} \neq 0}\sum_{r_{i}'s} 2^{N}\frac{2m+1}{2k+1}\binom{2k+1}{k+m+1} n^{k-N+1}\prod_{j}(3j)^{(j-1)N_{j}}.
\end{split}
\end{equation}
The rest of the argument is dedicated to bound \eqref{eq:totwordcountgen}. To begin with the number of $q_{i}$'s which are non zero are given by 
\begin{equation}\label{eq:noqnonzero}
N-M:= N-1 - \sum_{j}\sum_{i=1}^{N_{j}} \alpha_{i,j} +1= N - \sum_{j}\sum_{i=1}^{N_{j}} \alpha_{i,j}
\end{equation} 
Now we need to sum 
\begin{equation}\label{eq:sumpsumr}
\sum_{p_{1}\ldots,p_{N-1}~|~ \sum_{i}p_{i}=m} \sum_{r_{i}'s}1.
\end{equation}
This argument is somewhat similar to the calculation we did for type $j>2$ vertices. However we need to be more cautious here due to the additional factors $\prod_{j}(3j)^{(j-1)N_{j}}$. In particular for type $j$ instants, there will be $N_{j}$ many instants repeated $j-1$ times each. Now our task is to bound how many ways these positions can be arranged. Observe that we are dealing with $N_{j}$ positions repeated $(j-1)$ times where $j$ varies. In particular the number of ways to arrange these positions are 
\begin{equation}
\frac{N!}{\prod_{j} ((j-1)!)^{N_{j}} }.
\end{equation}  
However note that here we made some over counting. This is due to the fact when we take the choices of $n$ we already account for the ordering of the vertices. In particular $jN_{j}$ vertices of type $j$ will be counted $N_{j}!$ many times. So we get the exact count is 
\begin{equation}
\frac{N!}{\prod_{j} ((j-1)!)^{N_{j}} \prod_{j} N_{j}!}.
\end{equation}  
Now the total number of choices of \eqref{eq:sumpsumr} is bounded by 
\begin{equation}\label{eq:sumpsumrsim}
\begin{split}
&\sum_{p_{1}\ldots,p_{N-1}~|~ \sum_{i}p_{i}=m} \sum_{r_{i}'s}1\\
&\le \frac{1}{(N-1)!} m^{N-1} \frac{N!}{\prod_{j} ((j-1)!)^{N_{j}} \prod_{j} N_{j}!}m^{\sum_{j}N_{j}}\\
& =  \frac{N}{\prod_{j} ((j-1)!)^{N_{j}} \prod_{j} N_{j}!} m^{\sum_{j} jN_{j}-1}
\end{split}
\end{equation} 
%By the argument exactly same as Proposition \ref{lem:countingstrategy} we have this sum is bounded by 
%\begin{equation}\label{eq:typegenboundskeleton}
%\frac{1}{m}\prod_{j} \frac{1}{N_{j}!}m^{jn_{j}}
%\end{equation}
On the other hand by exactly same argument of the proof of part 1 of this proposition, we have that 
\begin{equation}\label{eq:sumq}
\begin{split}
&\sum_{q_{i}\neq 0} 1\\
&\le \frac{1}{\left(N- M\right)!} m^{N- \sum_{j}\sum_{x_{i,j}=1}^{N_{j}}\alpha_{i,j}}.
\end{split}
\end{equation} 
Let $N_{j}=\gamma_{j}N$. Hence $\sum_{j}(j-1)\gamma_{j}=1.$

\noindent
Putting \eqref{eq:sumpsumrsim} and \eqref{eq:sumq} in \eqref{eq:totwordcountgen} we have:
\begin{equation}\label{eq:totwordgensim}
\begin{split}
& \sum_{N_{j}}\sum_{\alpha_{i,j}}\sum_{p_{i},\ldots,p_{N}}\sum_{q_{i} \neq 0}\sum_{r_{i}'s} 2^{N}\frac{2m+1}{2k+1}\binom{2k+1}{k+m+1} n^{k-N+1}\prod_{j}(3j)^{(j-1)N_{j}}\\
& \le \sum_{N_{j}} \sum_{\alpha_{i,j}}\sum_{\text{positions of $0$ $q$ values}} \sum_{m}n^{k-N+1}2^{N}\frac{2m+1}{2k+1} \binom{2k+1}{k+m+1} \\
&~~~~~~~~~~\frac{N}{\prod_{j} ((j-1)!)^{N_{j}} \prod_{j} N_{j}!} m^{\sum_{j} jN_{j}-1} \frac{1}{\left(N- \sum_{j}\sum_{x_{i,j}=1}^{N_{j}}\alpha_{i,j}\right)!} m^{N- \sum_{j}\sum_{x_{i,j}=1}^{N_{j}}\alpha_{i,j}} \prod_{j}(3j)^{(j-1)N_{j}}\\
& \asymp n^{k}2^{2k+\frac{1}{2}}\sum_{N_{j}} \sum_{\alpha_{i,j}}\sum_{\text{positions}} \sum_{m}\left( \frac{1}{\sqrt{k}}\right)^{3(N-1)} 2^{N} \frac{\sqrt{k}}{\sqrt{\pi(k^2-m^{2})}}\exp\left( -\frac{m^2}{k} \right)\frac{2m+1}{2k+1}\\
&~~~~~~~~~~ \frac{N}{\prod_{j} ((j-1)!)^{N_{j}} \prod_{j} N_{j}!}\frac{1}{\left(N- \sum_{j}\sum_{x_{i,j}=1}^{N_{j}}\alpha_{i,j}\right)!} m^{\sum_{j}jN_{j}+ N- \sum_{j}\sum_{x_{i,j}=1}^{N_{j}}\alpha_{i,j}-1} \prod_{j}(3j)^{(j-1)N_{j}}\\
& \asymp n^{k}2^{2k+\frac{1}{2}} \sum_{N_{j}}\sum_{\alpha_{i,j}}\sum_{\text{positions}} \left(  \frac{1}{\sqrt{k}}\right)^{3N  -{\sum_{j}jN_{j}-N+\sum_{j}\sum_{x_{i,j}=1}^{N_{j}}\alpha_{i,j}}}2^{N} \Gamma\left( \frac{\sum_{j}jN_{j}+N-\sum_{j}\sum_{x_{i,j}=1 }\alpha_{i,j}+\xi}{2} \right)\\
& ~~~~~~~~~ \frac{N}{\prod_{j} ((j-1)!)^{N_{j}} \prod_{j} N_{j}!}\frac{1}{\left(N- \sum_{j}\sum_{x_{i,j}=1}^{N_{j}}\alpha_{i,j}\right)!}\prod_{j}(3j)^{(j-1)N_{j}}\\
& \le n^{k}2^{2k+\frac{1}{2}} \sum_{N_{j}}\sum_{\alpha_{i,j}}\sum_{\text{positions}}\left(  \frac{1}{\sqrt{k}}\right)^{3N  -{\sum_{j}jN_{j}-N+\sum_{j}\sum_{x_{i,j}=1}^{N_{j}}\alpha_{i,j}}}2^{N}\Gamma \left( N + \frac{\sum_{j} N_{j}}{2}-\frac{M}{2} + \frac{\xi}{2} \right)\\
&~~~~~~~~~ \frac{N}{\prod_{j} ((j-1)!)^{N_{j}} \prod_{j} N_{j}!}\frac{1}{\left(N- M\right)!} \prod_{j}\exp\left( (j-1)N_{j}\log(3j) \right)\\
&\le n^{k}2^{2k+\frac{1}{2}} \sum_{N_{j}}\sum_{\alpha_{i,j}}\sum_{\text{positions}}\left(  \frac{1}{\sqrt{k}}\right)^{3N  -{\sum_{j}jN_{j}-N+\sum_{j}\sum_{x_{i,j}=1}^{N_{j}}\alpha_{i,j}}} 2^{N}C^{N} (2N)^{\frac{\xi}{2}+ \sum_{j}\frac{N_{j}}{2}} \frac{\Gamma\left(N- \frac{M}{2}\right)}{\left( N-M \right)!}\\
&~~~~~~~~~~N \exp\left( - \sum_{j}\left(N\gamma_{j}\right)\log\left( N\gamma_{j}\right)- \sum_{j}N\gamma_{j}(j-1)\log(j-1) \right)\prod_{j}\exp\left( (j-1)N_{j}\log (3j) \right)
\end{split}
\end{equation}
This analysis is somewhat cumbersome and tedious. We analyze them term by term by term. 

\noindent
%We at first start with the term $\prod_{j}\prod_{x_{i,j}=1}^{N_{j}}\exp\left( \left(j-\alpha_{i,j}\right) \log\left(3\alpha_{i,j}\right) \right)$. We at first fix $j$ and maximize the value of 
%\begin{equation}
%f(\alpha)=(j-\alpha)\log(3\alpha)
%\end{equation}
%with respect to $\alpha$.
%Differentiating $f(\alpha)$ with respect to $\alpha$ we get 
%\begin{equation}
%\begin{split}
%\frac{df(\alpha)}{d\alpha}&= - \log(3\alpha) + (j-\alpha)\frac{3}{3\alpha}\\&=-\log(3\alpha) + \frac{j}{\alpha}-1.
%\end{split}
%\end{equation}
%Setting the derivative to $0$, we have 
%\begin{equation}
%\begin{split}
%-\log(3\alpha) +\frac{j}{\alpha}-1=0
%\end{split}
%\end{equation}
%Now observe that $f''(\alpha)$ is negative for all $\alpha$. Hence the solution to $f'(\alpha)=0$ is the unique maximizer. 
%
%It is easy to see that for large $j$, the optimal $\alpha$ is of the order $\frac{j}{\log j}$
%As a consequence, we have the maximum value of $(j-\alpha)\log(3\alpha)$ is less than or equal to $\log\left(3\frac{j}{\log j}\right)^{j}= j\log(3j)- j\log \log(j)$. Hence 
%\begin{equation}
%\begin{split}
%&\prod_{j}\prod_{x_{i,j}=1}^{N_{j}}\exp\left( \left(j-\alpha_{i,j}\right) \log\left(3\alpha_{i,j}\right) \right)\le \exp\left( \sum_{j} N_{j} \left(j\log(3j)- j\log \log j\right) \right)\\
%& \le \exp \left( N\sum_{j}\gamma_{j} \left( j\log j + j\log 3 - j \log \log j \right)  \right)
%\end{split}
%\end{equation}

We at first simply bound $\frac{\Gamma\left( N-\frac{M}{2} \right)}{N-M}\le N^{\frac{M}{2}}= \exp\left( \frac{M}{2}\log N \right)$. 

Next we look at the term 
\begin{equation}\label{eq:termentropy}
\begin{split}
&\exp\left( - \sum_{j}\left(N\gamma_{j}\right)\log\left( N\gamma_{j}\right) \right)\\
& = \exp \left( - N \sum_{j} \gamma_{j}\left(\log(N)+ \log(\gamma_{j})\right) \right)\\
& = \exp \left( - N\log N \sum_{j}\gamma_{j} - N \sum_{j} \gamma_{j}\log(\gamma_{j}) \right)
\end{split}
\end{equation}
Our first job is to bound the term 
\[
\exp\left( -N \sum_{j} \gamma_{j}\log \gamma_{j} \right)
\]
So we come to the following optimization problem 
\begin{equation}\label{eq:optentropy}
\begin{array}{ll}
&\max -\sum_{j}\gamma_{j}\log(\gamma_{j})\\
&\text{subj. to} \\
& \sum_{j}(j-1)\gamma_{j}=1 
\end{array}
\end{equation}
We at first transform the variables to $z_{j}= (j-1)\gamma_{j}$. This reduces the optimization problem to 
\begin{equation}\label{eq:optentropytransformed}
\begin{array}{ll}
&\max -\sum_{j}\frac{z_{j}}{j-1}\log\left(\frac{z_{j}}{j-1}\right)\\
&\text{subj. to} \\
& \sum_{j} z_{j}=1 
\end{array}
\end{equation}
Now we apply the method of Lagrange multiplier to optimize 
\begin{equation}\label{eq:lagrangeentropy}
\begin{split}
f(z_{1},\ldots, z_{N}, \lambda)= - \sum_{j}\frac{z_{j}}{(j-1)}\log \left( \frac{z_{j}}{(j-1)} \right) -\lambda \left( \sum_{j} z_{j}-1 \right).
\end{split}
\end{equation}
Observe that the function $(z_{1},\ldots,z_{N})\mapsto -\sum_{j}\frac{z_{j}}{j-1}\log\left(\frac{z_{j}}{j-1}\right)$ is a concave function and the set $\sum_{j}z_{j}=1$ is a convex set. Hence the method of Lagrange multiplier gives the unique maximizer.

Taking the partial derivative of $f(z_{1},\ldots, z_{N},\lambda)$ with respect to $z_{j}$ and setting it to $0$  we have 
\begin{equation}
\begin{split}
& \frac{\partial f(z_{1},\ldots, z_{N},\lambda)}{\partial z_{j}}=  - \frac{1}{j-1}\log \left( \frac{z_{j}}{j-1} \right)- \frac{1}{j-1} -\lambda =0 \\
& \Rightarrow - 1 - \lambda(j-1) = \log\left( \frac{z_{j}}{j-1} \right)\\
& \Rightarrow  \exp\left(-1 - \lambda (j-1) \right) = \frac{z_{j}}{(j-1)}\\
& \Rightarrow z_{j}= (j-1) \exp\left( -1 - \lambda(j-1) \right).
\end{split}
\end{equation}
Now we apply the constraint to have 
\begin{equation}\label{eq:constraint}
\sum_{j=2}^{N}\left( j-1 \right)^{2} \exp \left(-1 - \lambda (j-1)  \right)=1. 
\end{equation}
It is easy to observe that the solution to the above equation is unique and the value $\lambda$ remains uniformly bounded over $N$.
As a consequence the value 
\begin{equation}
\begin{split}
&-\sum_{j} \frac{z_{j}}{(j-1)}\log\left( \frac{z_{j}}{(j-1)} \right)\\
& = \sum_{j} \exp\left( -1 - \lambda(j-1) \right)\left( 1+ \lambda (j-1) \right)
\end{split}
\end{equation}
also remains uniformly bounded. 
Hence \eqref{eq:termentropy} is bounded by 
\begin{equation}\label{eq:termentropyfinal}
C'^{N} \exp\left( -N\log N  \sum_{j}\gamma_{j} \right).
\end{equation}
Now if we have a look at \eqref{eq:totwordgensim}, then inside the exponential only linear functions of $\gamma_{j}$'s remain. To make things clear we have a look at the term inside the summand of \eqref{eq:totwordgensim}. 
\begin{equation}\label{eq:firstreduction}
\begin{split}
& \left( \frac{1}{\sqrt{k}} \right)^{\sum_{j}(j-2)N_{j}+M}NC''^{N} \exp\left(\log N \frac{\xi}{2} + N\log N \sum_{j}\frac{\gamma_{j}}{2}  + \frac{M}{2}\log N\right.\\
&~~~~~~~~\left.  -  N \log N \sum_{j}\gamma_{j} - \sum_{j}N\gamma_{j}(j-1)\log(j-1)+ N \sum_{j} \gamma_{j}\left( (j-1)\log j + (j-1)\log 3 \right)\right)
\\
&= N^{\frac{\xi}{2}+1}  C''^{N} \exp\left(- \frac{\sum_{j}(j-2)N_{j}}{2}\log k - \frac{M}{2}\log k +\frac{M}{2}\log N - \frac{N\log N}{2} \sum_{j} \gamma_{j}\right.\\
&~~~~~~~~\left. - \sum_{j} N \gamma_{j}\left((j-1)\log(j-1) - (j-1)\log j  \right)  + N \sum_{j}\gamma_{j}(j-1)\log(3)\right)
\end{split}
\end{equation}
Now $(j-1)\log j -(j-1)\log(j-1)$ is of the order $\log(j+1) +1$. So
\begin{equation}
\begin{split}
\sum_{j}\gamma_{j}\left( (j-1)\log j - (j-1)\log (j-1) \right)\le c
\end{split}
\end{equation} 
implying 
\[
\exp\left( N \sum_{j} \gamma_{j}\left( j\log j - (j-1)\log(j-1) \right) \right)\le C'''^{N}.
\]
Same can be said about $\exp\left( N\log 3 \sum_{j}\gamma_{j}(j-1) \right)$. Also $\frac{M}{2}\left( \log N - \log k \right)<0$. 

As a consequence, the only interesting term is 
\begin{equation}\label{eq:interest}
\begin{split}
&\exp\left(- \frac{N\log N}{2} \sum_{j} \gamma_{j} -  \frac{\sum_{j}(j-2)N_{j}}{2}\log k  \right)\\
&= \exp\left(  -\frac{N\log N}{2} \sum_{j}\gamma_{j}- \frac{N\log k}{2}\sum_{j}(j-2)\gamma_{j}\right)
\end{split}
\end{equation}
Now we come to another optimization problem:
\begin{equation}\label{eq:opt2}
\begin{array}{ll}
&\max - N\log N \sum_{j}\gamma_{j} - N\log k \sum_{j} \gamma_{j} (j-2)\\
& \text{subj. to}\\
& \sum_{j} (j-1) \gamma_{j}=1
\end{array}
\end{equation}
We again transform it to $z_{j}=(j-1) \gamma_{j}$ to reduce the problem to
\begin{equation}\label{eq:optred}
\begin{array}{ll}
& \max - N\log N \sum_{j} \frac{z_{j}}{(j-1)} - N \sum_{j} \frac{z_{j}}{(j-1)} (j-2)\\
& \text{subj. to}\\
& \sum_{j} z_{j}=1.
\end{array}
\end{equation} 
This is equivalent to minimize 
\begin{equation}\label{eq:finalmin}
\begin{split}
&\frac{N\log N}{j-1} + \frac{N \log k (j-2)}{(j-1)}\\
&= \frac{N\log N}{j-1} + N\log k -\frac{N\log k}{j-1}\\
&= \frac{N}{j-1} \left( \log N - \log k \right) + N\log k. 
\end{split}
\end{equation}
Since $j-1\ge 1$ and $\log k\ge \log N$, we have the minimum value of \eqref{eq:finalmin} is $N\log N$.  
%over $j$. Now $\frac{j}{j-1}\le 2$. So we reduce the problem further to minimize 
%\begin{equation}
%\begin{split}
%f(j)=\frac{N\log N}{j-1} + 2N \log \log j 
%\end{split}
%\end{equation} 
%differentiating over $j$ and equate to $0$, we have 
%\begin{equation}\label{eq:optimalsol}
%f'(j)= -\frac{N\log N}{(j-1)^2} + 2 N \frac{1}{j\log j} =0.
%\end{equation}
%We now prove that the solution of \eqref{eq:optimalsol} is the global minimizer. To do this we consider the ratio $\frac{(j-1)^2}{j\log j}$. We prove that after some finite constant this function is monotonically increasing. This will prove the required result. Now 
%\begin{equation}
%\begin{split}
%&\tilde{f}(j) =\frac{(j-1)^2}{j\log j}= \frac{j}{\log j}- \frac{2}{\log j} + \frac{1}{j\log j}\\
%\Rightarrow & \frac{d\tilde{f}(j)}{dj}= \frac{1}{\log j}- \frac{j-2}{j(\log j)^2}- \frac{\log j +1}{(j\log j)^2}\\
%&= \frac{j^2\log j - j(j-2)-(\log j +1)}{(j\log j)^2}
%\end{split}
%\end{equation} 
%So $\frac{d\tilde{f}(j)}{dj}>0$ after some finite constant. As a consequence after some finite constant $f'(j)$ has a unique zero and before that $f'(j)$ is negative and after that $f'(j)$ is positive. On the other hand for large enough $N$ and uniformly bounded $j$, $f'(j)$ is strictly negative. So for large enough $N$ the solution to $f'(j)=0$ is the unique minimzer.   
%
%
%It is easy to see that the optimizer is of the order $\log N \log \log N$
%Hence the minimum value of \eqref{eq:optred} is $-N f(N)$ where $f(N)$ is a function of $N$ which go to infinity as $N$ diverges but at a very slow rate. 

Now we take the sum over all possible $N_{j}$'s and $\alpha_{i,j}$'s. Firstly $N_{j}$ are chosen in two steps. At first we fix $x$ and choose $x$ many positions from $\{ 1,2\ldots, N \}$. This can be done in $\binom{N}{x}$ ways. This chosen set is the collection of $j$'s such that $N_{j}\neq 0$. Now we partition $N$ balls into $\sum_{j}(j-1)$ bins such that each bin has nonzero number of balls. This can be done in at most $\binom{N-1}{\sum_{j}(j-1)-1}\le 2^{N-1}$ ways. As a consequence the total number of choices of $N_{j}$'s is bounded by $4^{N}$. Finally we choose  $\alpha_{i,j}$ many points from $(j-1)$. These are the positions where the corresponding $q$ value is $0$. So we have 
\begin{equation}
\sum_{\alpha_{i,j}} \sum_{\text{positions of $0$ $q$ values}}1\le \sum_{\alpha_{i,j}} \binom{j-1}{\alpha_{i,j}}\le 2^{\sum_{j}(j-1)N_{j}}= 2^{N}
\end{equation} 
Hence the sum in \eqref{eq:totwordgensim} is finite. Now we need to prove it actually goes to $0$. This is true due to the fact that we simply bound the term $\exp\left( \frac{M}{2}\left( \log N - \log k \right) \right)$ by $1$. However for small $N$ this term goes to zero. To make this clear we have a look at \eqref{eq:totwordgensim}. We have this equation reduces to  
\begin{equation}
\sum_{N} \exp\left( \frac{M}{2}(\log N -\log k) - \frac{N\log N}{2} + CN \right)
\end{equation}
The tail of the above equation also decays like $\exp\left( - c N\log N \right)$ for some finite constant $c$. So if we take $N_0$ large enough such that $\exp\left( -c N\log N \right)\le \frac{1}{\sqrt{k}}$, then $$\sum_{N\le N_{0}}\exp\left(\frac{M}{2}\left( \log N - \log k\right) - \frac{N\log N}{2}\right)\le \left(\frac{1}{\sqrt{k}}\right)^{M(1-\varepsilon)} $$ 
for any $\varepsilon >0$. On the other hand 
$$
\sum_{N\ge N_{0}} \exp\left( \frac{M}{2}\left( \log N - \log k \right) - \frac{N\log N}{2} \right)\le \left( \frac{1}{\sqrt{k}} \right).
$$
So \eqref{eq:totwordgensim} goes to $0$ whenever $M\ge 1$. On the other hand if there exists some $j>2$ such that $\gamma_{j}>0$, the same argument can be used to prove that \eqref{eq:totwordgensim} goes to $0$. This completes the proof. 
\end{proof}
%We now state a result about general word count of a Wigner matrix. 
%\begin{proposition}
%Suppose 
%\end{proposition}
Now we come to an important proposition for proving the universality. 
\begin{proposition}\label{prop:ge3}
Suppose the entries $x_{i,j}$ satisfy condition (iii) of Assumption \ref{ass:wig}. Let $\mathcal{W}_{\ge 3,k}$ denote the class of words of length $k+1$ where every edge is traversed at least twice and some edge is traversed at least thrice. Then following is true whenever $k= [tn^{\frac{2}{3}}]$  for some $t\in (0,\infty)$:
\begin{equation}\label{eq:wigge3}
\frac{1}{n^{\frac{k}{2}}}\sum_{w \in \mathcal{W}_{\ge 3,k}} \E\left[ X_{w} \right]\to 0.
\end{equation}
\end{proposition}
The proof Proposition \ref{prop:ge3} is somewhat long and tedious. However this is the most important result in the paper. So we divide the proof in several subsections and lemmas.
\subsection{The broad idea}
We at first sketch the outline of the proof. Firstly observe that allowing repetitions of edges more that twice will make $\E\left[ X_{w} \right]$ ``large" so we have to do a refined count of such words to kill this blow up. The decrease in counts of the words comes from the fact that whenever an edge is repeated more that twice, it fixes some constraints. The main idea of this refined count is to use the constraints in a systematic way. We shall analyze the count of the skeleton word and Dyck paths separately.Now we shall discuss the main steps of the proof very briefly\\
\textbf{(i)Two types of repetitions:}
At the beginning one might observe that when an edge is repeated more than twice, this repetition can happen the following two ways. Firstly the repeated edge can be such that the instant corresponding to one of its endpoint is an instant where the path comes to a given level several times. On the other hand the repetition of the edge might not use the level of an instant where the path comes several times. One might look at the following figures for clarification. 
\begin{figure}[H]
        \begin{center}
                \includegraphics[width=0.4\textwidth]{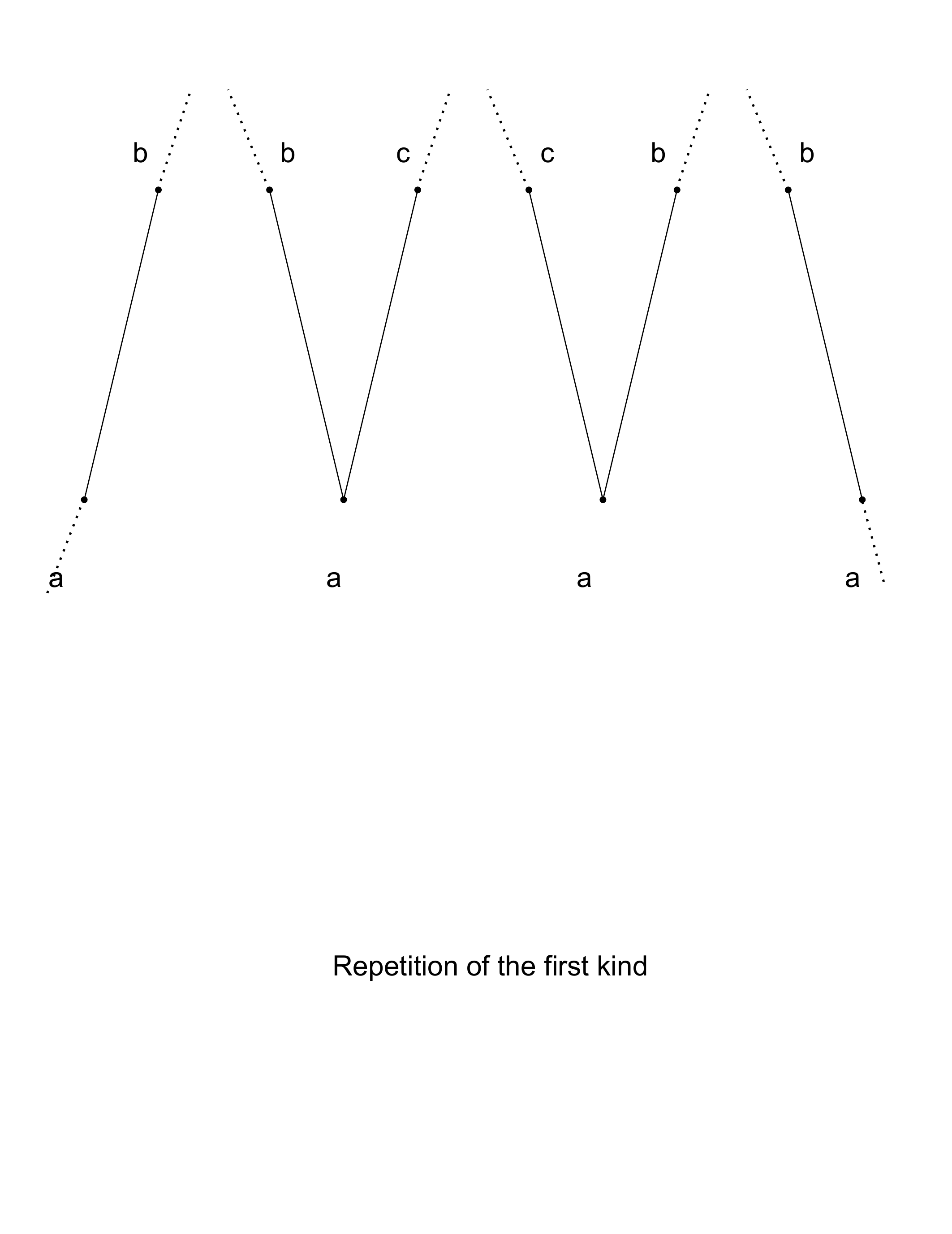}
        ~ %add desired spacing between images, e. g. ~, \quad, \qquad etc.
          %(or a blank line to force the subfigure onto a new line)
      \end{center}   
 \end{figure}
 \begin{figure}[H]
        \begin{center}
                \includegraphics[width=0.4\textwidth]{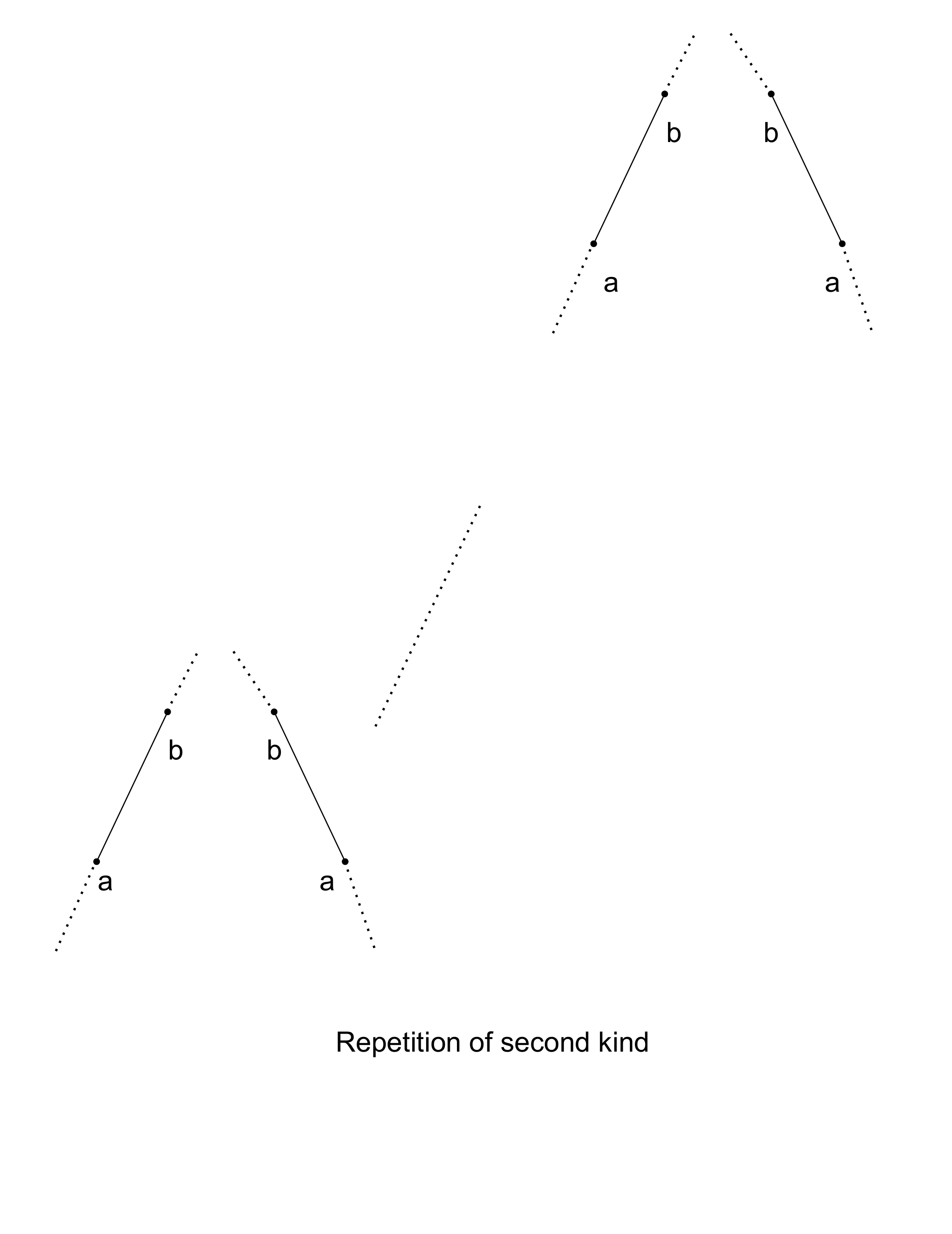}
        ~ %add desired spacing between images, e. g. ~, \quad, \qquad etc.
          %(or a blank line to force the subfigure onto a new line)
      \end{center}   
 \end{figure}
Lemma \ref{lem:universal} asserts that for a ``typical" Dyck path the return to a specific level can not be too large. The other case needs to be dealt a little differently.\\
\textbf{(ii)A strategy to construct the skeleton word keeping the returns to a specific level in mind:} As returns to a specific level reduce the number of free parameters, we produce a strategy to construct the skeleton words corresponding given the number of times the path corresponding to the skeleton word returns to a specific levels. This is done in two steps: Observe that the path corresponding to the skeleton word is also a Dyck path where some edges are removed from it. We at first construct the Dyck path. Then from the Dyck path we construct the skeleton word. The reader might observe by now that typically the Dyck path corresponding to a skeleton word is quite smooth in the sense there are only finitely many spikes ($N$ many).
This gives us an advantage. In Lemma \ref{lem:noncrossing} and Algorithm \ref{alg:simplewig!} we describe a strategy to construct the Dyck path corresponding to the skeleton word. Lemma \ref{lem:noncrossing} is well known in literature. However we created a bijection which is useful for our purpose. In Algorithm \ref{Alg:skeleton2} we describe the construction of the skeleton word from the Dyck path.\\
\textbf{(iii)Permutation of the type $j\ge 2$ vertices:} Another important aspect which reduces the count is the permutation of the type $j$ vertices. We have seen that in the skeleton word there are $N$ many and in the Dyck paths there are $\Gamma$ many type $j \ge 2$ vertices. All edges repeated more than twice will correspond to a type $j \ge 2$ vertex. However there are some additional constraints here. For example if an edge is a repetition of the first type, then the vertex corresponding to a type $j$ instant can appear only immediately after the return to a certain level. Hence we take care of such vertices in this way. On the other hand if an edge is a repetition of the second then the corresponding two type $j \ge 2$ vertices come as a pair. In this case we shall only consider one of them and seek a position where to insert the other one. As a consequence we only consider a permutation of a reduced number of vertices which we call non-ignored vertices and when all the other parameters are fixed we seek for positions to insert the other vertices. We call them ignored vertices. However there might be cases where different choices of non-ignored and ignored vertices lead to the same word. To avoid this over counting we shall impose certain constraints. These will be discussed in the course of the proof.\\
\textbf{(iv)Considering edges with odd number of traversals:}One of the novelties of this paper is being able to handle general non-symmetric distributions. This leads us to consider word when an edge is traversed odd number of times. This makes the calculation a little bit trickier. However this can be handled by implementing proper constraints. One might look at the subsequent parts.
%The main idea is as follows. When the skeleton word comes to a level many times, we are fixing certain constraints on the $p_{i}$, $q_{i}$ and $r_{i}$'s. This will reduce the number of word count. In this algorithm we formalize this intuition.

%To begin with at first we fix an integer $N'$. This number $N'$ is lesser than the number $N$. This actually denotes the number of free parameters in the skeleton word. However $N$ denotes the number of type $j\ge 2$ instants. We at first fix $N'$ and some values $p'_{1},\ldots, p'_{N'}$; $q'_{1},\ldots, q'_{N} $ where we impose the constraint that $\sum_{i} p'_{i}\le m+ m'$  and $\sum_{i} q'_{i}\le m+m'$. We also fix the $r_{i}$'s. As usual these positions denote the locations of the initial positions of a type $j \ge 2$ instants.
%\subsection{Some notations}
  
\subsection{Machinery to prove Proposition \ref{prop:ge3}}
In this subsection we give detailed description of the machinery needed to prove Proposition \ref{prop:ge3}.

\noindent
\textbf{(i) Results regarding the number of returns to a certain level of a Dyck path}
\begin{lemma}\label{lem:universal}
Consider the uniform probability measure on $m$ Dyck paths of length $2k_{1},\ldots, 2k_{m}$ such that $\sum_{i=1}^{m} 2k_{i}= 2k$ a fixed number. We call this probability measure $\mathbb{P}_{D,k,m}$. We now fix $\tau$ many levels $q_{1},\ldots, q_{\tau}$. Let $N_{k_{i},m}(q_{i})$ denote the number of returns to level $q_{i}$ in the $i$ th Dyck path before falling down. Then the random variables $N_{k_{1},m}(q_{1}),\ldots , N_{k_{\tau},m}(q_{\tau})$ can be stochastically dominated by $\tau$ many i.i.d. random variables with common distribution $X$ such that $\mathbb{P}\left[X\ge l\right]\le \frac{k^{\frac{3}{2}}}{2^{l}}$.

%Consider the uniform probability measure supported on Dyck paths of length $2s_{n}$. We call this probability measure $\mathbb{P}_{D,s_{n}}$. Let us fix a level $l$ and consider the paths which are such that the path returns to the level $l$ most $N_{s_{n}}(l)$ number of times in any chunk of the path before falling down from the level $l$. Then there is a sufficiently large constant $C$ such that $\mathbb{P}_{D,s_{n}}\left[ \max_{0\le l \le s_{n}} N(l) > C\log(s_{n})\right]\to 0$ as $n \to \infty$.
\end{lemma} 
\begin{proof}
The proof of this lemma is some what tricky. We divide it in the following steps:\\
\textbf{Step 1:} We at first consider the levels of the starting points of the Dyck paths. With slight abuse of notation we call any such level as level $0$. Let the corresponding length of the Dyck path be $2k_{i}$ for some $i$ and we look at the returns to level $0$ before falling down in the $i$ th Dyck path. We prove here that $\mathbb{P}_{D,k,m}\left[ N_{k_{i},m}(0) \ge t \right]\le \frac{k_{i}^{\frac{3}{2}}}{2^{t}}$. To prove this result we compare this probability with a similar probability of the simple symmetric random walk of length $2k_{i}$. At this point we at first fix all the other Dyck paths and consider the uniform probability measure of the simple symmetric random walk of length $2k_{i}$. We call this measure $\mathbb{P}_{R,k_{i}}.$ Let $\mathcal{N}_{k_{i}}(t,0)$ be the collection of random walk paths which starts from $0$ and returns to $0$ at time point $2k_{i}$ and in between time points $0$ to $2k_{i}$ it returns to $0$, exactly $t$ many times. Then 
\begin{equation}
\mathbb{P}_{R,k_{i}}\left[ \text{The random walk returns to $0$, $t+1$ times} \cap x(2k_{i})=0 \right]=\frac{\#\mathcal{N}_{k_{i}}(t,0)}{4^{k_{i}}}<1.
\end{equation} 
Now we inspect a typical path in $\mathcal{N}_{k_{i}}(t,0)$. Observe that the random walk starts from $0$ and in the very next step it can go either above or below $0$. Now once the path goes either above $0$ or below $0$ it stays there until it returns to $0$ for the first time. Now in the next step it also goes either $0$ or below $0$ and stays there until it returns to $0$ for the second time. In particular after any return to $0$ there are $2$ possible ways to choose in the very next step determining whether the random walk goes above or below $0$ in the next part of time before returning to $0$. On the other hand, if we consider a Dyck path, it always stays above $0$. Hence the number of Dyck paths which start from $0$ and returns to $0$ exactly $t+1$ number of times is given by $\frac{1}{2^{t}}\#\mathcal{N}_{k_{i}}(t,0)$. As a consequence if we consider the uniform measure over all Dyck paths of length  exactly $2k_{i}$ (call it $\mathbb{P}_{k_{i}}$), then under this measure the probability that a Dyck path comes to $0$ for exactly $t$ times is given by:
\begin{equation}
\mathbb{P}_{k_{i}}\left[ \text{The Dyck path comes to $0$ exactly $t+1$ times} \right]\le \frac{\#\mathcal{N}_{k_{i}}(t,0)}{2^{t}C_{2k_{i}}}\le \frac{k_{i}^{\frac{3}{2}}}{2^{t}}\le \frac{k^{\frac{3}{2}}}{2^{t}}. 
\end{equation}
Since the number of returns to $0$ in the $i$ th Dyck path doesn't depend on the other Dyck paths, the number of paths such that the $i$ th Dyck path comes to $0$ exactly $t+1$ many times is bounded by $\frac{k^{\frac{3}{2}}}{2^{t}} \prod_{i} C_{2k_{i}}$. Hence the probability
\begin{equation}
\begin{split}
&\mathbb{P}_{D,k,m}\left[ N_{k_{i},m}(0)=t \right]= \frac{\# \text{ of paths having the required property}}{\sum_{k_{1},\ldots, k_{m}=k} \prod_{i} C_{2k_{i}}}\\
&\le \frac{\frac{k^{\frac{3}{2}}}{2^{t}} \sum_{k_{1},\ldots, k_{m}~|~ \sum_{i}k_{i}=k}\prod_{i} C_{2k_{i}}}{\sum_{k_{1},\ldots, k_{m}~|~ \sum_{i}k_{i}=k} \prod_{i} C_{2k_{i}}}\le \frac{k^{\frac{2}{3}}}{2^{t}}.
\end{split}
\end{equation}
\textbf{Step 2:} In this step we prove the remaining of Lemma \ref{lem:universal}. The proof is done by conditioning and induction on the level. In \textbf{step 1}, we proved an upper bound on the tail of the number of returns to $0$ for the $i$ th Dyck path. Here we show that the conditional distribution of the number of returns to level $1$ given the number of returns to level $0$ follows the same upper bound. Since given the lengths of the Dyck paths, the number of times one path returns to a level is independent of the number of times another different path returns to another level, if $N_{k_{1},m}(q_{1})$ and $N_{k_{2},m}(q_{2})$ belong to different Dyck paths, then the proof can be done by simply looking at the distribution of $N_{k_{1},m}(q_{1}) \left|N_{k_{2},m}(q_{2}) \right.$. On the other hand, if $N_{k_{1},m}(q_{1})$ and $N_{k_{2},m}(q_{2})$ denote the return to two different levels of the same Dyck path, the proof can be completed by the repeated use of the argument we give next.

Now we go by conditioning. In particular we look at $N_{k_{i},m}(1)\left|N_{k_{i},m}(0) \right.$.
To this end we at first fix the value of $N_{k_{i},m}(0)$ to be $t_{1}$ and assume the random walk returns to $0$ at instants $2k_{i,1},2k_{i,1}+2k_{i,2},\ldots 2k_{i,1}+\ldots+2k_{i,t_{1}}$. Since any of such paths return to $0$ exactly at instants $2k_{i,1},2k_{i,1}+2k_{i,2},\ldots, 2k_{i,1}+\ldots+2k_{i,t_{1}}$, the total number of paths of this type is given by $C_{2k_{i,1}-2}\ldots C_{2k_{i,t_{1}}-2}$. Here we have used the fact that the random walk returns to $0$ exactly at instants $2k_{i,1},2k_{i,1}+2k_{i,2},\ldots, 2k_{i,1}+\ldots+2k_{i,t_{1}}$ which implies the random walk goes one step down at instants $2k_{i,1},2k_{i,1}+2k_{i,2},\ldots, 2k_{i,1}+\ldots+2k_{i,t_{1}}$ and goes one step up at instants $1,2k_{i,1}+1, \ldots, 2k_{i,1}+\ldots +2k_{i,t_{1}-1}+1$. This explains the quantity $C_{2k_{i,1}-2}\ldots C_{2k_{i,t_{1}}-2}$. By $N_{k_{i},m}(1)$ we shall denote the number of returns to $1$ in the left most chunk. The arguments for the other chunks is exactly the same. Now fix a path satisfying this property and call it $\omega$. Let the path returns to level $1$, $t_{2}$ times in the left most chunk. Observe that the conditional probability of this path given $N_{k_{i},m}(0)=t_{1}$ and $k_{i,1},\ldots,k_{i,t_{1}}$ is:
\begin{equation}
\begin{split}
&\mathbb{P}_{D,k,m}[\omega \left| N_{k_{i},m}(0)=t_{1} \cap \text{the returns are at $2k_{i,1},\ldots , 2k_{i,1}+\ldots + 2k_{i,t_{1}} $} \cap \right.\\
&~~~~~~~~\left. \text{The length of the Dyck paths are } k_{1},\ldots, k_{m} \right.]= \frac{\frac{1}{C_{2k_{i}}}}{\frac{C_{2k_{i,1}-2}\ldots C_{2k_{i,t_{1}}-2}}{C_{2k_{i}}}}\\
&=\frac{1}{C_{2k_{i,1}-2}C_{2k_{i,2}-2}\ldots C_{2k_{i,t_{1}}-2}}\\
 \Rightarrow & \mathbb{P}_{D,k,m} \left[ N_{k_{i},m}(1)= t_{2}\left| N_{k_{i},m}(0)=t_{1} \cap \text{the returns are at $2k_{i,1},\ldots , 2k_{i,1}+\ldots + 2k_{i,t_{1}} $} \cap \right. \right.\\
& ~~~~~~\left. \text{The length of the Dyck paths are } k_{1},\ldots, k_{m}  \right]\le \frac{k_{i,1}^{\frac{3}{2}}C_{2k_{i,1}-2}C_{2k_{i,1}-2}C_{2k_{i,2}-2}\ldots C_{2k_{i,t_{1}}-2}}{2^{t_{2}}C_{2k_{i,1}-2}C_{2k_{i,1}-2}C_{2k_{i,2}-2}\ldots C_{2k_{i,t_{1}}-2}}\\
&\le \frac{k_{i,1}^{\frac{3}{2}}}{2^{t_{2}}} \le \frac{k^{\frac{3}{2}}}{2^{t_{2}}}
\end{split}
\end{equation} 
Since the r.h.s. does not depend on $t_{1}$ and the values $2k_{i,1},\ldots , 2k_{i,1}+\ldots + 2k_{i,t_{1}} $ and $k_{1},\ldots, k_{m} $, we have the required result.

\end{proof}
We now fix some notations. Firstly observe that in the skeleton word, after the end of each downward chunk there is an upward chunk of length at least $1$. So there will be $N$ such endpoints. We enumerate them from left to right. Along with these points we also add the points which correspond to the first arrival to the level of $N$ endpoints considered here. Let the total number of points in this collection be $\tilde{N}$. Observe that $\tilde{N}\le 2N$. Now we introduce a partition of $\{ 1,2,\ldots, \tilde{N} \}$ in the following way. We place $i$ and $j$ in the same block if the $i$ th and $j$ th endpoint is in the same level and if we look the word from left to right it does not fall below the level of the $i$ th endpoint in between the $i$ th and $j$ th endpoint. Although, the total number of partitions of $\{ 1,\ldots ,\tilde{N} \}$ is quite large, we shall see the number of feasible partition is only of exponential order. One can see that this partition is a non-crossing partition of $\{1,\ldots , \tilde{N} \}$. It is a well known fact that there is a bijection from the set of all non-crossing partition of $\{1,\ldots , \tilde{N}  \}$ to the Dyck paths of length $2\tilde{N}$. As a consequence the total number of such partitions are bounded by $16^{N}$. The following argument explicitly constructs this map. Observe that the interpretation of this map will be useful to us to bound certain quantities. 
\begin{lemma}\label{lem:noncrossing}
For any $\tilde{N}$ let $NC(\tilde{N})$ denotes the set of non-crossing partitions of $\{1,\ldots , \tilde{N} \}$. Then there is a bijection from $NC(\tilde{N})$ to the all set of all Dyck paths of length $2\tilde{N}$.
\end{lemma}
\begin{proof}
We start with a non crossing partition of $\tilde{N}$ and look at the block containing $1$. The point $1$ corresponds to the starting point of the Dyck path and all the other entries correspond to the returns to $0$. Observe that here we are not counting the endpoint of the Dyck path which will be specified from the other parameters. From all these points we place an upward edge from left to right. Now we look at the blocks which are just one step above to the current block. Here the first entries of each block denote the entry to the level $1$ and all the other entries denote the returns to the level $1$ before falling down. We place these points at level $1$ and put an upward edge from left to right from these points. We continue in this fashion until we run out of choices. Observe that by performing this procedure, we have specified all the upward edges in the Dyck path. So we fill up the remaining downward edges to to get a Dyck path. Since from each point we construct exactly one upward edge, the length of the Dyck path is $2\tilde{N}$.  This way given a non-crossing partition we created a Dyck path. On the other hand given a Dyck path the non-crossing partition is obvious.
\end{proof}
\noindent
Here we give an example of a non-crossing partition and the corresponding Dyck path. Let the non-crossing partition be $\{ 1 \}, \{ 2,6,7 \}, \{ 3,5 \}, \{ 4 \},\{ 8 \}$. The corresponding Dyck path is given in the following figure:
\begin{figure}[H]
        \begin{center}
                \includegraphics[width=0.8\textwidth]{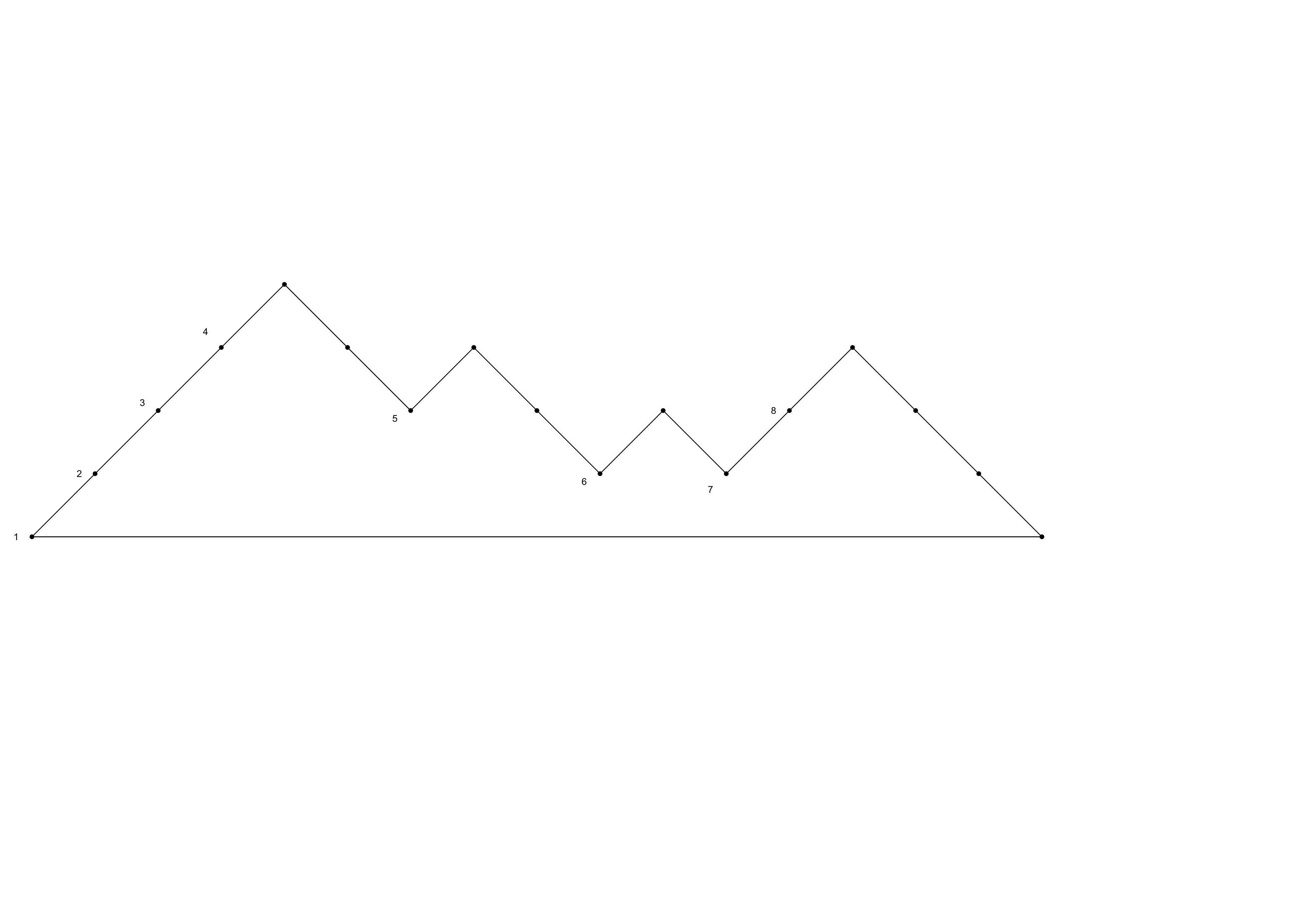}
        ~ %add desired spacing between images, e. g. ~, \quad, \qquad etc.
          %(or a blank line to force the subfigure onto a new line)
      \end{center}   
 \end{figure}

%To these end we map such partitions to Dyck path of length $2\tilde{N}$ in the following way. First of all we look for the positions which are in level $0$. We mark them. Say this points are $i_{0,1}<i_{0,2}<\ldots i_{0,k_{0}} $. Now we look at the level $1$. Observe that given $i_{0,1},\ldots, i_{0,k_{0}}$, there will be $k_{0}$ distinct blocks. In the $k$ th block the indexes denote the positions of return to level $1$ in between $k-1$ and $k$ th return $0$. This procedure continues similarly for level $2$, $3$ and so on. By defining the blocks in this way we are actually fixing the positions of returns to each level of Dyck path of length $2N$. As a consequence the total number of such partitions are bounded by $4^{N}$.

\noindent
Observe that the Dyck path formed in the process will have exactly $\tilde{N}$ points marked on it. From these points we choose the points which corresponds to the start of an upward chunk in the filled up path. A naive upper bound to the choice is $2^{\tilde{N}}\le 4^{N}$.

\noindent 
 %Now to the last entry of each block we assign a pointer taking value $0$ or $1$. 
%Here $0$ means that there is no unmatched upward edge starting from the corresponding level immediately after the last arrival to the level. On the other hand the value $1$ means the other way round. 
In our refined count of the skeleton word we shall fix one such partition and shall call it $\mathcal{P}$. 

\noindent
Now given this partition $\mathcal{P}$, and the parameters $p'_{i}$'s(to be specified), we at first create a path of length $2m$. This path is the path obtained by filling up the gaps in the skeleton path.

\noindent
So in the next part we give two algorithms (Algorithm \ref{alg:simplewig!} and Algorithm \ref{Alg:skeleton2}): at first we give an algorithm to construct the filled up path from the parameters $p'_{i}$'s and then we give an algorithm to form the skeleton path from this filled up path after fixing the additional parameters. 
\begin{algorithm}\label{alg:simplewig!}
Given a non-crossing partition $\mathcal{P}$, we fix the Dyck path ($\mathcal{D}$) according to Lemma \ref{lem:noncrossing}. This algorithm consists of the following steps:\\
\textbf{Step 1}
We look at the left most upward section of the path $\mathcal{D}$. Let us assume that in this section there are $\kappa$ many points which indicates the start of an upward chunk. In order to form the filled up word we at first move $p'_{1},p'_{2},\ldots,p'_{\kappa}$ step upward. This makes the length of the left most upward segment of the filled up path $\left(p'_{1}+\ldots + p'_{\kappa}\right)$.\\
\textbf{Step 2}
Now we go $p'_{\kappa+1}$ steps downward to reach a level. This level corresponds to the first level where the Dyck path $\mathcal{D}$ returns more than once after the first upward chunk. 
\\
\textbf{Step 3}
Next the Dyck path $\mathcal{D}$ has to move in the upward direction. Now we consider two cases as follows:
\begin{enumerate}
\item The next return to this level is the final return to the level before falling down.
\item The walk returns to this level more than once after the current return.
\end{enumerate}
Now we use the parameter $p'_{\kappa+2}$. In both the cases we move $p'_{\kappa+2}$ step up to reach a level. However for case $2$ we need to ensure the return to this level. Hence we need to fix a parameter depending on the future choices of the parameters.\\
\textbf{Step 4:} Now we arrive at a new level. There can be two further cases from here. 
\begin{enumerate}[i]
\item This level is the starting point of new upward chunk.
\item The next edge from this level is a downward edge. 
\end{enumerate}
For case $i$ we go back to \textbf{Step 1} with parameter $p'_{\kappa+3}$.

\noindent 
For case $ii$, if the next level in the path $\mathcal{D}$ in the downward direction is not a level previously explored, then we go $p'_{\kappa+3}$ step down and move to \textbf{Step 3}. 

\noindent 
Further if the next level in the path $\mathcal{D}$ in the downward direction is a level previously explored, we go down by a pre-specified number which denotes the difference between the current level and the next level. This explain the case 2. in \textbf{Step 3.}. We now again go to \textbf{Step 3}.

%
%
%In case $1.$ we go up by $p'_{\kappa+2}$ step up. In case $2.$ we go up by $\frac{p'_{\kappa+2}}{2}$ step and also we associate a downward chunk of length $\frac{p'_{\kappa+2}}{2}$ to ensure the return to the current level.
%
% Now we use the parameter $p'_{\kappa+2}$. In this case we move the filled up path in the upward direction for  $\frac{p'_{3}}{2}$ steps. Along with this upward chunk of length  $\frac{p'_{3}}{2}$ we also construct a downward chunk of length $\frac{p'_{3}}{2}$ to ensure the next return to this level.
%In case $1$, we go up $\frac{p'_{3}}{2}$ times and go down $\frac{p'_{3}}{2}$ times to come to the specific level. Then it goes up and down$\frac{p'_{4}}{2}$ times and do this for $\tau$ times. Next it goes to step 2 with parameter $p'_{\tau+3}$. 
%
%\noindent 
%In case $2$ we go up $p'_{3}$ many steps and after that we go to step $2$ with parameter $p'_{4}$.
\end{algorithm}

\noindent
\textbf{(ii) Strategies to construct the skeleton word given the filled up path and reduction of parameters:}
Before moving forward we make an important observation. Firstly observe that from Algorithm \ref{alg:simplewig!} we get a word with several points on it specified. We call these points marked points. We shall see later that these points denote the starting points of upward segments in the skeleton word. 

\noindent 
We now provide an algorithm to construct the skeleton word.
\begin{algorithm}\label{Alg:skeleton2}
The algorithm consists of the following steps.
\\
\textbf{Step 1} At first we follow the left most upward segment of the filled up path. This is the first upward segment of the skeleton path. Now at this point we use the value of $r_{i}$'s and the permutation of the non ignored vertices to get the location of the initial point of the first type $j\ge 2$ instant.\\
\textbf{Step 2} At this point there is at most three choices to continue the word:
\begin{enumerate}[(a)]
\item It can create an upward edge.
\item It can close the immediately traversed edge.
\item It can close one of the other two edges incident to the vertex of consideration. 
\end{enumerate}
In case (a) we move to \textbf{Step 1} and continue until we encounter the second type $j \ge 2$ instant.\\  
\textbf{Step 3} In cases (b) and (c) we start closing edges obeying the movement choices. We continue this until we encounter an instant where the movement choice directs the walk to go up. Now there can be two cases at this point. Firstly we can encounter an instant of moving upward while closing an edge left to right. Here we take follow the upward segment immediate right to the current level. Secondly we can encounter an instant of moving upward while closing an edge right to left. In this case we follow the upward segment starting from the current level which is the next available one from left to right.  After completion of \textbf{Step 3} we move to \textbf{Step 1} and move upwards until we encounter the next marked instant.
\end{algorithm}

\noindent
\textbf{(iii) Ignored and non-ignored instants:}
%Recall that we discussed about two kinds of repetitions of edges. We shall address these one by one separately.
%We have mentioned earlier that there can be two kinds of repetitions. Firstly, there can be type $j\ge 2$ instants which correspond to repetitions of edges and use the return to a level having the label of an endpoint of the edge. These type $j \ge 2$ instants require the skeleton word to return to a certain level at least twice. Here we construct the type $j\ge 2$ instant by going one step up from the endpoint of the corresponding downward chunk. Secondly, there can be type $j \ge 2$ instants which correspond to a repetition of edge but do not use the return to a level having the label of an endpoint of the edge. %This also requires the vertex corresponding to this type $j\ge 2$ instant to appear immediately after at least one arrival of the walk to the level of the endpoint of downward chunk before falling down.
\noindent
Our goal is to give an upper bound to the number of skeleton words. Hence we fix sufficient parameters (including the filled up words) which specify the skeleton word uniquely. 

\noindent
First of all instead of taking any arbitrary permutation of the type $j\ge 2$ instants we at first take a permutation of a subclass of type $j \ge 2$ instants. These instants will be called non-ignored instants. All the other instants will be called ignored instants. Our strategy is to at first take a permutation of the non-ignored instants and given this permutation we give an upper bound on the number of all permutations of the type $j \ge 2$ instants satisfying certain properties.  Which instants we take for the initial permutation depends on the other parameters in somewhat complicated way. Further when we want to use some properties of the permutation depending on the constraints, we shall assume that the properties hold for all the permutation. Mathematically this is formalized in the following way: Suppose we have a set $\mathcal{C}$ and we consider elements $x\in \mathcal{C}$ having some property $\varrho$. We consider the indicator function $\mathbb{I}_{\varrho}(x)$ to denote whether $x$ has property $\varrho$ or not. Let $f$ be a positive function of the properties. Then $\sum_{x \in \mathcal{C}} f(\varrho)\mathbb{I}_{\varrho}(x) \le f(\varrho) \# \mathcal{C}$ . 

\noindent
To implement this, we consider a skeleton word and look at the type $j \ge 2$ instants. We now tag some of the type $j$ instants as ignored and the others as non-ignored depending upon the imposed properties we shall discuss next.

\noindent
Given a skeleton word following is the strategy to classify the type $j\ge 2$ instants as non-ignored and ignored:
\begin{enumerate}
\item Firstly, if any type $j \ge 2$ instant is an endpoint of the first traversal of an edge, then tag it as non-ignored.
\item Now we look at a given level and the returns to this level. In this level we look at the repetitions of the first kind. We mark the instants just one step above the current level corresponding to the repetitions as ignored. Here we follow the strategy that when ever an ignored and non-ignored instant merge we mark that instant as ignored.
\item Finally for each repetitions of second kind we mark exactly one instant appearing as an endpoint of that edge as ignored. Note that the other instant might be marked as ignored as well while considering some other edge. Which one will be ignored will be described elaborately how we handle edges traversed odd number of times. Further, if an ignored and a non-ignored instant merge we mark that instant as ignored.  
\end{enumerate}
%For this purpose we additionally fix the values of $r_{i}$'s which denotes the initial locations of the type $j>2$ vertices.
\textbf{(iv) Counting the number of positions of the ignored instants}
Now we consider a fixed level $q$. By the filled up path we know the positions of returns to this level before falling down. Let there be $\Delta_{q}$ such indexes $ i_{1}< \ldots < i_{\Delta_{q}}$. We shall make the ignored instants such that apart from the first appearance of an edge all the instants appearing immediately after a return to the level $q$ are ignored instants.   

 %Given these suppose now a vertex corresponding to a type $j\ge 2$ instant comes $\kappa$ times just immediately after the positions of return to the level $l$. We shall ignore all such positions apart from the first. We also shall fix a permutations of the non ignored vertices. Fixing this permutation of the non ignored vertices we come to the choice of the ignored vertices.
 
%So we fix a level and look at the positions of returns to this level before falling down. Let these positions be $ i_{1}< \ldots < i_{\zeta}$. Suppose immediately after the returns to the level $l$, the $i$ th type $j$ instant appears $\eta_{}$
%Obviously $\sum_{i=1}^{\omega}\eta_{i}=\zeta$. Then fixing all the other parameters the number of choices for the ignored vertices is bounded by 

%Now let there be $\omega$ different type $j\ge 2$ instants repeated at the top of the upward segment starting immediately after these indexes. Let $\eta_{1},\ldots, \eta_{\omega}$ be the number of times they appear respectively.
 
%\begin{equation}\label{boundignored}
%\frac{\zeta^{\sum_{i}\eta_{i}}}{ \prod_{i=1}^{\omega} \eta_{i}!}.
%\end{equation}
Now for a given level $q$, suppose the filled up word comes to the level $\Delta_{q}$ times and among these $\Delta_{q}$ returns say the $\eta_{i,j,q}$ be the number of times the $i$ th vertex of type $j$ appears immediately after a return to $q$. Then the number of choices for such cases are given by 
\begin{equation}
\begin{split}
\prod_{q} \frac{\Delta_{q}^{\sum_{j,i}\eta_{i,j,q} }}{\prod_{j}\prod_{i}\eta_{i,j,q}!}
\end{split}
\end{equation}

%We now fix certain choices which we call the movement choices. These choices determined dynamically while forming the skeleton word. A choice of such kind determines what the skeleton word should do in the next step. Firstly at every instant it determines whether the word should create a new edge or close an existing edge immediately after an instant.  If at the next instant the word  closes an edge it determines exactly which edge should it close when there are multiple choices for the closing edge. One should also notice that if the word comes to a level for which the corresponding vertex appear only in the current level, then there is a unique movement choice for that instant.%(****** Will explain this later *****)

%We now fix $\mathcal{P}(N)$ and the parameters $p'_{i},q'_{i}$ and $r_{i}$'s. 

\noindent 
\noindent 
Before moving further, we now count the number of free parameters in the filled up word. We actually also take into consideration the fact that whenever we have a type $j\ge 2$ instant corresponding to repetitions of edges that uses an endpoint of at least one previous appearance of the edge, the corresponding value of $p_{i}$ is just $1$. By specifying the filled up path we have specified the starting points of each upward segment. So without any further constraints there are $2N-1$ free parameters. However multiple returns to a level before falling down decreases the number of free parameters. First of all we have seen that suppose the filled up word comes to a level $q$, $\Delta_{q}$ times, then there is a decrease of $\Delta_{q}$ free parameters. %This decreases the free parameters by $\Delta_{q}$.
 Now among these free parameters let there $\delta$ many positions corresponding to the ignored vertices. Then the corresponding values of $p_{i}$'s are just $1$. %If the non-ignored vertex is the first appearance of that vertex then the number of decrease in free parameters is same as the number of ignored vertices. Otherwise the decrease in free parameters is even smaller.
In other words the total decrease in free parameters are greater than or equal to $\Delta_{q} + \delta \ge 2\delta$.

Now along with the repetitions of discussed type, there can be another type of repetitions. Here the repeated edge does not use the level of an endpoint of at least one previous appearance of the edge. %We at first choose the permutation of the non-ignored instants in such a way that it covers first traversal of all the edges in the skeleton word. Now we consider the each repeated edge one by one. 
For every edge traversed more than thrice we shall consider only one of its endpoint as an ignored instant at that step and we shall look for a position to place the endpoint. Special care needs to be taken for edges traversed odd number of times where we also need to specify which endpoint of the edge we are considering as an ignored instant. We shall discuss this later in details.

 %We now provide a counting strategy to enumerate these words. 
%Firstly we count the number of edges repeated in such way. We look at the edges from left to right. Observe that here we are dealing with repeated edges which come as a collection of disjoint single edges. In particular any endpoint of such edges will not come to that level more than twice. All other cases have discussed previously. 
We at first fix a permutation of the non-ignored instants. Now we look at the repeated edges one by one and start filling up the ignored instants. 
We assume that for $i$ th type $j$ vertex there are $\eta_{i',j', i,j}$ many instants of the $i'$ th type $j'$ vertex appearing as ignored instants corresponding to the repetitions of the second kind. This gives us a choice of 
\begin{equation}
\begin{split}
\prod_{j}\prod_{i=1}^{N_{j}} \frac{j^{\zeta_{i,j}}}{\prod_{j'}\prod_{i'} \eta'_{i,j,i',j'}!}.
\end{split}
\end{equation}
Here $\zeta_{i,j}=\sum_{j'}\sum_{i'} \eta'_{i,j,i',j'}$.

Now we come to the reduction of free parameters due to this kind of repetitions. Suppose there are total $\delta'$ ignored instants, then the reduction of parameter is $2\delta'$. This is due to the fact that the corresponding lengths of upward chunks are $1$ and the downward chunks are $0$. In particular here also for each such ignored instants we fix two parameters for each such ignored instants.

\noindent 
\textbf{(v) The type $j$ instants where all the instants are ignored:} One need to consider the case when we have a vertex but all its appearances apart from the first appearance come as ignored instants slightly differently. We consider the cases for the skeleton word and Dyck paths separately. Here we deal with the skeleton word part. The Dyck paths will be dealt while we prove the bound for the Dyck paths.

 Firstly we consider the skeleton word. Here the ignored instants can correspond to either an edge repetition of first or second type. For the first type the word comes to specific level multiple number of times. Hence the instant just before the ignored instant is a beginning of an upward chunk. There are at most $N$ of them. Also the instant where the vertex of consideration appears for the first time will be within one step of an appearance of one of the verteices corresponding to the instants to which the vertex in consideration will be placed as an ignored instant. The same thing can be said for the repetitions of the second kind. Let there be $\psi$ many such type two instants. Now the possible positions of the first appearances of these instant will be a subset of size $\psi$ of a set with cardinality at most $4N$. So these positions can be chosen in at most $2^{4N}$ ways.
 
%For the Dyck paths we simply bound the total number of choices for the first position of the vertices by $k^{\psi}$ whenever the vertices are adjacent to a type $1$ vertex. However if the first positions are adjacent to type $j$ vertices, like before we bound it by $2^{4\Gamma}$ by following the arguments just given above.\\
\textbf{(vi) Reduction of count due to repetitions of edges:}
Next we need an argument before we state a bound for the number of skeleton words keeping the repetitions of edges in mind. This argument is for the number of choices while closing the edges corresponding to the repetitions. Here we fix an edge $e$ with endpoints $\{ e_{1}, e_{2} \}$. Suppose it is repeated $r_{e}$ times. Then the edge $e$ is traversed $\left[\frac{r_{e}}{2}\right]$ times in the down ward direction. Among these downward traversals let $r_{e,e_{1}}$ denotes the number of times the edge is closed in the order $(e_{1},e_{2})$ and $r_{e,e_{2}}$ be the number of times the edge is closed in the order $(e_{2},e_{1})$. Clearly $r_{e,e_{1}}+ r_{e,e_{2}}= \left[\frac{r_{e}}{2}\right]$. Now we recall the construction of the skeleton word. While closing edges we considered all possible choices to close an edge. This gives rise to the $(3j)^{j}$ factor. However all these choices might not lead to a feasible word. For example when we consider the edge $e$ which repeated $r_{e}$ times, these downward traversals are ordered. In particular the first downward traversal closes the first upward traversal, the second downward closes the second upward traversal and so on. However this ordering is not preserved while we calculate all the possible choices for closing the edges. This reduces the word count. In particular, every feasible word can be mapped to $\prod_{e} r_{e,e_{1}}! r_{e,e_{2}}!$ many infeasible words in the following way. We start the skeleton word from the starting point and we continue until we hit the first instant where an edge is traversed multiple times is closed. Say this edge is $e$ and we close the edge in the order $(e_{1},e_{2})$. Now we choose any one from the $r_{e,e_{1}}$ closings of the edge and close that edge. Next we continue until hit the next such instant and so on. Observe that traversals in these fashion may not give the complete word since we can run out of choices to continue before all edges are traversed. In this case we start with the left most edge which is not traversed and there is no other edge which comes immediately before that edge if we include that edge in any word. Although these words are not feasible but we have counted them. Now $r_{e,e_{1}}! r_{e,e_{2}}! \ge \frac{1}{2}^{\left[\frac{r_{e}}{2}\right]} \left[\frac{r_{e}}{2}\right]!$. We get the actual feasible word count is less than $2^{N} \prod_{e} \frac{1}{\left[\frac{r_{e}}{2}\right]!}$ times the count we introduced in \eqref{countskeleton}.  

\noindent 
\textbf{(vii)Handling the edges traversed odd number of times:}
We find the argument in this part to be the most difficult. We break it into several steps.

\noindent
\textbf{Step 1.(Choices for closing an edge)} We start with a basic but fundamental observation. To begin with we fix the filled up path. Suppose an edge is traversed multiple number of times in the skeleton word, then it can be traversed odd number of time only under the following condition. If we look at the filled up word, then in the skeleton word all but the final exploration of the edge is covered. If this is not the case, then the given filled up path has to be modified. Since both the filled up word and the skeleton word are closed and the skeleton word is obtained by removing some edges from the exploration of the filled up word the removed word is a collection of closed words. Further as we just discussed, no removed edge can be traversed more than once, the removed edges are disjoint collection of cycles. Hence number of edges incident to each vertex is even. Now we look at a vertex. Say this vertex is of type $j\ge 2$ for some $j$ and there are $2\delta$ many removed edges incident to that vertex. By looking at the argument where we bounded the number of choices for closing an edge for each vertex, we see that in this case the number of choices for closing the edges for this vertex is bounded by $(3j)^{j-\delta}$ instead of $(3j)^{j}$.  

\noindent
\textbf{Step 2. (A bound to the number of edges repeated odd number of times specific to the return to a level):}Now we prove a few important facts about edges traversed odd number of times. We consider a level $q$ and assume that the filled up path comes $\Delta_{q}$ number of times before falling down. In this case we claim that there can be at most two instants among the $\Delta_{q}$ returns to the level $q$, such that immediate upward edge after a return to level $q$ is traversed but it is never closed. Further this position is uniquely determined by the movement choices. This is due to the following observation. In the skeleton word suppose we look at the returns to the level $q$ before falling down. Then among these returns there can be only one return after closing an edge from right to left. All the other returns will be by closing an edge left to right. Now in order to create an upward edge from the level $q$ the walk has to return to the level $q$. Further if the return is from left to right the only choice for creating the upward edge is the upward edge is placed right next to the closed edge. As a consequence the only possibility that an immediate upward edge after a return to level $q$ is traversed but it is never closed is the upward edge right before the walk returns to the level $q$ from right to left. However in this case observe that the edge corresponding to the last return to $q$ might never been closed. One might look at the following figure for further insights. 
\begin{figure}[H]
        \begin{center}
                \includegraphics[width=0.5\textwidth]{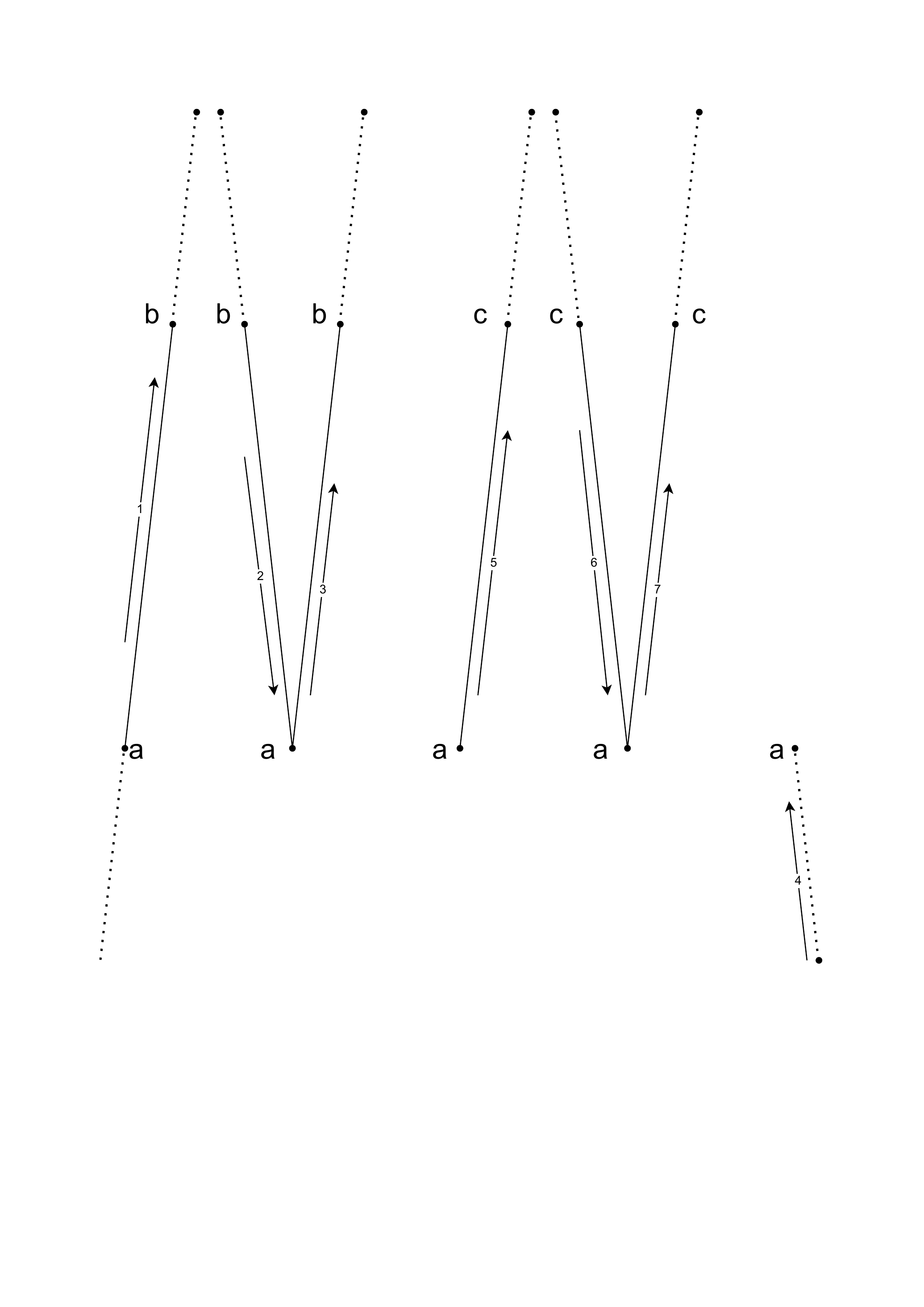}
        ~ %add desired spacing between images, e. g. ~, \quad, \qquad etc.
          %(or a blank line to force the subfigure onto a new line)
      \end{center}   
 \end{figure}
Here the labeled arrows denote the order and the direction along which the edges are traversed.

\noindent
For the shake of the proof we need to consider the edges which are traversed odd number of times, the last repetition is a repetition of first kind and the edge appears exactly once in the corresponding level needs to be dealt a little differently. Observe that to each level there exist at most two such edges. As these edges are repeated exactly once in the corresponding level, these repetitions enjoys all the properties of an edge repetition of second kind. So we consider these edges as repetitions of second kind. Please refer to the following figure.
\begin{figure}[H]
        \begin{center}
                \includegraphics[width=0.5\textwidth]{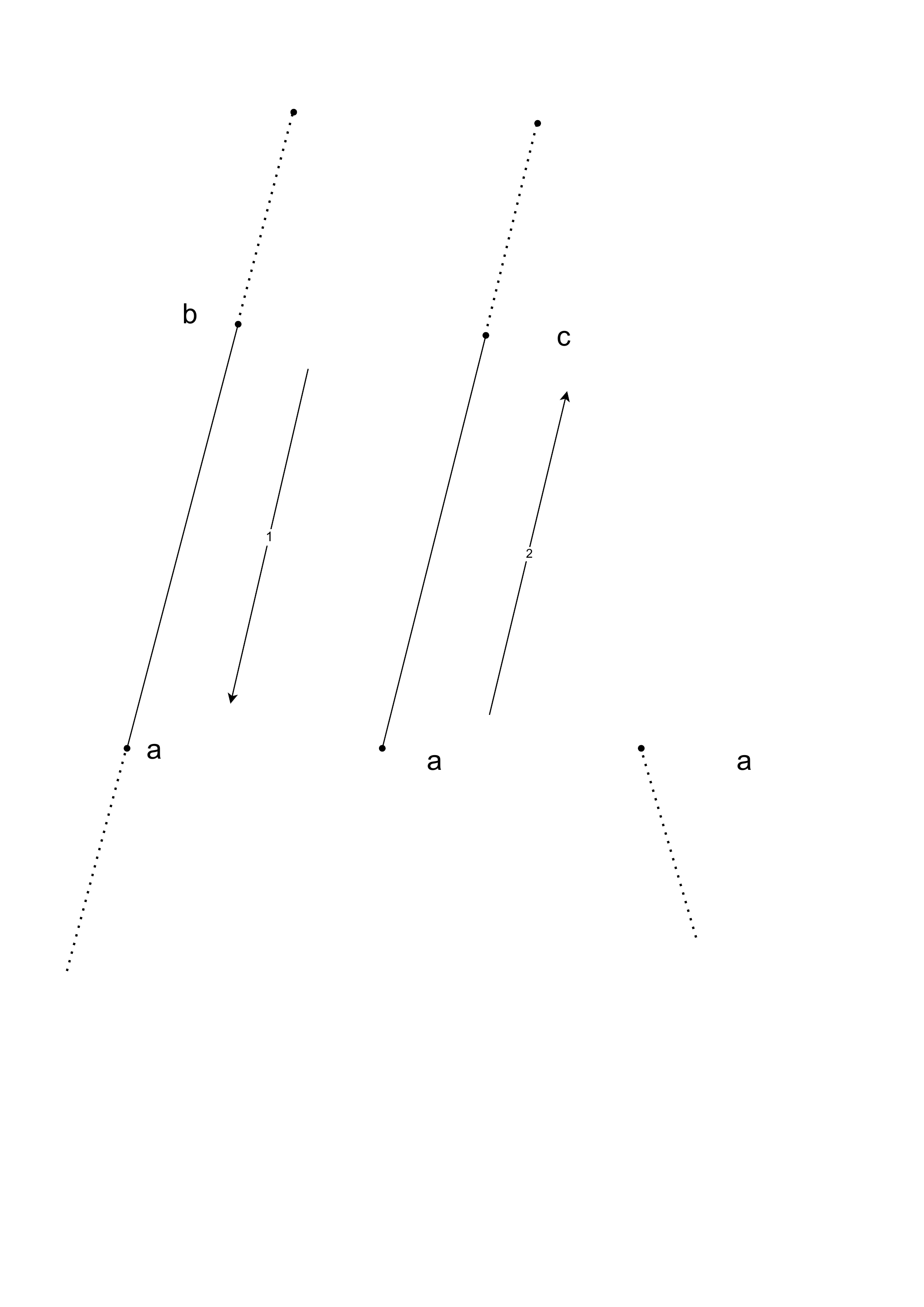}
        ~ %add desired spacing between images, e. g. ~, \quad, \qquad etc.
          %(or a blank line to force the subfigure onto a new line)
      \end{center}   
 \end{figure}
 Here the labeled arrows denote a possible exploration of the Euler circuit.
\noindent
\textbf{Step 3.(A strategy to tag instants corresponding to the repetitions of second kind as ignored and non-ignored)} Now we spend some time on specifying the parameters in counting the skeleton words. By specifying the parameters which determine the skeleton word, we make a function from the space of the parameters to the space of the skeleton words. Now given the filled up word, the permutations of type $j$ vertices, the movement choice and other parameters the skeleton word is uniquely determined. However in general there is no reason for this function to be one to one. This give us some liberty to describe some relationship between the parameters. In particular, we shall impose some constraints on the number of non-ignored vertices. First of all if we look at a given level $l$ and consider the multiple ($\ge 2$) returns of the filled up word to that level, then the vertices just one step above the level which are connected to the level have to be ignored vertices. Now we come to the other type of edge repetitions. Here we have some options. For example we consider the $i$ th vertex of type $j$ and look at the number of edges this vertex appeared as an endpoint of an oddly traversed edge. We have argued that this has to be even. Now among these edges some will be repetitions of the first kind and the others will be repetitions of the second kind. For the first kind we have no choices as we have just discussed. However, for the repetition of the second kind we have some options since the both endpoints of the such edges are symmetric. Here we mark the ignored and non-ignored  vertices in such a way that the difference between the number of times the that vertex appears as a non-ignored vertex and the number of times the vertex appears as an ignored vertex is at most $1$. Given the skeleton word, this is done in the following simple way. For every vertex of type $j\ge 2$ we mark it as ignored or non-ignored. First of all some non-ignored vertices are fixed due to the repetitions of the first kind.
As we know that if we consider the edges traversed odd number of times then to each vertex there are even number of such edges incident to it. As a consequence if we look at the graph corresponding to these edges, they will be a disjoint union of several disjoint Euler circuits. Observe that here we can choose any Euler circuit.
In particular if our edge set is such that a chunk of consecutive edges belong to the edge set of consideration, then we include them in the same Euler circuit and travel the edges consecutively. 
Further if we encounter a situation where we have repetitions of the first kind but the edges of consideration corresponding to that level appears exactly once (these edges are considered as repetitions of second kind now), then we consider the upward chunks containing the edges and corresponding to the level as consecutive and traverse these edges consecutively while considering the Euler circuits.  One might observe that whenever there is two such edges corresponding to a level, the edge through which the word enters the level is closed. So there will be no problem defining the segments as consecutive. %The following figure explains this scenario. *****(Insert figure)*****.

Now we fix one such Euler circuit and consider the traversal of the Euler circuit. We now divide the edges in two parts. Firstly the edges for which the edge coming immediately after is an edge corresponding to repetition of first kind. The other edges are the edges for which the edge coming immediate after is an edge corresponding to repetition of second kind. For the edges corresponding to repetition of the second kind we denote the instant corresponding to the second endpoint of the next edge as an ignored instant. However for edges corresponding to repetition of first kind we do not have any choices for the second endpoint of the next edge. %Following figure explains this strategy. *******(insert figure)******

\noindent 
\textbf{Step 4. (Double assignments) }
Here we consider two things. Firstly here also our tagging strategy can be such that we tag one instant ignored and non-ignored at the same time. Here we mark the instant as ignored. On the other hand there might be cases when one instant is tagged ignored more than once (at most twice). Observe that these scenarios only happen in the case of edges repeated odd number of times and at least one of the assignment comes from a repetition of the second kind where we are forced to tag some vertices as ignored. We have just discussed that whenever we encounter a double tagging such that one of the tagging is ignored and one of the tagging is non-ignored we consider the instant as ignored. On the other hand whenever both the tagging are ignored, there can be two cases: one of the ignored instant corresponds to a repetition of the first kind and both the ignored instants correspond to repetitions of second kind. However when both the ignored instants correspond to repetitions of second kind, we have the edges are consecutive. Hence by our construction of the Euler circuit these cases do not exist. Finally when one of the edge repetition is of the first we ignore the edge and consider the instant as an ignored instant corresponding to the edge repetition of second kind.

 %In this case we shall consider that instant as an ignored instant to the vertex for which it is counted for the first time. In this cases some additional parameters will be fixed. In particular $2$ for every such cases. This will help us to get the required bound. We shall denote $x$ to be the total number of cases of double assignment. Finally the type $j$ non-ignored instants that will be fixed in this way is just a subset of all the type $j$ instants. Hence the cardinality of such cases are bounded by $2^{N}$.

%On the other hand as we choose the Euler circuits such that we include consecutive edges consecutively while tagging the vertices ignored, we shall not encounter a situation where we have to mark the same instant ignored at two different situations when both the consecutive edges are repetitions of second kind. Further if two conse  

\noindent 
\textbf{Step 5.(Upper bounds to the number of times a vertex appears at an ignored instant)}
Now given a permutation of such kind we enumerate the number of times a vertex appears at an ignored instant. At first we consider the edges in the first part. Suppose for the $i$ th type $j$ vertex there are $2\delta'_{i,j}$ many edges incident to it with the following property. When we consider the Euler circuit, both the edges entering and going out of the vertex belong to the repetition of the first kind. We call this set to be the edges of type $I$. Similarly define $2\delta''_{i,j}$ to be the number of edges where both the edges entering and going out of the vertex belong to the repetition of the second kind. We call this set to be the edges of type $II$. Now there can be two other types of edges. Firstly there can be edges incident to a vertex such that when the edge enters a vertex according to the Euler circuit it is a repetition of the first kind and the edge going out of the vertex is a repetition of the second kind. We call this set to be the edges of type $III$. Similarly the exact opposite of it can also happen which will be called type $IV$. We denote the number of such edges by $2\kappa''_{i,j}$ and $2\kappa'_{i,j}$ respectively. We now give an upper bound to the number of times the given vertex appears at an ignored instant. Firstly observe that among the edges of type $I$  we don't have a choice for which vertex is ignored and which is not ignored. So for edges of this type we put an upper bound of $2\delta'_{i,j}.$ On the other hand for edges of type $II$ we put an upper bound of $\delta'_{i,j}$ instead of $2\delta'_{i,j}$. Now for edges of type $III$ we upper bound it by $\kappa''_{i,j}$. Since here we have the edge coming out of the vertex is a repetition of second kind. So we don't consider that instant as new ignored instant. On the other hand the edge entering to the vertex can  correspond to a new ignored instant for which we do not have any choice. Finally for type $IV$ we bound it by $2\kappa'_{i,j}$ as the edge entering the vertex is a repetition of second kind so we count that instant as a new ignored instant and the edge coming out of the vertex is a repetition of first kind so it is possible that that instant is an ignored instant too. So our upper bound is $2\delta'_{i,j}+2\kappa'_{i,j}+ \delta''_{i,j}+ \kappa''_{i,j}$. 

%Now we spend some time on count of the permutation of the type $j$ vertices formed by the discussed procedure. 

We now count the number of permutations of the non-ignored vertices. There are $N-L$ many of them and assume that $N'_{j}$ of them are of type $j$. So that $\sum_{j}(j-1)N_{j}'= N-L$. Further we assume that for the $i$ th vertex of type $j$, $L_{i,j}$ many of them are non-ignored. So permuting all these gives us a choice of 
\begin{equation}\label{eq:boundprelim}
\begin{split}
\frac{(N-L)!}{\prod_{j}N_{j}'! \prod_{j}\prod_{i=1}^{N_{j}}(j-L_{i,j})!}.
\end{split}
\end{equation}
Compiling these factors and the arguments given before we have the following upper bound to the number of skeleton words.

\begin{lemma}
Fixing the parameters we just discussed, the number of skeleton words is bounded by 
\begin{equation}\label{eq:modskebound}
\begin{split}
& n^{\sum_{i}p_{i}- N +1}C^{N}\prod_{e} \frac{1}{\left[\frac{r_{e}}{2}\right]!} \frac{(N-L)!\prod_{j}\prod_{i=1}^{N_{j}}j^{j-\delta_{i,j}}}{\prod_{j}N'_{j}! \prod_{j}\prod_{i=1}^{N_{j}}(j-L_{i,j})!}\prod_{q} \frac{\Delta_{q}^{\sum_{j,i}\eta_{i,j,q} }}{\prod_{j}\prod_{i}\eta_{i,j,q}!}\prod_{j}\prod_{i=1}^{N_{j}} \frac{j^{\zeta_{i,j}}}{\prod_{j'}\prod_{i'} \eta'_{i,j,i',j'}!}\times \\
&\frac{(2m)^{2N-2L-2+ \sum_{j}N_{j}'}}{(2N-2L-2)!}.
\end{split}
\end{equation} 
\end{lemma} 
\begin{proof}
The proof of this lemma is essentially compiling the arguments given after the statement of Proposition \ref{prop:ge3} until now. So we omit the details.
\end{proof} 

\noindent
\textbf{(viii) Counting the number of free parameters:}
Here we show that the number of remaining parameters are bounded by $c^{N}$ for some given $N$. In order to do this, we at first give a step by step algorithm to form the permutation of type $j$ instants. This algorithm is useless if there is no multiple assignment of tags. However if there is multiple assignments this algorithm helps us to count the remaining free parameters.
\begin{algorithm}\label{alg:freeparameter}
The algorithm starts with the permutation of the non-ignored vertices. It then consists of the following steps:\\
\textbf{Step 1:} Once the non-ignored vertices are fixed, we look at the existing edges among the non-ignored vertices and place the ignored vertices   adjacent to the non-ignored vertices corresponding to repetitions of the edges among the existing edges. We continue this as long as we can.\\
\textbf{Step 2.} In \textbf{Step 1} we create some new vertices and edges. We now look at the newly created vertices in \textbf{Step 1} as candidates for placing the newer ignored vertices adjacent to. We continue as long as we can and move to \textbf{Step 3}(subsequently \textbf{Step 4} etc.) with the same strategy but with the vertices of \textbf{Step 2}(subsequently \textbf{Step 3} etc.) as candidates for placing the newer ignored vertices adjacent to. We continue this procedure and finish when all the edges have been covered.
\end{algorithm}
However observe that given a word, the initial set of vertices we chose as candidates for being non-ignored might not all turn up to be non-ignored at the end. Due to the multiple assignments some non-ignored vertex might be tagged as ignored. So one might think that there might be some cases where the \textbf{Step 1} of Algorithm \ref{alg:freeparameter} might not be possible. However this is not the case. Observe that whenever we assign a position to an ignored instant, it corresponds to a repetition of an edge which previously existed in the skeleton word. We call this edge to be the predecessor of the current edge. We now look at the first non-ignored instant which is tagged as ignored. The edge corresponding to this ignored instant has a predecessor. We continue finding the predecessors until we are unable to find one. In this case both the instants corresponding to that edge is non-ignored. As a consequence, there is at least one edge among the existing edge sets of non-ignored instants which is repeated. Hence \textbf{Step 1} is always feasible.

As every ignored instant corresponds to a type $j\ge 2$ instant, we at first introduce a partition of $N$. Since the number of steps  to conclude Algorithm \ref{alg:freeparameter} is at most $N$, this can be done in $2^{2N}$ ways. In particular we have $\sum_{i} \varsigma_{i}=N$ where $\varsigma_{i}\ge 0$ are integers. Here $\varsigma_{i}$ denotes the number of  ignored instants corresponding to the repeated edges which were created at \textbf{Step $i$} of Algorithm \ref{alg:freeparameter}. Once $\varsigma_{i}$'s are fixed, we look at the set $\mathcal{E}_{i}$ denoting the repeated edges which were created at \textbf{Step $i$}. We again partition $\varsigma_{i}$ in $\mathcal{E}_{i}$ many groups. We call the cardinality of the $j$ th group $\varsigma_{i,j}$. Here we have $\sum_{j=1}^{|\mathcal{E}_{i}|} \varsigma_{i,j}=\varsigma_{i}$. As $\varsigma_{i,j}$'s can be $0$ as well, we bound this by $2^{\varsigma_{i}+ |\mathcal{E}_{i}|}$. Now for each edge there can be at most two choices determining which vertex appears as an ignored instant. Hence this contributes another $2^{\sum_{i,j} \varsigma_{i,j}}=2^{\varsigma_{i}}$ to the factor. As a consequence, the total number of choices is bounded by $2^{2N + \sum_{i} \varsigma_{i}+ |\mathcal{E}_{i}|+ \sum_{i}\varsigma_{i}}\le 32^{N}$.

We now have the enough machineries to prove Proposition \ref{prop:ge3}. 
\begin{proof}[Proof of Proposition \ref{prop:ge3}]
The proof is divided into two parts. In the first part we fix the choices in the associated Dyck paths and calculate the quantities for the skeleton word. In the second part we consider the repetitions of edges in the Dyck paths. 

%Before going in to the \textbf{step 1} of the proof, we observe that there can be edges which are present in both the skeleton word and the Dyck paths. This leads us to modify the definition of skeleton words. In this modified definition we consider the minimal word which contains all the edges in the skeleton word and the edges in Dyck path which are common between the skeleton word and the Dyck paths and we associate the corresponding exploration order. We call it the modified skeleton word. Observe that in the modified skeleton word all the edges are repeated at least thrice. This modified skeleton word now has a different parameter $N$ and $m$. However one might note that all the properties derived for the skeleton word remain valid. So from now on we shall call the modified skeleton word, the skeleton word and define the parameters accordingly. 

\noindent
Now we go into \textbf{Step 1} of the proof.

\noindent 
\textbf{Step 1}
Here we start with the skeleton word. There are $\left(\sum_{i=1}^{N}p_{i}+q_{i}\right)=: m+m'$ many edges in the skeleton word. Associated to the endpoints of these edges there will be $m+m'+1$ many Dyck paths. Let their lengths be $2k_{1},\ldots, 2k_{m+m'+1}$ respectively. So that we have $\sum_{i=1}^{m+m'+1} 2k_{i}+ m+m'= k$. We shall show in \textbf{step 2} the factor coming from the Dyck paths in \eqref{eq:wigge3} is $o(n^{\sum_{i=1}^{m+m'+1}k_{i}})$ whenever there is at least one edge in the Dyck path traversed at least four times. 

Now we calculate the total number of choices of the Dyck paths. By \eqref{def:mthcatalan} this is exactly equal to $\frac{m+m'+1}{k+1}\binom{k+1}{\frac{k+m+m' +2}{2}}$. Since by our assumption the  skeleton word and the Dyck paths have disjoint edges for any word $w$, $\E[X_{w}]= \E[X_{S(w)}]\E[X_{D(w)}]$. Here  $S(w)$ is the skeleton word and $X_{D(w)}$ denotes the corresponding random variables in the Dyck paths. 

The main goal of this step is to show that 
\begin{equation}\label{eq:skeletonrep}
\sum_{S(w)} \E[X_{S(w)}] \frac{m+m'+1}{k+1}\binom{k+1}{\frac{k+m+m'+2}{2}}= 2^{k-m-m'}o\left(n^{\frac{m+m'}{2}}\right)
\end{equation}
whenever the words $S(w)$ varies over all skeleton words with $m+m'$ number of edges, every edge traversed at least twice and some edge is traversed at least thrice. Suppose an edge $e$ is traversed $r_{e}$ number of times in the skeleton word. Then by (iv) of Assumption \ref{ass:wig} we have 
\begin{equation}
\E[X_{S(w)}]\le \frac{1}{2^{m+m'}}C^{N}\prod_{e} \left[ \frac{r_{e}}{2} \right]!.
\end{equation}
Now we fix all the required parameters of the skeleton word. Our main calculation tool is \eqref{eq:modskebound}. By \eqref{eq:modskebound} we have 
\begin{equation}\label{eq:finalskeleton}
\begin{split}
&\sum_{S(w)} \E[X_{S(w)}] \frac{m+m'+1}{k+1}\binom{k+1}{\frac{k+m+m'+2}{2}}\\
&\le \frac{1}{2^{m+m'}}\sum_{\text{parameters}}n^{m- N +1}C^{N}\\
&~~\prod_{e} \frac{1}{\left[\frac{r_{e}}{2}\right]!} \frac{(N-L)!\prod_{j}\prod_{i=1}^{N_{j}}j^{j-\delta_{i,j}}}{\prod_{j}N'_{j}! \prod_{j}\prod_{i=1}^{N_{j}}(j-L_{i,j})!}\prod_{q} \frac{\Delta_{q}^{\sum_{j,i}\eta_{i,j,q} }}{\prod_{j}\prod_{i}\eta_{i,j,q}!}\prod_{j}\prod_{i=1}^{N_{j}} \frac{j^{\zeta_{i,j}}}{\prod_{j'}\prod_{i'} \eta'_{i,j,i',j'}!} \times \\
&~~ \frac{(2m)^{2N-2L-2+ \sum_{j}N_{j}'}}{(2N-2L-2)!}\frac{m+m'+1}{k+1}\binom{k+1}{\frac{k+m+m +2}{2}} \prod_{e} \left[ \frac{r_{e}}{2} \right]!.
\end{split}
\end{equation}  
We analyze \eqref{eq:finalskeleton} step by step. 
We at first cancel the $\prod_{e}\left[ \frac{r_{e}}{2} \right]!$ from numerator and denominator. Next we cancel the term $\prod_{j}\prod_{i=1}^{N_{j}}j^{j-\delta_{i,j}}$ in the numerator with the term $\prod_{j}\prod_{i=1}^{N_{j}}(j-L_{i,j})!$ as far as we can and we bound the remaining term by $\prod_{j}\prod_{i=1}^{N_{j}}j^{(L_{i,j}-\delta_{i,j})}$. By Lemma \ref{lem:universal}, we know that $\Delta_{q}$'s can be dominated by i.i.d. sub-exponential random variables under of uniform measure all simple symmetric random walks of length $2N$. On the other hand the total number of non-crossing partition is bounded by $4^{N}$.  So the term
\begin{equation}
\sum_{\mathcal{P} \in NC(N)}\prod_{q}\Delta_{q}^{\sum_{i}\eta_{i,q}}\le 4^{N}\prod_{q}\left( \sum_{i} \eta_{i,q}\right)!.
\end{equation} 
Our next task is to fix the parameters $N$, $L_{i,j}$, $N_{j}$, $\omega_{q}$, $\eta_{i,q}$'s and take the following sum
\begin{equation}\label{eq:wordcountrepgen}
n^{-N+1}\sum_{m\ge N}\frac{(N-L)!}{\prod_{j}N'_{j}!}\frac{(2m)^{2N-2L-2+ \sum_{j}N_{j}'}}{(2N-2L-2)!}\frac{m+m'+1}{k+1}\binom{k+1}{\frac{k+m+m' +2}{2}}
\end{equation}
By the same arguments as the arguments given to bound \eqref{eq:totwordgensim}, we have \eqref{eq:wordcountrepgen} is bounded by 
\begin{equation}\label{eq:summgen}
\left(  \frac{1}{n}\right)^{L}2^{k}\exp\left( - \frac{(N-L)\log (N-L)}{2} \right).
\end{equation}
Observe that $\exp\left( - \frac{(N-L)\log (N-L)}{2} \right)$ is bounded by $C^{N}\left( \frac{N-L}{2} \right)!$.
Now we write 
\begin{equation}
\begin{split}
& \left( \frac{N-L}{2} \right)!= \frac{\frac{N}{2}!}{\left( \frac{N-L}{2} +1 \right) \ldots \frac{N}{2}}\ge \frac{\frac{N}{2}!}{\left(\frac{N}{2}\right)^{\frac{L}{2}}}\\
& \Rightarrow \exp\left( - \frac{(N-L)\log (N-L)}{2} \right)\le \exp\left( -\frac{N\log N}{2} + \frac{L\log N}{2}\right).
\end{split}
\end{equation}
So \eqref{eq:summgen} is bounded by $2^{k}\exp\left( - L \log k - \frac{N\log N}{2}\right).$
We plug this in \eqref{eq:finalskeleton} to get the following reduced form: 
\begin{equation}\label{dot}
\begin{split}
&\sum_{\text{parameters}}2^{k-m-m'}C^{N} \prod_{j}\prod_{i=1}^{N_{j}}j^{L_{i,j}-\delta_{i,j}}\prod_{q} \frac{\left(\sum_{j,i}\eta_{i,j,q}\right)! }{\prod_{j}\prod_{i}\eta_{i,j,q}!}\prod_{j}\prod_{i=1}^{N_{j}} \frac{j^{\zeta_{i,j}}}{\prod_{j'}\prod_{i'} \eta'_{i,j,i',j'}!}\\
&\exp\left(  - L \log k - \frac{N\log N}{2}   \right)\\
&= \sum_{\text{parameters}}2^{k-m-m'}C^{N} \prod_{j}\prod_{i=1}^{N_{j}}j^{L_{i,j}-\delta_{i,j}}\prod_{q} \frac{\left(\sum_{j,i}\eta_{i,j,q}\right)! }{\prod_{j}\prod_{i}\eta_{i,j,q}!}\prod_{j}\prod_{i=1}^{N_{j}} \frac{j^{\zeta_{i,j}}}{\prod_{j'}\prod_{i'} \eta'_{i,j,i',j'}!}\\
& \exp\left(  - L \log k - \frac{N\log N}{2}  \right)
\end{split}
\end{equation}
We now focus on the term 
\[
\prod_{j}\prod_{i=1}^{N_{j}}j^{L_{i,j}-\delta_{i,j}}\prod_{q} \frac{\left(\sum_{j,i}\eta_{i,j,q}\right)! }{\prod_{j}\prod_{i}\eta_{i,j,q}!}\prod_{j}\prod_{i=1}^{N_{j}} \frac{j^{\zeta_{i,j}}}{\prod_{j'}\prod_{i'} \eta'_{i,j,i',j'}!}
\]
%Firstly in the product $\prod_{j}\prod_{i=1}^{N_{j}} j^{L_{i,j}-\delta_{i,j}}$, we ignore the factor coming from $j\le 1000$. Since this will give only a $C^{N}$ contribution.

Now recall the discussion we had the edges traversed odd number of times. We divide the parameter $2\delta_{i,j}$ into four parts $2\delta'_{i,j},2\delta''_{i,j},2\kappa'_{i,j},2\kappa''_{i,j}$. Here $2\delta_{i,j}= 2\delta'_{i,j}+2\delta''_{i,j}+2\kappa'_{i,j}+2\kappa''_{i,j}$. Now among $\eta_{i,j,q}$'s and $\eta'_{i,j,i',j'}$'s some will correspond to the odd number of repetitions. In these cases we replace $\eta_{i,j,q}$ by $(\eta_{i,j,q}-1)$ and $\eta'_{i,j,i',j'}$ by $(\eta'_{i,j,i',j'}-1)$ in the factorial of the denominator. To denote this mathematically we introduce parameters $\tau_{i,j,q}$ and $\tau'_{i,j,i',j'}$. This parameter is either $1$ or $0$ where $1$ means the corresponding edge is traversed odd number of times and $0$ means the corresponding edge is traversed even number of times. 

For the double assignments we have some repetitions of the first kind to be not considered. Among these some edges might be covered odd number of times. We have argued earlier, if both of these happen, then at that level the edge of consideration comes at least thrice implying $\eta_{i,j,q}\ge 1$ among considered cases. We replace $\eta_{i,j,q}$ by $\eta_{i,j,q}-1$ as well. 

Now observe that the $i$ th type $j$ vertex appears as the endpoint other than the ignored one for repetitions of second type $\kappa''_{i,j}+ \delta''_{i,j}$ times. From the factor $j^{\zeta_{i,j}}$ we take out $\delta''_{i,j}+\kappa''_{i,j}$ th power of $j$.

As a consequence, our factor reduces to 
\begin{equation}\label{1}
j^{\delta''_{i,j}+\kappa''_{i,j}}\frac{j^{\zeta_{i,j}-\delta''_{i,j}-\kappa''_{i,j}}}{\prod_{j'}\prod_{i'} \left(\eta'_{i,j,i',j'}-\tau'_{i,j,i',j'}\right)!}
\end{equation} 

We have proved that for repetitions of the first type at a given level $q$ there can be at most two instants with edges traversed odd number of times. So we take out $\left(\sum_{j,i}\eta_{i,j,q}\right)^2$ from $\left(  \sum_{j,i}\eta_{i,j,q}\right)!$ and bound the squared term by $2^{\sum_{j,i}\eta_{i,j,q}}$. So our factor is lesser than or equal to 
\begin{equation}\label{2}
2^{\sum_{j,i}\eta_{i,j,q}} \frac{(\sum_{j,i}\eta_{i,j,q}-\sum_{j,i}\tau_{i,j,q})!}{\prod_{j}\prod_{i}\left(\eta_{i,j,q}-\tau_{i,j,q}\right)!}
\end{equation}

Now we apply the following fact about multinomial coefficient. Suppose we have $\tau$ numbers $\gamma_{1},\ldots, \gamma_{\tau}$ and for each $\tau$, $\gamma_{\tau}$ is partitioned into $N$ many groups. Let $\gamma_{t,\kappa}$ be the frequency of $\kappa$ th group. Then 
\begin{equation}
\prod_{t=1}^{\tau}\frac{\gamma_{t}!}{\prod_{\kappa=1}^{N}\gamma_{t,\kappa}!}\le \frac{\left(\sum_{t}\gamma_{t}\right)!}{\prod_{\kappa=1}^{N}(\sum_{t}\gamma_{t,\kappa})!}.
\end{equation} 
Applying this to \eqref{1} and \eqref{2} we arrive at the following quantity.
\begin{equation}
\begin{split}
&\prod_{j}\prod_{i}j^{\delta''_{i,j}+\kappa''_{i,j}}\frac{j^{\zeta_{i,j}-\delta''_{i,j}-\kappa''_{i,j}}}{\prod_{j'}\prod_{i'} \left(\eta'_{i,j,i',j'}-\tau'_{i,j,i',j'}\right)!}\prod_{q}2^{\sum_{j,i}\eta_{i,j,q}} \frac{(\sum_{j,i}\eta_{i,j,q}-\sum_{j,i}\tau_{i,j,q})!}{\prod_{j}\prod_{i}\left(\eta_{i,j,q}-\tau_{i,j,q}\right)!}\\
& \le \prod_{j}\prod_{i} j^{\delta''_{i,j}+\kappa''_{i,j}} 2^{N} \frac{(L-\delta)!}{\prod_{j}\prod_{i=1}^{N_{j}}(L_{i,j}-\varphi_{i,j})!}.
\end{split}
\end{equation}
Here $\varphi_{i,j}\le 2\delta'_{i,j}+ 2\kappa'_{i,j}+ \delta''_{i,j}+\kappa''_{i,j}$.
%Here $\varphi_{i,j}$ is the parameter denoting the number of time $i$ th type $j$ vertex appears as an endpoint of an edge traversed odd number of times and type $2$ non-ignored or an ignored vertex simultaneously. Now 
%$\varphi_{i,j}= \varphi_{i,j}'+ \varphi_{i,j}''$ where $\varphi_{i,j}'$ denotes number edges traversed odd number of time with one  endpoint to be the $i$ th type $j$ vertex and the final upward edge is a repetition of the first type. $\varphi''_{i,j}$ denotes the other case. Observe that $\varphi''_{i,j}= \delta''_{i,j}$ and $\varphi'_{i,j}\le 2\delta'_{i,j}$. 

So now we have the following reduction
\begin{equation}
\begin{split}
\frac{\prod_{j}\prod_{i=1}^{N_{j}} j^{L_{i,j}-\delta_{i,j}+ \delta''_{i,j}+\kappa''_{i,j}}}{\prod_{j}\prod_{i=1}^{N_{j}}(L_{i,j}-\varphi_{i,j})!}\le C^{N} \prod_{j}\prod_{i=1}^{N_{j}} j^{\delta'_{i,j}+\delta''_{i,j}+\kappa'_{i,j}+\kappa''_{i,j}}= C^{N}\prod_{j}\prod_{i} j^{\delta_{i,j}}. 
\end{split}
\end{equation}
One might note that this equation is also consistent for the double assignments. This is because for the double assignments we possibly ignored some edges traversed odd number of times. However, when we replace $(\eta_{i,j,q})$ by $(\eta_{i,j,q}-1)$ for double assignments we automatically consider the edge and this adjusts in the $\delta_{i,j}$'s.

\noindent 
%One thing to note that the bound $j^{\delta_{i,j}}$ can be further strengthened. Here we consider the case when we have an edge repetition of the first kind and both the ignored and non-ignored instants are type $j\ge 2$ instants. As one of the type $j$ instant is considered as non-ignored, for this particular edge and the non-ignored instant we do not do any deduction in denominator. 
%
%\noindent 
%For this refer to the following figure and explanation:
%*****(insert figure)*****
%
%Here we consider four consecutive levels $a$, $b$, $c$ and $a$ with the oddly traversed edges are the triangle for these three vertices. Further the edge $\{ b,c \}$ is a repetition of first kind. We also assume $b$ is an ignored instant of the first $a$, $c$ is an ignored instant of $b$ and the second $a$ is an ignored instant of $c$. So all the vertices come as an ignored instant once. However the edge $\{ b,c \}$ is a repetition of first kind so there will not be any $j^{\delta''_{b}+ \kappa''_{b}}$ in the numerator. As $\delta_{b}=1$, the power of $j$ for the instant $b$ is $-\delta_{b}+ 1=0$ instead of $1$. Hence we use $\hat{\delta}_{i,j}$ instead of ${\delta}_{i,j}$ by ignoring the repetitions of first kind for which the type $j\ge 2$ instant appears as non-ignored. 
 
\noindent
Now we neglect $j\le 1000$. Since the product $\prod_{j\le 1000}\prod_{i=1}^{N_{j}}j^{{\delta}_{i,j}}\le \left(1000^{1000}\right)^{N}$. Now for $j\ge 1000$, we have $\sum_{j\ge 1000} j N_{j}\le \frac{1000}{999} \sum_{j\ge 1000} (j-1)N_{j}= \frac{1000}{999} N$. 

Now the rest of the argument is dedicated to bound $\sum_{j\ge 1000}\sum_{i=1}^{N_{j}}{\delta}_{i,j}$. Observe that the $i$ th type $j$ vertex has appeared as an endpoint of $2\delta_{i,j}$ many edges which are traversed odd number of times. As a consequence $$\sum_{j\ge 1000}\sum_{i}\delta_{i,j}\le \sum_{j}\sum_{i} \delta_{i,j}= \# \text{ of edges traversed odd number of times}=\delta.$$

Now to each such edge $e$ we assign two type $j\ge 2$ instants in the following way. We know that all these edges are traversed at least thrice. Hence each edge appears in the upward direction at least twice. The right endpoints of such edges are type $j\ge 2$ instants. Now we only need to consider those edges whose at least one endpoint is type $j\ge 1000$. Now say there are $E_{1}$  edges with exactly one vertex is type $j\ge 1000$ the other vertex is of type $2 \le j < 1000 $ and $E_{2}$ edges with both vertex is type $j\ge 1000$. So we need to calculate an upper bound for $E_{2}+\frac{1}{2}{E}_{1}$. Now for edges corresponding to ${E}_{1}$ instead of two we consider only one instant and we can choose the instant to be at least the second appearance of that instant in case the chosen instant is type $j< 1000$. As a consequence $2E_{2}+ E_{1}\le N+ \sum_{j\ge 1000} N_{j}\le \frac{1000}{999}N.$ So $E_{2}+\frac{1}{2}E_{1}\le \frac{1000}{999}\frac{N}{2}$.

%Firstly we have seen earlier that if $j$ is an endpoint of $2\delta_{i,j}$ many edges the corresponding coefficient is $\delta_{i,j}.$ Hence $\sum_{j}\sum_{i=1}^{N_{j}}\delta'_{i,j}= \# \text{of edges corresponding to such repetitions}=E(say)$. Now to each such edge $e$, we assign three instants. The first one being the instant of first arrival to the corresponding level (say $l$) and second and third one are the instants at level $l+1$ denoting the first and last upward traversal of the edge. We call these set of instants $V_{e}$. In particular each $e$ corresponds a different level $l$ and at each level there are exactly two instants. Hence the total number of such instants are greater than or equal to $2E$. Now we only need to consider those edges whose at least one endpoint is type $j\ge 1000$. Now say there are $E_{1}$  edges with exactly one vertex is type $j\ge 1000$ and $E_{2}$ edges with both vertex is type $j\ge 1000$. So we need to calculate an upper bound for $E_{2}+\frac{1}{2}E_{1}$. Now for edges corresponding to $E_{1}$ instead of two we consider only one instant and we can choose the instant to be at least type $2$. As a consequence $2E_{2}+ E_{1}\le N+ \sum_{j\ge 1000} N_{j}\le \frac{1000}{999}N.$ So $E_{2}+\frac{1}{2}E_{1}\le \frac{1000}{999}\frac{N}{2}$.  

Plugging this into \eqref{dot} we have the following upper bound to \eqref{dot}:
\begin{equation}
\begin{split}
&\sum_{\text{parameters}} 2^{k-m-m'} C^{N} \exp\left\{ \left(\sum_{j\ge 1000}\sum_{i}\delta_{i,j} \right)\log N- \delta  \log k - \frac{N\log N}{2} \right\}\\
&\le \sum_{\text{parameters}} 2^{k-m-m'} C^{N} \left( \frac{1}{\sqrt{n}} \right)^{\delta}\exp\left\{ \left( \sum_{j\ge 1000}\sum_{i}\delta_{i,j} \right) \log N - \frac{\delta}{4}\log N - \frac{N\log N}{2} \right\}\\
& \le \sum_{\text{parameters}} 2^{k-m-m'} C^{N}\left( \frac{1}{\sqrt{n}} \right)^{\delta} \exp\left\{ \frac{3\sum_{j\ge 1000}\sum_{i}\delta_{i,j}}{4} \log N - \frac{N\log N}{2} \right\}\\
& \le \sum_{\text{parameters}} 2^{k-m-m'} C^{N} \left( \frac{1}{\sqrt{n}} \right)^{\delta} \exp\left\{ N\log N \left(  \frac{3\times 500}{4\times 999}- \frac{1}{2}\right) \right\}.
\end{split}
\end{equation}
Since $\frac{1}{2}> \frac{3\times 500}{4\times 999}$ this concludes our proof.

\textbf{Step 2:} Now we come to the \textbf{step 2} of the proof. Here we consider the repetitions in the Dyck paths. The main idea of the proof is similar to the proof we have just done but some calculations are different. Like \textbf{step 1} here also we consider two types of repetitions. The definitions are the same.% In the first type we consider the repetitions of the edges when they use the level  of an endpoint of at least one previous appearance of the edge. In the second type we consider the repetitions when they do not use the level of an endpoint of any previous appearance of the edge. 

At this point we fix the skeleton word and the Dyck paths . Let the Dyck paths be $\mathcal{D}_{1},\ldots, \mathcal{D}_{m+m'+1}$ with their respective lengths $2k_{1},\ldots, 2k_{m+m'+1}$. We also denote $w_{1},\ldots, w_{m+m'+1}$ to be all possible words corresponding to the Dyck paths $\mathcal{D}_{1},\ldots , \mathcal{D}_{m+m'+1}$.
Our fundamental goal is to prove the following bound.
\begin{equation}
\sum_{w_{1}',\ldots, w_{m}'} \E\left[ X_{D(w)} \right]= \frac{1}{2^{
k-(m+m')}}o\left( n^{\sum_{i}k_{i}} \right)
\end{equation} 

Here we assume $\Gamma$ be the total number of type $j\ge 2$ instant where for any given $j$ there are $\Gamma_{j}$ many instants of type $j$. In particular $\Gamma = \sum_{j\ge 2} (j-1)\Gamma_{j}$. 

Here also we shall at first fix a permutation of the non-ignored instants (they are defined in the same manner as the skeleton word) which along with the type $1$ instants specify the first traversal of all the edges. Let $L_{D}$ denote the number of ignored instants in this case and $\Gamma_{j}'$  be the number of non-ignored instants repeated exactly $j$ times. Further we assume that the $i$ th type $j$ instant appears $L_{i,j,D}$ times as an ignored instant. Now the permutations of the non-ignored vertices can be done in 
\begin{equation}\label{eq:permdyck}
\frac{(\Gamma-L_{D})!}{\prod_{j} \Gamma'_{j}! \prod_{j}\prod_{i=1}^{N_{j}}(j-L_{i,j,D})!}
\end{equation}
 ways.    

\noindent 
We have discussed earlier that there might be type $j$ instants such that all the appearances of the corresponding vertex come as ignored instant. So $L_{D}=\sum_{j}\sum_{i}L_{i,j,D} + \hat{L}_{D}$. Let $\hat{L}_{j,D}$ denote the number of type $j$ instant of this kind. Clearly $\hat{L}=\sum_{j} j\hat{L}_{j,D}$. Now we divide $\hat{L}_{j,D}$ into further two cases:
firstly those which appear as ignored instants to type $1$ vertices (call them $\hat{L}_{j,D}'$) and secondly those which appear as ignored instants to type $j \ge 2$ vertices(call them $\hat{L}_{j,D}''$). 
 %firstly those which are adjacent to type $1$ vertices which do not correspond to repetition of first kind of other instants belonging to $L_{i,j,D}$'s for some $i,j$ (call them $\hat{L}_{j,D}'$) and those which are adjacent to type $j \ge 2$ vertices only or type $1$ vertices such that it corresponds to repetition of first kind of other instants belonging to $L_{i,j,D}$'s for some $i,j$(call them $\hat{L}_{j,D}''$). 
Observe that the vertices for $\hat{L}_{j,D}'$'s will correspond to repetitions of the first kind and they come in groups. In particular for a single level there might be multiple instants. Let there be $\xi$ such groups. The location of such groups can be chosen in $\left( \sum_{i} k_{i}\right)^{\xi}$ many ways. Let these levels be $q_{1},\ldots, q_{\xi}$ and for level $q_{i}$, there are $\hat{L}_{j,q_{i},D}$ many instants of type $j$. So there are $\sum_{j}j\hat{L}_{j,q_{i},D}$ many positions to be chosen for a level $q_{i}$. Hence the total number of choices for level $q_{i}$ is bounded by $$\frac{\Delta_{q_{i},D}^{\sum_{j}j\hat{L}_{j,q_{i},D}}}{\prod_{j}\hat{L}_{j,q_{i},D}!(j-1)!^{\hat{L}_{j,q_{i},D}} }.$$ On the other hand the total number of choices for the first appearances for the vertices corresponding to $\hat{L}''_{j,D}$ is bounded by $2^{c\Gamma}$.
 
Now for each level $q$ we assume the corresponding Dyck path come to the level $\Delta_{q,D}$ times. From Lemma \ref{lem:universal} we know that if we consider the uniform distribution of over the feasible paths, then we know that the $\Delta_{q,D}$'s can be dominated by i.i.d. copies of $X$ where $\mathbb{P}\left[ X \ge t \right]\le n^{\vartheta}\frac{1}{2^{t}}$ where $\vartheta$ is a fixed constant not depending on any other parameters. In particular 
\[
\E\left[ \Delta_{q,D}^{\eta} \right]\le \left\{ 
\begin{array}{ll}
c & \text{If $\eta \le 1000$}\\
n^{\vartheta}\eta! & \text{otherwise}.
\end{array}
\right. 
\]
From the choices of the non-ignored vertices which edges are repeated within the $\Delta_{q,D}$ returns to the level $q$ are determined. Now like before suppose the $i$ th type $j$ instant comes immediately after the return to the level $q$ before falling is $\eta_{i,j,q,D}$ times. These positions can be chosen in 
\begin{equation}\label{eq:dyckrepty1}
\frac{\Delta_{q,D}^{\sum_{j,i}\eta_{i,j,q,D}}}{\prod_{j}\prod_{i}\eta_{i,j,q,D}!}
\end{equation}
ways. 

Now like the calculation for the skeleton words we assume that for $i$ th type $j$ instant there are $\eta'_{i,j,i',j',D}$ many ignored instants corresponding to $i'$ th type $j'$ instant. Let $\zeta_{i,j}=\sum_{j',i'} \eta_{i,j,i',j',D}$. The total choices for this is bounded by 
\begin{equation}\label{eq:dyckrepty2}
\frac{j^{\zeta_{i,j}}}{\prod_{j'}\prod_{i'}\eta'_{i,j,i',j',D}!}.
\end{equation}
Now we come to the expectation of a random variable coming in the product $X_{D(w)}$. Suppose a random variable (corresponding to edge $e=\{ e_{1},e_{2} \}$) is repeated $2r_{e}$ times in the product $X_{D(w)}$. Then we divide this repetition in four parts $r_{e,1},r_{e,2},r_{e,3}$ and $r_{e,4}$. Here $2r_{e,1}$ denotes the number of times the edge is repeated as first kind where the instants corresponding to the level of return has label $e_{1}$. Similarly $2r_{e,2}$ denotes the number of times the edge is repeated as first kind where the instants corresponding to the level of return has label $e_{2}$. Finally $2r_{e,3}$ and $2r_{e,4}$ denotes the number of time the edge is repeated as second kind with the ignored instant having label $e_{1}$ and $e_{2}$ respectively. Now the expectation of the random variable is bounded by $r_{e}!$. It is easy to see that $r_{e}! \le C^{r_{e}} \prod_{i=1}^{4} r_{e,i}!$. Now for the repetitions of the first kind several different values of $q$ can correspond to same vertex. So $r_{e,1}$ and $r_{e,2}$ is further partitioned into say $\tau$ many groups. Each corresponds to returns to a single level. Now if we have different levels corresponding to the same value of the label, we divide them into two parts. In the first part we consider where the value of $\sum_{i,j}\eta_{i,j,q,D}$ is less than $1000$ and where the sum is greater than $1000$. We assume $\E_{\text{Unif}}$ denotes the uniform measure on all the feasible paths and $\tau'$ be the number of groups where the sum is less than $1000$ and $\tau''$ be the number of groups where the sum is greater than $1000$.  Here we are dealing with the following sum:
\begin{equation}\label{eq:otherpart}
\begin{split}
&\sum_{\eta_{i,j,q_{f},D}} \E_{\text{Unif}}\left[ \prod_{f=1}^{\tau}\frac{\Delta_{q_{f},D}^{\sum_{i,j}\eta_{i,j,q_{f},D}}}{\prod_{j}\prod_{i} \eta_{i,j,q_{f},D}!} \right]\\
&\le \sum_{\eta_{i,j,q_{f},D}} \prod_{f=1}^{\tau} \frac{1}{\prod_{j}\prod_{i}\eta_{i,j,q_{f},D}!} \prod_{f=1}^{\tau} \left[ c\mathbb{I}_{\sum_{j,i}\eta_{i,j,q_{f},D}\le 1000} + \mathbb{I}_{\sum_{j,i}\eta_{i,j,q_{f},D}> 1000}n^{\vartheta}\left( \sum_{j,i} \eta_{i,j,q_{f},D} \right)! \right]\\
& \le \sum_{\eta_{i,j,q_{f},D}} C^{\tau'} \prod_{f=1}^{\tau''} n^{\vartheta \tau''} \frac{\left(\sum_{j,i}\eta_{i,j,q_{f},D}\right)!}{\prod_{j}\prod_{i} \eta_{i,j,q_{f},D}!}\\
& \le C^{\tau'} \sum_{\sum_{f}\eta_{i,j,q_{f},D}} n^{\frac{\vartheta \sum_{f=1}^{\tau''}\sum_{j,i}\eta_{i,j,q_{f},D}}{1000}} \frac{\left( \sum_{f=1}^{\tau''} \sum_{j,i} \eta_{i,j,q_{f},D}\right)!}{\prod_{j}\prod_{i} \left( \sum_{f=1}^{\tau'} \eta_{i,j,q_{f},D} \right)!} 
\end{split} 
\end{equation}
Now we give the upper bound to $\sum_{w_{1},\ldots, w_{m+m'+1}}\E\left[ X_{D(w)}  \right]$. This is done by the modified bound on the words we discussed so far keeping the repetitions of the edges in mind. In particular, 
\begin{equation}\label{eq:lastlast}
\begin{split}
&\sum_{w_{1},\ldots , w_{m+m'+1}} \E\left[ X_{D(w)} \right]\\
& \le \left( \frac{1}{2} \right)^{k-m-m'}n^{\sum_{i}k_{i}- \sum_{j}(j-1)\Gamma_{j}}\exp\left( - \frac{\left( \sum_{i} k_{i} \right)^2}{2n} \right)\sum_{\Gamma_{2}'}\frac{1}{\Gamma_{2}'!}\left( \frac{\left( \sum_{i} k_{i} \right)^2}{2} \right)^{\Gamma_{2}'}\\
& ~~~~~~~~~ \sum_{\Gamma_{j}'~|~ j \ge 3}  \frac{\left(\Gamma-L_{D}\right)!}{\prod_{j}\Gamma_{j}'! \prod_{j}\prod_{i=1}^{N_{j}} (j- L_{i,j,D})!}\frac{\left( \sum_{i} k_{i} \right)^{ \sum_{j}j\Gamma_{j}'}}{\left(\Gamma- L_{D}\right)!}\times \prod_{i=1}^{\xi}\left( \sum_{
i}k_{i} \right)\frac{\Delta_{q_{i},D}^{\sum_{j}j\hat{L}_{j,q_{i},D}}}{\prod_{j}\hat{L}_{j,q_{i},D}!(j-1)!^{\hat{L}_{j,q_{i},D}} }\times  2^{c\Gamma} \\
& ~~~~~~~ \prod_{j}\prod_{i}\sum_{\zeta_{i,j}} \frac{j^{\zeta_{i,j}}}{\prod_{j'}\prod_{i'} \eta'_{i,j,i',j',D}!} \times \prod_{j'}\prod_{i'} \left( \eta'_{i,j,i',j',D}! \right)\\
& ~~~~~~~~ \prod_{q}\sum_{\sum_{f}\eta_{i,j,q_{f},D}} n^{\frac{\vartheta \sum_{f=1}^{\tau''}\sum_{j,i}\eta_{i,j,q_{f},D}}{1000}} \frac{\left( \sum_{f=1}^{\tau''} \sum_{j,i} \eta_{i,j,q_{f},D}\right)!}{\prod_{j}\prod_{i} \left( \sum_{f=1}^{\tau'} \eta_{i,j,q_{f},D} \right)!} \prod_{j}\prod_{i}\left( \sum_{f=1}^{\tau''} \eta_{i,j,q_{f},D} \right)!\\
&\le \left( \frac{1}{2} \right)^{k-m-m'}n^{\sum_{i}k_{i}}2^{c\Gamma}\exp\left( - \frac{\left( \sum_{i} k_{i} \right)^2}{2n} \right)\sum_{\Gamma_{2}'}\frac{1}{\Gamma_{2}'!}\left( \frac{\left(\sum_{i}k_{i}\right)^2}{n} \right)^{\Gamma_{2}'}\sum_{\Gamma_{j}' ~|~ j\ge 3} \frac{1}{\Gamma_{j}'!}\left(\frac{\left( \sum_{i} k_{i} \right)^{j}}{(j-1)! n^{j-1}}\right)^{\Gamma_{j}'}\\
&~~~~~~~~~~ \times \prod_{i=1}^{\xi}\left( \mathbb{I}_{\sum_{j}j\hat{L}_{j,q_{i},D}\le 1000} \frac{\left(\sum_{i}k_{i}\right)\left(c\log n\right)^{\sum_{j}j\hat{L}_{j,q_{i},D}}}{n^{\sum_{j}j\hat{L}_{j,q_{i},D}}} \right.\\
&~~~~~~~~~~~~~~~\left. + \mathbb{I}_{\sum_{j}j\hat{L}_{j,q_{i},D}> 1000} 2^{1000 \sum_{j}j\hat{L}_{j,q_{i},D}}\left( \sum_{i} k_{i} \right)\prod_{j=1}^{1000}\left(\frac{(\sum_{i}k_{i})^{j-1}}{n^{j-1}}\right)^{\hat{L}_{j,q_{i},D}}\prod_{j>1000}\left( \frac{\left( \sum_{i} k_{i} \right)^{j}}{n^{j-1}} \right)^{\hat{L}_{j,q_{i},D}}\right)\\
&~~~~~~~~~~\times \sum_{L_{i,j,D}}\left(\frac{n^{\varsigma} \left( \sum_{i} k_{i} \right)}{n}\right)^{\sum_{j}\sum_{i} L_{i,j,D}} \times \sum_{\hat{L}_{j,D}''}\left( \frac{n^{\varsigma}\sum_{i}k_{i}}{n} \right)^{\hat{L}_{j,D}''}\\
&= o\left(n^{\sum_{i}k_{i}}\left(\frac{1}{2}\right)^{k-m-m'}\right)
\end{split}
\end{equation} 
As we can choose $\varsigma>0$ any positive number, choosing $\varsigma$ small enough we have the result.
Here the factors $\left( \frac{1}{2}^{k-m-m'} \right)$ and the factorials in the numerator comes from $\E[X_{D(w)}].$ %Also note that we have skipped the details analysis of \eqref{eq:lastlast} since this is very similar to the previous parts.
\end{proof}

Compiling Propositions \ref{prop:ge3}, \ref{lem:countingstrategy} and \ref{prop:2times} we arrive at the proof of Theorem \ref{thm:traceconvergence}. 
%\begin{theorem}\label{thm:traceconvergence}
%Consider the Wigner matrix $W$ satisfying Assumption \ref{ass:wig}. Then for any fixed $t \in (0,\infty)$ taking $k= \left[tn^{\frac{2}{3}}\right]$, we have the following results 
%\begin{enumerate}
%\item $\E\Tr\left[ W^{2k} \right]=O(1)$ and $\E\Tr\left[W^{2k+1}\right]=o(1)$.
%\item If the limit of $\lim_{n \to \infty} \E\Tr\left[ W^{2k} \right] $ for some $t\in (0,\infty)$ exists, then the limit only depends on the first and second moment of entries. 
%\item As the limit exists for Gaussian entries, the limit exists and is universal for any Wigner matrix satisfying Assumption \ref{ass:wig}. 
%\end{enumerate}
%\end{theorem}
\section{Proof of Theorem \ref{thm:jointmom}}
From discussion at the beginning of this paper, we have seen that in order to prove the Tracy Widom distribution at the edge, in addition to bounding the expectation of high value of traces we also need to bound the joint moments. In this section we deal with this problem. 

\noindent 
Before going into the proof of Theorem \ref{thm:jointmom}, we state another important algorithm. This algorithm takes two closed words $w_{1}$ and $w_{2}$ of lengths $k_{1}+1$ and $k_{2}+1$ as input such that the words have at least one edge in common and gives a closed word $w_{3}$ of length $k_{1}+k_{2}+1$ as an output which has the same edge set as the union of the edges of $w_{1}$ and $w_{2}$. Variants of this algorithm has appeared in \citet{Ban16}, \citet{BanB16} and \citet{Banerjee2017}.    
\begin{algorithm}\label{alg:embed}
We start with two words $w_{1}$ and $w_{2}$ such that $w_{1}$ and $w_{2}$ shares an edge. Let $\{ \alpha, \beta \}$ be the first edge in $w_{2}$ which is repeated in $w_{1}$. We consider the first appearance of $\{\alpha, \beta \}$ in $w_{2}$. Without loss of generality we assume that the  first appearance of the edge $\{ \alpha , \beta \}$ appears in the word $w_{2}$ in the order $(\alpha,\beta)$. We now consider any appearance (for concreteness say the first) of the edge $\{ \alpha, \beta \}$ in the word $w_{1}$. This appearance $\{ \alpha ,\beta \}$ can be traversed in $w_{1}$ in the order $(\alpha,\beta)$ or $(\beta, \alpha)$.
Considering these we have the word $w_{2}$ looks like 
\begin{equation}\label{eq:w1before}
w_{2}= (\alpha_{0},\alpha_{1},\ldots,\alpha_{p_{1}}, \alpha, \beta,\ldots,\alpha_{k_{1}-1},\alpha_{0})
\end{equation} 
and the word $w_{1}$ looks like 
\begin{equation}\label{eq:w21}
w_{1}= \left( \beta_{0}, \beta_{1},\ldots,\beta_{q_{1}},\alpha, \beta, \ldots, \beta_{k_{2}-1},\beta_{0} \right)
\end{equation}
or
\begin{equation}\label{eq:w22} 
 w_{1}=(\beta_{0},\beta_{1},\ldots,\beta_{q_{1}},\beta,\alpha,\ldots,\beta_{k_{2}-1},\beta_{0}).
 \end{equation}
Now we output the word $w_{3}$ as follows:
\begin{enumerate}
\item Suppose $w_{1}$ is of the form \eqref{eq:w21}, then 
\begin{equation}\label{eq:w31}
w_{3}= \left(  \alpha_{0},\alpha_{1},\ldots,\alpha_{p_{1}},\alpha,\beta,\beta_{q_{1}+3},\ldots, \beta_{k_{2}-1}, \beta_{0},\ldots,\beta_{q_{1}},\alpha,\beta,\alpha_{p_{1}+3},\ldots,\alpha_{k_{1}-1},\alpha_{0} \right).
\end{equation}
\item On the other hand when $w_{1}$ is of the form \eqref{eq:w22}, 
\begin{equation}\label{eq:w32}
w_{3}= \left( \alpha_{0},\alpha_{1},\ldots,\alpha_{p_{1}},\alpha,\beta,\beta_{q_{1}},\ldots,\beta_{0},\beta_{k_{2}-1},\ldots,\beta_{q_{1}+3},\alpha,\beta,\alpha_{p_{1}+3},\ldots,\alpha_{k_{1}-1},\alpha_{0}\right). 
\end{equation}
\end{enumerate}
\end{algorithm}
With Algorithm \ref{alg:embed} in hand we are now ready to prove Theorem \ref{thm:jointmom}.
\begin{proof}[Proof of Theorem \ref{thm:jointmom}]
Let for any word $w= \left(\alpha_{0},\ldots, \alpha_{k-1},\alpha_{0}\right)$, $X_{w}$ denote the random variable 
\begin{equation}
X_{w}:= \prod_{j=0}^{k-1} x_{\alpha_{j},\alpha_{j+1}}.
\end{equation}
 Firstly observe that 
\begin{equation}
\begin{split}
&\E\left[ \prod_{i=1}^{i} \left[\Tr\left[W^{k_{i}}\right] - \E\left[  \Tr\left[W^{k_{i}}\right]\right] \right] \right]\\
& = \left(\frac{1}{n}\right)^{\frac{\sum_{i=1}^{l}k_{i}}{2}}\sum_{w_{1}\ldots w_{l}} \E\prod_{i=1}^{l}\left( X_{w_{i}} - \E\left[ X_{w_{i}} \right] \right)
\end{split}
\end{equation}
where $w_{i}$ is of length $k_{i}+1$ for $1\le i \le l$ respectively. 
Now given the words $w_{1},\ldots, w_{l}$ we introduce a partition $\varsigma$ of $\{ 1,\ldots , l \}$ in the following way: We put $i$ and $j$ in the same block if $G_{w_{i}}$ and $G_{w_{j}}$ share an edge. Let $d(\varsigma)$ be the number of blocks in $\varsigma$. Given the words $w_{1},\ldots, w_{l}$ we sequentially embed them in $d(\varsigma)$ closed words with disjoint edges by applying Algorithm \ref{alg:embed} sequentially. Now given $d(\varsigma)$ closed words with disjoint edge set our goal is to find the number of possible choices of $w_{1},\ldots,w_{l}$ which give rise to the given words. We call these words $\mathfrak{w}_{1},\ldots, \mathfrak{w}_{d(\varsigma)}$. To this end we shall also fix the partition $\varsigma$ and the order in which Algorithm \ref{alg:embed} is used. Let for a block $B_{i}= {j_{1,i},\ldots, j_{| B_{i}|,i}}$ of the partition $\varsigma$, the order is given by $\left(j_{\tau_{i}(1),i},\ldots, j_{\tau_{i}(|B_{i}|),i}\right)$. Here $\tau_{i}$ is a permutation of $\{ 1,\ldots, |B_{i}|\}$.

\noindent 
Now we start with the word $\mathfrak{w}_{1}$. From $\mathfrak{w}_{1}$ we at first extract $w_{\tau_{1}(|B_{1}|),1}$ and the word $\mathfrak{w}'_{1,1}$ such that we apply Algorithm \ref{alg:embed} to  $\mathfrak{w}'_{1,1}$ and $w_{\tau_{1}(|B_{1}|),1}$. We proceed in this way $|B_{1}|$ times where at each step we replace $\mathfrak{w}_{1}$ by the first word we obtain after reverting Algorithm \ref{alg:embed}. Then we work with $B_{2},\ldots , B_{d(\varsigma)}$ in the similar fashion.

\noindent 
Now we enter into one of the main parts of the proof. Here we discuss the procedure to invert Algorithm \ref{alg:embed} to reconstruct the words. First of all given $\mathfrak{w}_{1}$, in order to reconstruct $w_{\tau_{1}(|B_{1}|),1}$ and $\mathfrak{w}'_{1,1}$ we need to choose three points on the word $\mathfrak{w}_{1}$. The first and third point have the same label and they determine where the word $w_{\tau_{1}(|B_{1}|),1}$ is cut to implement the algorithm. The second point determines the endpoint of the word $\mathfrak{w}'_{1,1}$. Also observe that once the first point is chosen, the third point comes exactly after $l(\mathfrak{w}'_{1,1})-1$ steps. However the second point can be anywhere in between the first and the third point. 

\noindent 
In the next part we prove that the first point can't be arbitrary. Here we consider the following cases. Firstly it can happen that the edge $\{ \alpha, \beta \}$ appears exactly once in each of the words $\mathfrak{w}'_{1,1}$ and $w_{\tau_{1}(|B_{1}|),1}$. In this case we shall prove that the vertex $\alpha$ is a type $j\ge 2$ instant. At this point recall from Algorithm \ref{alg:embed} that both appearances of the edge $\{ \alpha, \beta\}$ happen in the order $(\alpha,\beta)$ in the word $\mathfrak{w}_{1}$. Also this edge is the first edge in $w_{\tau_{1}(|B_{1}|),1}$ which is repeated in $\mathfrak{w}'_{1,1}$. So the level of the first point can not be reached after closing an edge from right to left within $l(\mathfrak{w}'_{1,1})-1$ steps after first point. Also since the edge $\{\alpha, \beta \}$ appears exactly once in both $\mathfrak{w}'_{1,1}$ and $w_{\tau_{1}(|B_{1}|),1}$, the edge is not closed within $l(\mathfrak{w}'_{1,1})-2$ steps after first point. So the only way to reach the vertex $\alpha$ is, another level has the label $\alpha$. Hence $\alpha$ is a type $j \ge 2$ vertex. 

\noindent 
In all the other cases since the edge $\{ \alpha, \beta \}$ is traversed at least twice in the word $\mathfrak{w}'_{1,1}$, at least one of the vertices $\alpha,\beta$ is a type $j \ge 2$ vertex.

\noindent
Now we consider two cases depending on the position of the third point. 
\begin{enumerate}[(i)]
\item The third point is at a different level than the first point.
\item The third point is at the same level as the first point. 
\end{enumerate}
We at first consider case (i).
Since $\{\alpha,\beta\}$ is the first edge in $w_{\tau_{1}(|B_{1}|),1}$ which is repeated in $\mathfrak{w}'_{1,1}$, if we look at $l(\mathfrak{w}'_{1,1})-1$ steps after first point, the walk never falls down of the level of the first point. As a consequence, within the $l(\mathfrak{w}'_{1,1})-1$ steps after first point there is always an open edge incident to the vertex $\alpha$. This implies if the third point is at a different level than the first point, then the level of the third point at a level of a type $j \ge 2$ open instant. Hence in order to choose the first and the third point, one needs to look at the skeleton word, choose two appearances of a type $j \ge 2$ instant and look at the returns to these chosen levels. Now we spend some time on the total number of words in this case. Like usual we at first fix the skeleton word and fix two instants where same vertex is repeated as type $j\ge 2$ instants. Let $m_{2}$ be the length of the skeleton word between these two points while $m_{1}+m_{2}$ be the total length of the skeleton word. Then the total number of Dyck paths in between the chosen points is given by $\frac{m_{2}+1}{l(\mathfrak{w}'_{1,1})}\binom{l(\mathfrak{w}'_{1,1})}{\frac{l(\mathfrak{w}'_{1,1})+m_{2}+1}{2}}$. On the other hand the number of Dyck paths in the other part can be chosen in $\frac{m_{1}+1}{l(w_{\tau_{1}(|B_{1}|),1})}\binom{l(w_{\tau_{1}(|B_{1}|),1})}{\frac{l(w_{\tau_{1}(|B_{1}|),1})+m_{1}+1}{2}}$ ways.  	

\noindent 
Now we look at case (ii). We further reduce it into the following cases: 
\begin{enumerate}[(a)]
\item The whole word $\mathfrak{w}'_{1,1}$ is inside a Dyck path. 

In this case we choose a point on the word $w_{\tau_{1}(|B_{1}|),1}$ and from that point we choose a Dyck path of length $l(\mathfrak{w}'_{1,1})-1$. Next just after last point of the Dyck path we create an upward edge with endpoint $\beta$. Considering the first edge in the Dyck path to be $(\alpha, \beta )$ with starting point $\alpha$, the upward edge after the last point of the Dyck path is a type $j\ge 2$ instant. The words in this case is chosen in the following way. We pick a point on the word $w_{\tau_{1}(|B_{1}|),1}$ and place a Dyck path of length $l(\mathfrak{w}'_{1,1})-1$ immediately after that point. So the the choices for the word $\mathfrak{w}'_{1,1}$ is simply $C_{\frac{l(\mathfrak{w}'_{1,1})-1}{2}}$.

\item The word $\mathfrak{w}'_{1,1}$ spans beyond a Dyck path. Observe that the first point in this case has to be in the same level as a point on the skeleton word. Here we divide it into further two cases depending on the edge $\{\alpha,\beta\}$ after the third point belonging or not belonging to a Dyck path.   

\noindent 
\textbf{Case I:} Here we assume that the edge $\{ \alpha, \beta \}$ after the third point belongs to the skeleton word. Hence the edge $\{ \alpha,\beta \}$ after the first point belongs to skeleton word.  Observe that here the instant immediately after the first point and the instant immediately after the third point are type $j \ge 2$. Hence here we arrive at a situation similar to case (i). The only difference is instead of choosing the type $j$ instants as the first and the third point we choose the points immediately before them. The calculation of the number of words is also the same as case (i). 

\noindent 
\textbf{Case II:} Here we assume that the edge $\{ \alpha, \beta \}$ after the third point belongs to a Dyck path. Since the first and the third point is in the same level of a point in the skeleton word, the calculation of the number of words is also same as case (i). However here we have an additional constraint that the instant immediately after the third point is a type $j \ge 2$ instant and the third point is at the same level as a point on the skeleton word. So it is equivalent to choose a point on the skeleton word and construct the type $j \ge 2$ instant from that level. This will reduce the count. One might look at the detailed explanation below. 
\end{enumerate}
We now prove Theorem \ref{thm:jointmom}. We only prove part $1$ and $2$ of it. Further for convenience we assume that $l=2$. The case for general $l$ can be proved by repeated use of the arguments given here. The fundamental idea of the proof is as follows: We shall prove that only case (i) and \textbf{Case I} of part (b) of case (ii) give higher contribution. All the other cases give a negligible contribution. Further in these cases we know that with high probability $m_{1}, m_{2}$ are of $O(n^{\frac{1}{3}})$ while $l(\mathfrak{w}'_{1,1})$ is of $O(n^{\frac{2}{3}})$. This makes the value 
$\frac{m_{2}+1}{l(\mathfrak{w}'_{1,1})}\binom{l(\mathfrak{w}'_{1,1})}{\frac{l(\mathfrak{w}'_{1,1})+m_{2}+1}{2}}= 2^{l(\mathfrak{w}'_{1,1})}O\left(\frac{1}{l(\mathfrak{w}'_{1,1})}\right)$. The factor $\left(\frac{1}{l(\mathfrak{w}'_{1,1})}\right)$ cancels out with the possible choices of the second point which is also of $O(l(\mathfrak{w}'_{1,1}))$. As a result we get the total contribution is of $O(1)$. In the next part we formalize these arguments.

\noindent 
We now consider each case discussed in the proof separately.\\
(i): This is the most important case and is responsible for the main contributions. We shall use the results in Propositions \ref{prop:ge3} and \ref{prop:2times}. First of all observe that 
\begin{equation}
\begin{split}
&\sum_{w_{1},w_{2}}\left( \frac{1}{n} \right)^{\frac{k_{1}+k_{2}}{2}}\E\left[ \left(X_{w_{1}}-\E[X_{w_{1}}]\right)\left(X_{w_{2}}-\E[X_{w_{2}}]\right) \right]\\
& \sum_{w_{1},w_{2}} \left(\frac{1}{n}\right)^{\frac{k_{1}+k_{2}}{2}} \left(\E\left[ X_{w_{1}} X_{w_{2}} \right]- \E\left[ X_{w_{1}} \right]\E\left[ X_{w_{2}} \right]\right).
\end{split}
\end{equation}
Now if $w_{1}$ and $w_{2}$ do not share an edge then $\E[X_{w_{1}}X_{w_{2}}]=\E[X_{w_{1}}]\E[X_{w_{2}}]$. So we shall ignore such words. Among the remaining words, we have dealt with the term $\sum_{w_{1}w_{2}} \left( \frac{1}{n} \right)^{\frac{k_{1}+k_{2}}{2}} \E[X_{w_{1}}]\E[X_{w_{2}}]$ in Propositions \ref{prop:ge3} and \ref{prop:2times}. So we only need to consider the first term. Also following the notation used in the proof so far, we shall denote $w_{1}$ by $\mathfrak{w}'_{1,1}$ and $w_{2}$ by $w_{\tau_{1}(|B_{1}|),1}$. We apply Algorithm \ref{alg:embed} to these words to get $\mathfrak{w}_{1}$. Let the parameters be defined analogously for the word $\mathfrak{w}_{1}$. Then %Now assume that $N_{1},N_{2},L_{1}, L_{2}$ be the corresponding parameters for $\mathfrak{w}'_{1,1}$ $w_{\tau_{1}(|B_{1}|),1}$ respectively. Then 
\begin{equation}\label{eq:joinmomcasei}
\begin{split}
&\sum_{\mathfrak{w}'_{1,1},w_{\tau_{1}(|B_{1}|),1}}\left( \frac{1}{n} \right)^{\frac{l(\mathfrak{w}'_{1,1})+ l(w_{\tau_{1}(|B_{1}|),1})-2}{2}} \E\left[ X_{\mathfrak{w}'_{1,1}}X_{w_{\tau_{1}(|B_{1}|),1}} \right]  \\     
& =\sum_{\mathfrak{w}_{1}} \sum_{\mathfrak{w}'_{1,1},w_{\tau_{1}(|B_{1}|),1}~|~ f(\mathfrak{w}'_{1,1},w_{\tau_{1}(|B_{1}|),1})= \mathfrak{w}_{1}} \left( \frac{1}{n} \right)^{\frac{l(\mathfrak{w}'_{1,1})+ l(w_{\tau_{1}(|B_{1}|),1})-2}{2}} \E\left[ X_{\mathfrak{w}_{1}} \right]\\
& = \sum_{\mathfrak{w}_{1}} \left( \frac{1}{n} \right)^{\frac{l(\mathfrak{w}_{1})-1}{2}}\left| f^{-1}(\mathfrak{w}_{1}) \right|  \E[X_{\mathfrak{w}_{1}}]\\
& \le \sum_{\mathfrak{w}_{1}} \left( \frac{1}{n} \right)^{\frac{l(\mathfrak{w}_{1})-1}{2}}  N^{2}  l(\mathfrak{w}'_{1,1}) \E[X_{\mathfrak{w}_{1}}] \\
& \le \left( \frac{1}{n} \right)^{\frac{l(\mathfrak{w}_{1})-1}{2}}\left( \frac{1}{2} \right)^{k-m-m'}n^{\sum_{i}k_{i}- \sum_{j}(j-1)\Gamma_{j}}\exp\left( - \frac{\left( \sum_{i} k_{i} \right)^2}{2n} \right)\sum_{\Gamma_{2}'}\frac{1}{\Gamma_{2}'!}\left( \frac{\left( \sum_{i} k_{i} \right)^2}{2} \right)^{\Gamma_{2}'}\\
& ~~~~~~~~~ \sum_{\Gamma_{j}'~|~ j \ge 3}  \frac{\left(\Gamma-L_{D}\right)!}{\prod_{j}\Gamma_{j}'! \prod_{j}\prod_{i=1}^{N_{j}} (j- L_{i,j,D})!}\frac{\left( \sum_{i} k_{i} \right)^{(\Gamma - L_{D})+ \sum_{j}\Gamma_{j}'}}{\left(\Gamma- L_{D}\right)!}\times \prod_{i=1}^{\xi}\left( \sum_{
i}k_{i} \right)\frac{\Delta_{q_{i},D}^{\sum_{j}j\hat{L}_{j,q_{i},D}}}{\prod_{j}\hat{L}_{j,q_{i},D}!(j-1)!^{\hat{L}_{j,q_{i},D}} }\times  2^{c\Gamma}\\
& ~~~~~~~ \prod_{j}\prod_{i}\sum_{\zeta_{i,j}} \frac{j^{\zeta_{i,j}}}{\prod_{j'}\prod_{i'} \eta'_{i,j,i',j',D}!} \times \prod_{j'}\prod_{i'} \left( \eta'_{i,j,i',j',D}! \right)\\
& ~~~~~~~~ \prod_{q}\sum_{\sum_{f}\eta_{i,j,q_{f},D}} n^{\frac{\vartheta \sum_{f=1}^{m''}\sum_{j,i}\eta_{i,j,q_{f},D}}{1000}} \frac{\left( \sum_{f=1}^{m''} \sum_{j,i} \eta_{i,j,q_{f},D}\right)!}{\prod_{j}\prod_{i} \left( \sum_{f=1}^{m'} \eta_{i,j,q_{f},D} \right)!} \prod_{j}\prod_{i}\left( \sum_{f=1}^{m''} \eta_{i,j,q_{f},D} \right)! \times \\
& ~~~ \frac{1}{2^{m+m'}}\sum_{\text{parameters}}n^{m- N +1}C^{N}\\
&~~\prod_{e} \frac{1}{\left[\frac{r_{e}}{2}\right]!} \frac{(N-L)!\prod_{j}\prod_{i=1}^{N_{j}}j^{j-\delta_{i,j}}}{\prod_{j}N'_{j}! \prod_{j}\prod_{i=1}^{N_{j}}(j-L_{i,j})!}\prod_{q} \frac{\Delta_{q}^{\sum_{j,i}\eta_{i,j,q} }}{\prod_{j}\prod_{i}\eta_{i,j,q}!}\prod_{j}\prod_{i=1}^{N_{j}} \frac{j^{\zeta_{i,j}}}{\prod_{j'}\prod_{i'} \eta'_{i,j,i',j'}!} \times \\
&~~ \frac{(2m_{1}+2m_{2})^{2N-2L-2+ \sum_{j}N_{j}'}}{(2N-2L-2)!}\\
& ~~~N^{2}l(\mathfrak{w}'_{1,1})\frac{m_{2}+1}{l(\mathfrak{w}'_{1,1})}\binom{l(\mathfrak{w}'_{1,1})}{\frac{l(\mathfrak{w}'_{1,1})+m_{2}+1}{2}}\frac{m_{1}+1}{l(w_{\tau_{1}(|B_{1}|),1})}\binom{l(w_{\tau_{1}(|B_{1}|),1})}{\frac{l(w_{\tau_{1}(|B_{1}|),1})+m_{1}+1}{2}} \prod_{e} \left[ \frac{r_{e}}{2} \right]!
\end{split}
\end{equation}
\eqref{eq:joinmomcasei} can be handled exactly in the same way we handled Proposition \ref{prop:2times} and \ref{prop:ge3} to get that \eqref{eq:joinmomcasei} is of $O(1)$. Further if the word $\mathfrak{w}$ has at least one edge repeated at least thrice, then the sum goes to $0$. 

It can be proved that the sums in all the other cases go to $0$.
Since the equations are somewhat long and are not very informative, for the next cases we shall only point out due to which factors the corresponding sum goes to $0$.  

\noindent 
(ii):
\\
Case (a): Here the word $\mathfrak{w}'_{1,1}$ is given by a Dyck path of length $l(\mathfrak{w}'_{1,1})$. This Dyck path can be placed in anywhere of the word. So in order to form $\mathfrak{w}_{1}$, we at first form a word of length $l(w_{\tau_{1}(|B_{1}|),1})$. Now we choose a left endpoint of an upward edge in the word (this choice is necessary since the edge $\{ \alpha,\beta \}$ from the third point is an upward edge and the given upward edge serves for that). Now the Dyck path can be chosen in $C_{\frac{l(\mathfrak{w}'_{1,1})-1}{2}}= 2^{l(\mathfrak{w}'_{1,1})-1}O\left( \frac{1}{n} \right)$. Now the position of the first point has $O\left(n^{\frac{2}{3}}\right)$ many choices. However we loose a factor of $n$, since the instant immediately after the third point is a type $j \ge 2$ instant. On the other hand the second point has $O\left( n^{\frac{2}{3}} \right)$ many choices. Compiling these we have $$\sum_{\mathfrak{w}_{1}} \left( \frac{1}{n} \right)^{\frac{l(\mathfrak{w}_{1})-1}{2}}\left| f^{-1}(\mathfrak{w}_{1}) \right|  \E[X_{\mathfrak{w}_{1}}]=O\left( \frac{n^{\frac{4}{3}}}{n^{2}} \right)=O\left( \frac{1}{n^{\frac{2}{3}}} \right)\to 0.$$
\\
Case (b):\\
\textbf{Case I}: This case is almost identical to case (i). However the only difference is, since the edge $\{ \alpha, \beta \}$ is repeated at least thrice, we have 
\[
\sum_{\mathfrak{w}_{1}} \left( \frac{1}{n} \right)^{\frac{l(\mathfrak{w}_{1})-1}{2}}\left| f^{-1}(\mathfrak{w}_{1}) \right|  \E[X_{\mathfrak{w}_{1}}] \to 0.
\]
We omit the details.\\
\textbf{Case II}: In this case the word count will be similar to case (i). However the edge $\{ \alpha,\beta \}$ immediately after the third point belongs to a Dyck path. Since the first point is in the same level of a point in the skeleton word, the number of choices of the first point is $O(m_{1})=O(n^{\frac{1}{3}})$. On the other hand since the instant immediately after the third point is a type $j \ge 2$ instant, we loose a factor of $n$. Finally the choice of the second point is of $O\left(n^{\frac{2}{3}}\right)$ and the factor $\frac{m_{2}+1}{l(\mathfrak{w}'_{1,1})}\binom{l(\mathfrak{w}'_{1,1})}{\frac{l(\mathfrak{w}'_{1,1})+m_{2}+1}{2}}= 2^{l(\mathfrak{w}'_{1,1})}O\left(\frac{1}{l(\mathfrak{w}'_{1,1})}\right)$. Compiling these we get
\begin{equation*}
\sum_{\mathfrak{w}_{1}} \left( \frac{1}{n} \right)^{\frac{l(\mathfrak{w}_{1})-1}{2}}\left| f^{-1}(\mathfrak{w}_{1}) \right|  \E[X_{\mathfrak{w}_{1}}] = O\left( \frac{n^{\frac{1}{3}+ \frac{2}{3}}}{n^{\frac{2}{3}+1}}\right)=O\left( \frac{1}{n^{\frac{2}{3}}} \right) \to 0.
\end{equation*}
%So given the word $\mathfrak{w}_{1}$, in order to choose the first and the third 
This concludes the proof.
\end{proof}
\noindent 
\textbf{Acknowledgment} This work is supported by an Inspire faculty fellowship. Moreover I thank Soumendu Sundar Mukherjee for many detailed discussions and Prof. Arup Bose, Prof. Riddhipratim Basu and Prof.  Rajat Subhra Hazra for useful discussions during the initial version of the paper. For the current version I am indebted to Prof. Krishna Maddaly for valuable discussions and suggestions. I also thank Prof. Arijit Chakrabarty for encouragement and supports. 
 \bibliography{PAR_SPI}
 \end{document}